\documentclass{article}

\usepackage{amsmath,amsthm,amssymb}
\usepackage{BussProofs}
\usepackage{epsfig}
\usepackage{psfrag}

\newtheorem{theorem}{Theorem}[section]
\newtheorem{lemma}[theorem]{Lemma}
\newtheorem{proposition}[theorem]{Proposition}

\newcommand{\na}[1]{\mathit{#1}}    

\newcommand{\limplies}{\rightarrow}

\newcommand{\ex}[1]{\exists #1 \;} 
\newcommand{\fa}[1]{\forall #1 \;} 

\newcommand{\dash}{\mathalpha{\mbox{-}}} 
\newcommand{\ph}{\varphi}
\newcommand{\la}{\langle}
\newcommand{\ra}{\rangle}

\newcommand{\AXM}[1]{\AXC{$#1$}}

\newcommand{\BIM}[1]{\BIC{$#1$}}
\newcommand{\TIM}[1]{\TIC{$#1$}}
\newcommand{\AXN}[1]{\AX$#1$}
\newcommand{\BIN}[1]{\BI$#1$}
\newcommand{\UIN}[1]{\UI$#1$}
\newcommand{\TIN}[1]{\TI$#1$}

\EnableBpAbbreviations
\newcommand{\fCenter}{\Rightarrow}

\newcommand{\on}{\mathrm{on}}
\newcommand{\sameside}{\mathrm{same\dash side}}
\newcommand{\diffside}{\mathrm{diff\dash side}}
\newcommand{\mybetween}{\mathrm{between}}
\newcommand{\mycenter}{\mathrm{center}}
\newcommand{\inside}{\mathrm{inside}}
\newcommand{\outside}{\mathrm{outside}}
\newcommand{\segment}{\mathrm{segment}}
\newcommand{\seg}[1]{\overline{#1}}
\newcommand{\myangle}{\mathrm{angle}}
\newcommand{\myarea}{\mathrm{area}}
\newcommand{\area}[1]{\triangle #1}
\newcommand{\intersects}{\mathrm{intersects}}
\newcommand{\rightangle}{\mathrm{right\mbox{-}angle}}

\newcommand{\RR}{\mathbb{R}}

\newcommand{\showdiagram}[1]{#1}

\newcommand{\Esymm}{\ensuremath{\mathrm{E1}}}
\newcommand{\Etrans}{\ensuremath{\mathrm{E2}}}
\newcommand{\Eid}{\ensuremath{\mathrm{E3}}}
\newcommand{\Seg}{\ensuremath{\mathrm{SC}}}
\newcommand{\Pasch}{\ensuremath{\mathrm{P}}}
\newcommand{\Five}{\ensuremath{\mathrm{5S}}}
\newcommand{\Betw}{\ensuremath{\mathrm{B}}}
\newcommand{\LD}{\ensuremath{\mathrm{2L}}}
\newcommand{\UD}{\ensuremath{\mathrm{2U}}}
\newcommand{\Euc}{\ensuremath{\mathrm{PP}}}
\newcommand{\NNeg}{\ensuremath{\mathrm{Neg}}}
\newcommand{\Int}{\ensuremath{\mathrm{Int}}}

\newcommand{\OOn}{\ensuremath{\on}}
\newcommand{\SSide}{\ensuremath{\sameside}}
\newcommand{\Between}{\ensuremath{\mybetween}}
\newcommand{\OnCirc}{\ensuremath{\on}}
\newcommand{\Inside}{\ensuremath{\inside}}

\newcommand{\Intersect}{\ensuremath{\intersects}}

\title{A formal system for Euclid's \emph{Elements}}

\author{Jeremy Avigad, Edward Dean, and John Mumma\footnote{We are
    especially indebted to Ken Manders, whose work and encouragement
    set us on our way. We would also like to thank Jesse Alama, Alan
    Baker, Michael Beeson, Karine Chemla, Annalisa Coliva, Mic
    Detlefsen, Harvey Friedman, Mark Goodwin, Jeremy Gray, Jeremy
    Heis, Anthony Jones, Danielle Macbeth, Penelope Maddy, Paolo
    Mancosu, Henry Mendell, Marco Panza, Jan von Plato, Vaughan Pratt,
    Dana Scott, Lisa Shabel, Stewart Shapiro, Wilfried Sieg, Neil
    Tennant, Freek Wiedijk, and an anonymous referee for helpful
    comments, criticisms, corrections, and suggestions; and we
    apologize to any others who also deserve our thanks but have been
    inadvertently overlooked. Avigad's work has been partially
    supported by NSF grant DMS-0700174 and a grant from the John
    Templeton Foundation.}}

\begin{document}

\maketitle

\begin{abstract}
  We present a formal system, $\na{E}$, which provides a faithful
  model of the proofs in Euclid's \emph{Elements}, including the use
  of diagrammatic reasoning.
\end{abstract}

\tableofcontents

\section{Introduction}
\label{introduction:section}

For more than two millennia, Euclid's \emph{Elements} was viewed by
mathematicians and philosophers alike as a paradigm of rigorous
argumentation. But the work lost some of its lofty status in the
nineteenth century, amidst concerns related to the use of diagrams in
its proofs. Recognizing the correctness of Euclid's inferences was
thought to require an ``intuitive'' use of these diagrams, whereas, in
a proper mathematical argument, every assumption should be spelled out
explicitly. Moreover, there is the question as to how an argument that
relies on a single diagram can serve to justify a general mathematical
claim: any triangle one draws will, for example, be either acute,
right, or obtuse, leaving the same intuitive faculty burdened with the
task of ensuring that the argument is equally valid for \emph{all}
triangles.\footnote{The question was raised by early modern
  philosophers from Berkeley \cite[Section 16]{Berk} to Kant
  \cite[A716/B744]{Kant}. See
  \cite{friedman:85,goodwin:03,mumma-miller-review,shabel:03a,shabel:03,shabel:04}
  for discussions of the philosophical concerns.}  Such a reliance on
intuition was therefore felt to fall short of delivering mathematical
certainty.

Without denying the importance of the \emph{Elements}, by the end of
the nineteenth century the common attitude among mathematicians and
philosophers was that the appropriate \emph{logical} analysis of
geometric inference should be cast in terms of axioms and rules of
inference. This view was neatly summed up by Leibniz more than two
centuries earlier:
\begin{quote}
\ldots it is not the figures which furnish the proof with geometers,
though the style of the exposition may make you think so. The force of
the demonstration is independent of the figure drawn, which is drawn
only to facilitate the knowledge of our meaning, and to fix the
attention; it is the universal propositions, i.e.~the definitions,
axioms, and theorems already demonstrated, which make the reasoning,
and which would sustain it though the figure were not
there. \cite[p.~403]{Leibniz}
\end{quote}
This attitude gave rise to informal axiomatizations by Pasch
\cite{pasch:82}, Peano \cite{peano:89}, and Hilbert \cite{hilbert:99}
in the late nineteenth century, and Tarski's formal axiomatization
\cite{tarski-whatis} in the twentieth.

Proofs in these axiomatic systems, however, do not look much like
proofs in the \emph{Elements}. Moreover, the modern attitude belies
the fact that for over two thousand years Euclidean geometry was a
remarkably stable practice. On the consensus view, the logical gaps in
Euclid's presentation should have resulted in vagueness or ambiguity
as to the admissible rules of inference. But, in practice, they did
not; mathematicians through the ages and across cultures could read,
write, and communicate Euclidean proofs without getting bogged down in
questions of correctness. So, even if one accepts the consensus view,
it is still reasonable to seek some sort of explanation of the success
of the practice.

Our goal here is to provide a detailed analysis of the methods of
inference that are employed in the \emph{Elements}. We show, in
particular, that the use of diagrams in a Euclidean proof is not soft
and fuzzy, but controlled and systematic, and governed by a
discernible logic.  This provides a sense in which Euclid's methods
are more rigorous than the modern attitude suggests.

Our study draws on an analysis of Euclidean reasoning due to Ken
Manders~\cite{manders:08}, who distinguished between two types of
assertions that are made of the geometric configurations arising in
Euclid's proofs. The first type of assertion describes general
topological properties of the configuration, such as incidence of
points and lines, intersections, the relative position of points along
a line, or inclusions of angles.  Manders called these \emph{co-exact
  attributions}, since they are stable under perturbations of the
diagram; below, we use the term ``diagrammatic assertions'' instead.
The second type includes things like congruence of segments and
angles, and comparisons between linear or angular magnitudes. Manders
called these \emph{exact attributions}, because they are not stable
under small variations, and hence may not be adequately
represented in a figure that is roughly drawn. Below, we use the term
``metric assertions'' instead. Inspecting the proofs in the
\emph{Elements}, Manders observed that the diagrams are only used to
record and infer co-exact claims; exact claims are always made
explicit in the text.  For example, one might infer from the diagram
that a point lies between two others on a line, but one would never
infer the congruence of two segments without justifying the conclusion
in the text.  Similarly, one cannot generally infer, from inspecting
two angles in a diagram, that one is larger than the other; but one
can draw this conclusion if the diagram ``shows'' that the first is
properly contained in the second.

Below, we present a formal axiomatic system, $\na{E}$, which spells
out precisely what inferences can be ``read off'' from the diagram.
Our work builds on Mumma's PhD thesis~\cite{mumma-phd}, which
developed such a diagram-based system, which he called $\na{Eu}$. In
Mumma's system, diagrams are bona-fide objects, which are introduced
in the course of a proof and serve to license inferences. Mumma's
diagrams are represented by geometric objects on a finite coordinate
grid. However, Mumma introduced a notion of ``equivalent diagrams'' to
explain how one can apply a theorem derived from a different diagram
that nonetheless bears the same diagrammatic information.  Introducing
an equivalence relation in this way suggests that, from a logical
perspective, what is really relevant to the proof is the equivalence
class of all the diagrams that bear the same information. We have thus
chosen a more abstract route, whereby we identify the ``diagram'' with
the co-exact information that the physical drawing is supposed to
bear. Nathaniel Miller's PhD dissertation \cite{miller:08} provides
another formal system for diagrammatic reasoning, along these lines,
employing ``diagrams'' that are graph-theoretic objects subject to
certain combinatorial constraints.

Both Mumma and Miller address the issue of how reasoning based on a
particular diagram can secure general conclusions, though they do so in
different ways. In Miller's system, when a construction can result in
topologically distinct diagrammatic configurations, one is required to
consider all the cases, and show that the desired conclusion is
warranted in each. In contrast, Mumma stipulated general rules,
based on the particulars of the construction, that must be followed to
ensure that the facts read off from the particular diagram are
generally valid. Our formulation of $\na{E}$ derives from this latter
approach, which, we argue below, is more faithful to Euclidean
practice.

Moreover, we show that our proof system is sound and complete for a
standard semantics of ``ruler-and-compass constructions,'' expressed
in modern terms. Thus, our presentation of $\na{E}$ is accompanied by
both philosophical and mathematical claims: on the one hand, we claim
that our formal system accurately models many of the key
methodological features that are characteristic of the proofs found in
books I through IV of the \emph{Elements}; and, on the other hand, we
claim that it is sound and complete for the appropriate semantics.

The outline of this paper is as follows. In
Section~\ref{characterizing:section}, we begin with an informal
discussion of proofs in the \emph{Elements}, calling attention to the
particular features that we are trying to model. In
Section~\ref{formal:system:section}, we describe the formal system,
$\na{E}$, and specify its language and rules of inference. In
Section~\ref{comparison:section}, we justify the claim that our
system provides a faithful model of the proofs in the \emph{Elements},
calling attention to points of departure as well as points of
agreement. In Section~\ref{completeness:section}, we show that our
formal system is sound and complete with respect to ruler-and-compass
constructions. In Section~\ref{implementation:section}, we discuss ways
in which contemporary methods of automated reasoning can be used to
implement a proof checker that can mechanically verify proofs in our
system.  Finally, in Section~\ref{conclusions:section}, we summarize
our findings, and indicate some questions and issues that are
not addressed in our work.

\section{Characterizing the \emph{Elements}}
\label{characterizing:section}

In this section, we clarify the claim that our formal system is more
faithful to the \emph{Elements} than other axiomatic systems, by
describing the features of the \emph{Elements} that we take to be
salient.

\subsection{Examples of proofs in the \emph{Elements}}
\label{examples:one:section}

To support our discussion, it will be helpful to have two examples of
Euclidean proofs at hand.

\newtheorem*{Prop10}{Proposition I.10}
\begin{Prop10}
To bisect a given finite straight line.
\end{Prop10}

\begin{center}
\psfrag{c}{$c$}\psfrag{d}{$d$}\psfrag{a}{$a$}\psfrag{b}{$b$}
\includegraphics[height=2cm]{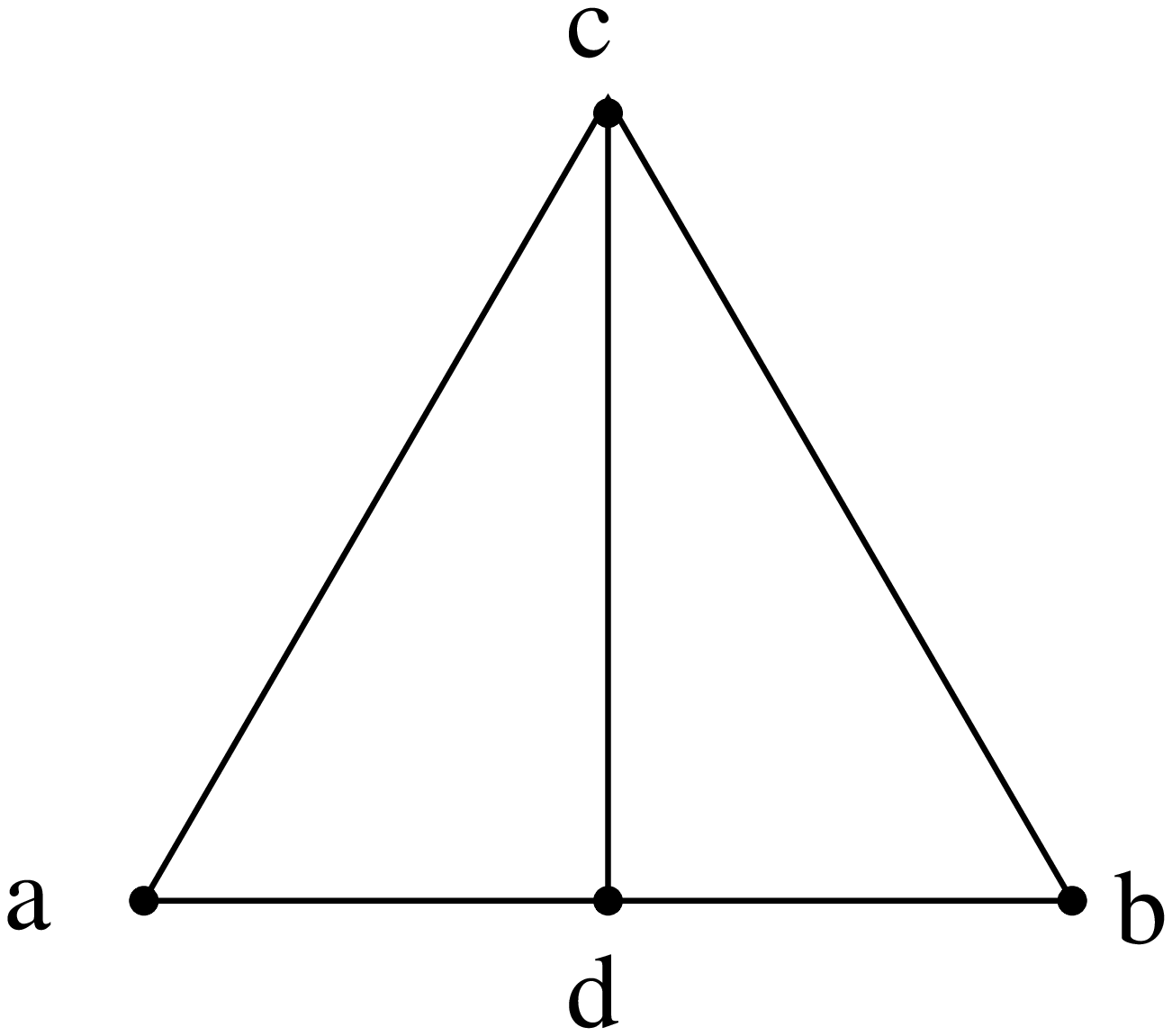}
\end{center}

\begin{proof}
  Let $ab$ be the given finite straight line.\\
  It is required to bisect the finite straight line $ab$.\\
  Let the equilateral triangle $abc$ be constructed on it [I.1],
  and let the angle $acb$ be bisected by the straight line $cd$. [I.9]\\
  I say that the straight line $ab$ is bisected at the point $d$.\\
  For, since $ac$ is equal to $cb$, and $cd$ is common, the two sides
  $ac$, $cd$ are equal the two sides $bc$, $cd$ respectively; and the
  angle $acd$ is equal to the angle $bcd$; therefore the base $ad$ is
  equal to the base $bd$. [I.4]\\
  Therefore the given finite straight line $ab$ has been bisected at $d$.\\
  Q.E.F.
\end{proof}

\noindent This is Proposition 10 of Book I of the \emph{Elements}. All
our references to the \emph{Elements} refer to the Heath translation
\cite{euclid}, though we have replaced upper-case labels for points by
lower-case labels in the proof, to match the description of our formal
system, $\na{E}$.

As is typical in the \emph{Elements}, the initial statement of the
proposition is stated in something approximating natural language. A
more mathematical statement of the proposition is then given in the
opening lines of the proof. The annotations in brackets refer back to
prior propositions, so, for example, the third sentence of the proof
refers to Propositions 1 and 9 of Book I. Notice that what it means
for a point $d$ to ``bisect'' the finite segment $ab$ can be analyzed
into topological and metric components: we expect $d$ to lie on the
same line as $a$ and $b$, and to lie between $a$ and $b$ on that line;
and we expect that the length of the segment from $a$ to $b$ is equal
to the length of the segment from $b$ to $d$. Only the last part of
the claim is made explicit in the text; the other two facts are
implicit in the diagram.

In his fifth century commentary on the first book of the
\emph{Elements}, Proclus divided Euclid's propositions into two
groups: ``problems,'' which assert that a construction can be carried
out, or a diagram expanded, in a certain way; and ``theorems,'' which
assert that certain properties are essential to a given diagram (see
\cite[pp.~63--67]{Proclus}, or \cite[vol.~I, pp.~124--129]{euclid}).
Euclid himself marks the distinction by ending proofs of problems with
the phrase ``that which it was required to do'' (abbreviated by
``Q.E.F.,'' for ``quod erat faciendum,'' by Heath); and ending proofs
of theorems with the phrase ``that which it was required to prove''
(abbreviated by ``Q.E.D.,'' for ``quod erat
demonstratum''). Proposition I.10 calls for the construction of a
point bisecting the line, and so the proof ends with ``Q.E.F.''

\newtheorem*{Prop16}{Proposition I.16}

\begin{Prop16}
In any triangle, if one of the sides be
  produced, then the exterior angle is greater than either of the
  interior and opposite angles.
\end{Prop16}

\begin{center}
\psfrag{c}{$c$}\psfrag{d}{$d$}\psfrag{a}{$a$}\psfrag{b}{$b$}\psfrag{e}{$e$}\psfrag{f}{$f$}\psfrag{g}{$g$}
\includegraphics[height=2cm]{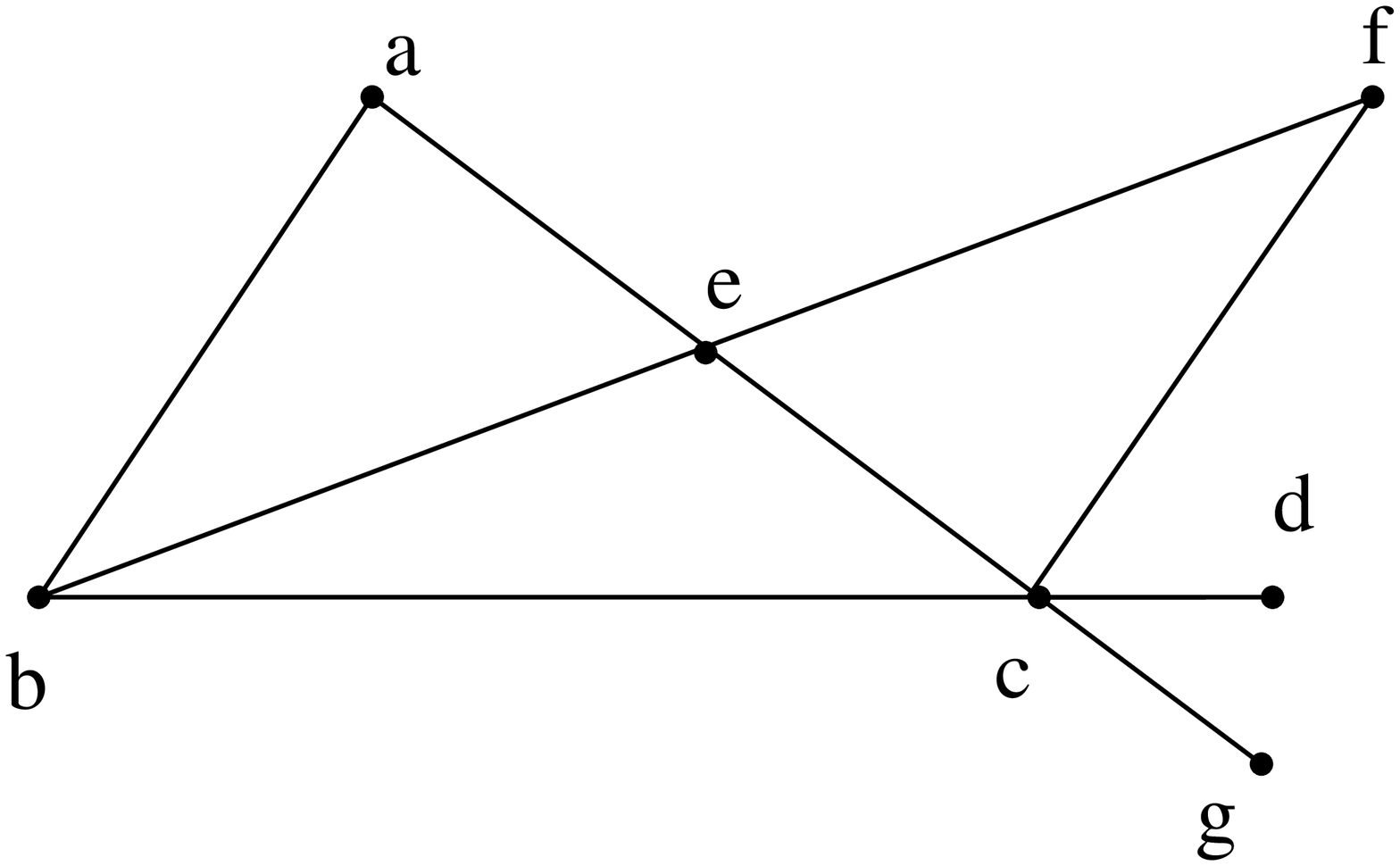}
\end{center}

\begin{proof}
Let $abc$ be a triangle, and let one side of it $bc$ be produced to $d$.\\
I say that the exterior angle $acd$ is greater than either of the
interior and opposite angles $cba$, $bac$.\\
Let $ac$ be bisected at $e$ [I.10],\\
and let $be$ be joined and produced in a straight line to $f$;\\
Let $ef$ be made equal to $be$ [I.3],\\
let $fc$ be joined, [Post.1]\\
and let $ac$ be drawn through to $g$. [Post.2]\\
Then, since $ae$ is equal to $ec$, and $be$ to $ef$, the two sides
$ae$, $eb$ are equal the two sides $ce$, $ef$ respectively; and the angle
$aeb$ is equal to the angle $fec$, for they are vertical angles. [I.15]\\
Therefore the base $ab$ is equal to 
the base $fc$, the triangle $abe$ is equal to the triangle $cfe$, and the
remaining angles equal the remaining angles respectively, namely those
which the equal sides subtend; [I.4]\\
therefore the angle $bae$ is equal to the angle
$ecf$.\\
But the angle $ecd$ is greater than the angle $ecf$; [C.N.5]\\
therefore the angle $acd$ is greater than the angle $bae$.\\
Similarly also, if $bc$ be bisected, the angle $bcg$, that is, the angle
$acd$ [I.15], can be proved greater than the angle $abc$ as well.\\
Therefore etc.\\
Q.E.D.
\end{proof}

\noindent Here, the abbreviation ``Post.''\ in brackets refers to
Euclid's postulates, while the abbreviation ``C.N.''\ refers to the
common notions. Notice that the proposition assumes that the triangle
is nondegenerate. Later on, Euclid will prove the stronger Proposition
I.32, which shows the the exterior angle $acd$ is exactly equal to the
sum of the interior and opposite angles $cba$ and $bac$. But to do
that, he has to develop properties of parallel lines, for which the
current proposition is needed.

In both cases, after stating the theorem, the proofs begin with a
construction phrase (\emph{kataskeue}), in which new objects
are introduced into the diagram. This is followed by the deduction
phase (\emph{apodeixis}), where the desired conclusions are drawn. The
demonstration phase is, for the most part, devoted towards registering
metric information, that is, equalities and inequalities between
various magnitudes. But some of the inferences depend on the
diagrammatic configuration. For example, seeing that angles $aeb$ and
$fec$ are equal in the second proof requires checking the diagram to
see that they are vertical angles. Similarly, seeing that $ecd$ is
greater than $ecf$ is warranted by common notion 5, ``the whole is
greater than the part,'' requires checking the diagram to confirm that
$ecf$ is indeed contained in $ecd$.

\subsection{The use of diagrams}
\label{diagrammatic:nature:section}

The most salient feature of the \emph{Elements} is the fact that
diagrams play a role in the arguments. But what, exactly, does this
mean?

Our first observation is that whatever role the diagram plays, it is
inessential to the communication of the proof. In fact, data on the
early history of the text of the \emph{Elements} is meager, and there
is no chain linking our contemporary diagrams with the ones that
Euclid actually drew; it is likely that, over the years, diagrams were
often reconstructed from the text (see Netz \cite{Netz99}). But a
simple experiment offers more direct support for our claim. If you
cover up the diagrams and reread the proofs in the last section, you
will find that it is not difficult to reconstruct the
diagram. Occasionally, important details are only represented in the
diagram and not the text; for example, in the proof of Proposition
I.10, the text does not indicate that $d$ is supposed to mark the
intersection of the angle bisector and the opposite side of the
triangle. But there is no reason why it couldn't; for example, we
could replace the second sentence with the following one:
\begin{quote}
Let the equilateral triangle $abc$ be constructed on it, let the angle
$acb$ be bisected by the straight line $L$, and let $d$ be the intersection
of $L$ and $ab$.
\end{quote}
The fact that minor changes like this render it straightforward to
construct an adequate diagram suggests that the relevant information
can easily be borne by the text. 

But, to continue the experiment, try reading these proofs, or any of
Euclid's proofs, without the diagram, and without drawing a diagram.
You will likely finding yourself trying to \emph{imagine} the diagram,
to ``see'' that the ensuing diagrammatic claims are justified. So even
if, in some sense, the text-based version of the proof is
self-contained, there is something about the proof, and the tasks we
need to perform to understand the proof, that makes it
``diagrammatic.''

To make the point clear, consider the following example: 
\begin{quote}
  Let $L$ be a line. Let $a$ and $b$ be points on $L$, and let $c$ be
  between $a$ and $b$. Let $d$ be between $a$ and $c$, and let $e$ be
  between $c$ and $b$. Is $d$ necessarily between $a$ and $e$?
\end{quote}
Once again, it is hard to make sense of the question without drawing a
diagram or picturing the situation in your mind's eye; but doing so
should easily convince you that the answer is ``yes.'' With the
diagram in place, there is nothing more that needs to be said. The
inference is immediate, whether or not we are able to cite the axioms
governing the betweenness predicate that would be used to justify the
assertion in an axiomatic proof system.

A central goal of this paper is to analyze and describe these
fundamental diagrammatic inferences. In doing so, we do not attempt to
explain why it is easier for us to verify these inferences with a
physical diagram before us, nor do we attempt to explain the social or
historical factors that made such inferences basic to the
\emph{Elements}. In other words, in analyzing the \emph{Elements}, we
adopt a methodological stance which focuses on the logical structure
of the proofs while screening off other important issues. We return to
a discussion of this in Section~\ref{methodological:section}.

\subsection{The problem of ensuring generality}
\label{generality:section}

On further reflection, the notion of a diagrammatic inference becomes
puzzling. Consider the following example:
\begin{quote}
Let $a$ and $b$ be distinct points, and let $L$ be the line through $a$
and $b$. Let $c$ and $d$ be points on opposite sides
of $L$, and let $M$ be the line through $c$ and $d$. Let $e$ be the
intersection of $L$ and $M$. Is $e$ necessarily between $c$ and $d$?
\end{quote}
Drawing a diagram, or picturing the situation in your mind's eye,
should convince you that the answer is ``yes,'' based on an
``intuitive'' understanding of the concepts involved:
\showdiagram{
\begin{center}
\psfrag{L}{$L$}\psfrag{a}{$a$}\psfrag{b}{$b$}\psfrag{c}{$c$}\psfrag{d}{$d$}\psfrag{e}{$e$}\psfrag{M}{$M$}
\includegraphics[height=2.5cm]{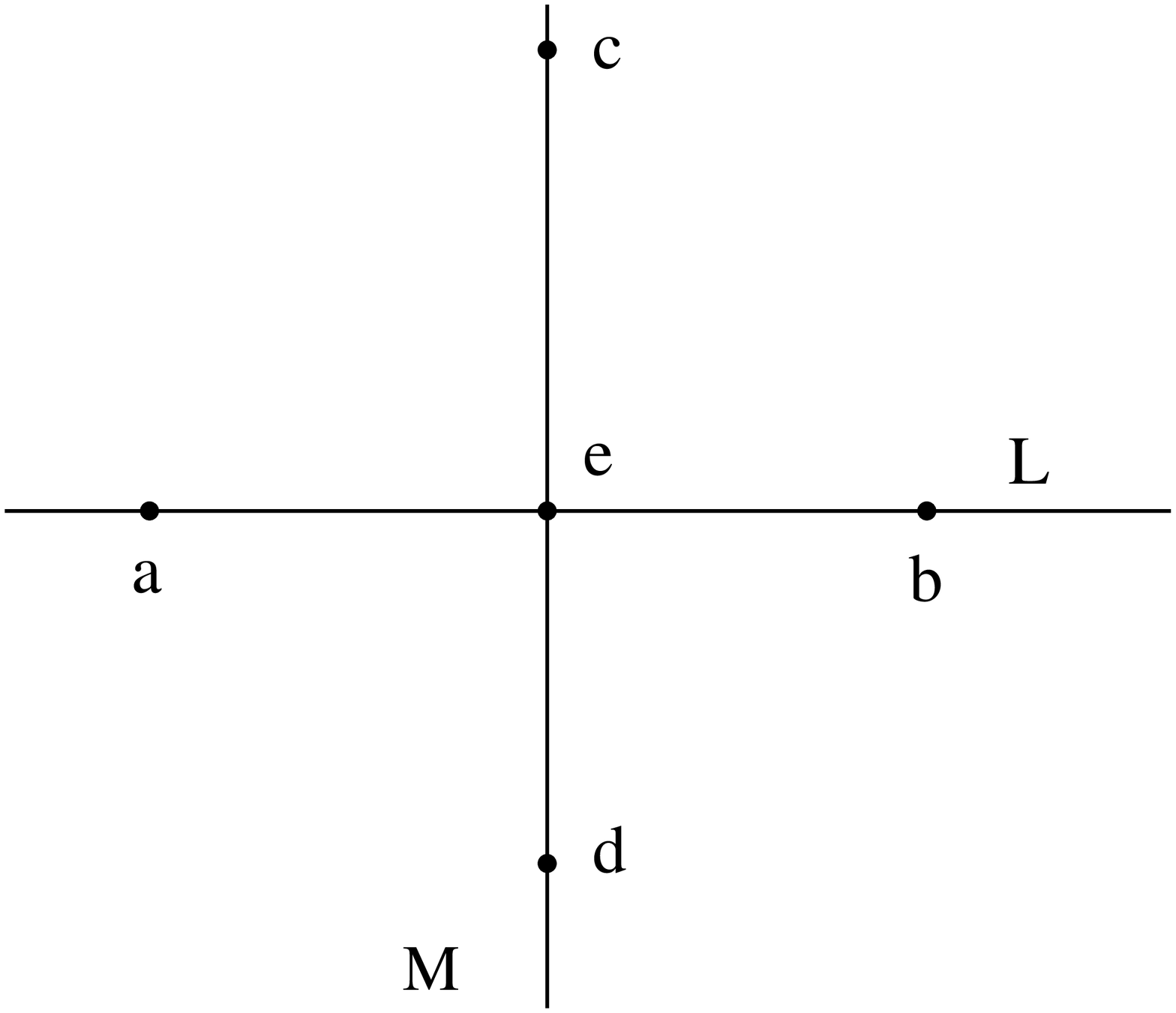}
\end{center}
} In fact, a diagrammatic inference was even implicit in the
instruction ``let $e$ be the intersection of $L$ and $M$,'' namely, in
seeing that $L$ and $M$ necessarily intersect.

So far, all is well. But now suppose we replace the last question in
the example with the following:
\begin{quote}
Is $e$ necessarily between $a$ and $b$?
\end{quote}
Consulting our diagram, we should perhaps conclude that the answer is
``yes.'' But that is patently absurd; we could easily have drawn the
diagram to put $e$ anywhere along $L$. Neither Euclid nor any
competent student of Euclidean geometry would draw the invalid
inference. Thus any respectable notion of ``diagrammatic inference''
has to sanction the first inference in our example, but bar the
second.

There are two morals to be extracted from this little exercise. The
first is that, however the diagram functions in a Euclidean proof,
using the diagram is not simply a matter of reading off features found
in the physical instantiation. Any way of drawing the diagram will
give $e$ a position relative to $a$ and $b$, but none of them can be
\emph{inferred} from the givens. The physical instance of the diagram
thus serves as a token, or artifact, that is intended to be used in
certain ways; understanding the role of the diagram necessarily
involves understanding the intended use.\footnote{Danielle Macbeth
  \cite{macbeth:unp} has characterized this sort of diagram use in
  terms of the Gricean distinction between ``natural'' and
  ``non-natural'' meaning. Manders \cite{manders:08} underscores this
  point by observing that Euclidean diagrams are used equally well in
  reductio proofs, where the conclusion is that the illustrated
  configuration cannot exist. One finds a nice example of this in
  Proposition 10 of Book III, which shows that two distinct circles
  cannot intersect in more than two points. Clearly, in cases like
  this, the diagram does not serve as a ``literal'' or direct
  representation of the relevant configuration.}

The second moral is that the physical instance of the diagram, taken
out of context, does not bear \emph{all} the relevant inferential
data. In the example above, the diagram is symmetric: if we rotate the
diagram a quarter turn and switch the order of the questions, the new
diagram and questionnaire differs from the previous one only by the
labels of the geometric objects; but whereas ``yes'' and then ``no''
are the correct answers to the first set of questions, ``no'' and then
``yes'' are the correct answers to the second.  What this means is
that the inferences that we are allowed to perform depend not just on
the illustration, but also on the preamble; that is, the inference
depends on knowing the construction that the diagram is supposed to
illustrate. Hence, understanding the role of the diagram in Euclidean
practice also involves understanding how the details of the
construction bear upon the allowable inferences.

In Nathaniel Miller's formal system for Euclidean geometry
\cite{miller:08}, every time a construction step can give rise to
different topological configurations, the proof requires a case split
across all the possible configurations. His system provides a calculus
by which one can determine (an upper bound on) all the realizable
configurations (and systematically rule out some of the configurations
that are not realizable). This can result in a combinatorial explosion
of cases, and Miller himself concedes that it can be difficult to work
through them all. (See also Mumma's review \cite{mumma-miller-review}.)
Thus, although Miller's system is sound for the intended semantics and
may be considered ``diagrammatic'' in nature, it seems far removed
from the \emph{Elements}, where such exhaustive case splits are
nowhere to be found. (We will, however, have a lot more to say about
the case distinctions that do appear in the \emph{Elements}; see
Sections~\ref{direct:section} and \ref{departures:section}.)

Mumma's original proof system, $\na{Eu}$ \cite{mumma-phd,mumma:unp:a},
used a different approach. Although proofs in $\na{Eu}$ are based on
particular diagrams, not every feature found in a particular diagram
can be used in the proof.  Rather, one can only use those features of
the diagram that are guaranteed to hold \emph{generally}, given the
diagram's construction. Mumma's system therefore includes precise
rules that determine when a feature has this property. Our system,
$\na{E}$, pushes the level of abstraction one step further: in
$\na{E}$ the diagram \emph{is} nothing more than the collection of
generally valid diagrammatic features that are guaranteed by the
construction. In other words, given the construction in the example
above, we identify the diagram with the information provided by the
construction --- that $a$ and $b$ are distinct points, $L$ is a line,
$a$ is on $L$, $b$ is on $L$, $c$ and $d$ are on opposite sides of
$L$, and so on --- and all the direct diagrammatic consequences of
these data. This requires us to spell out the notion of a ``direct
diagrammatic consequence,'' which is exactly what we do in
Section~\ref{direct:section}.

\subsection{The logical form of proofs in the \emph{Elements}}
\label{logical:section}

It is commonly noted that Euclid's proofs are constructive, in the
sense that existence assertions are established by giving explicit
constructions. One would therefore not expect Euclidean reasoning to
use the full range of classical first-order logic, which allows
nonconstructive existence proofs, but, rather, a suitably constructive
fragment.

In fact, when one surveys the proofs in the \emph{Elements}, one is
struck by how little logic is involved, by modern standards. Go back
to the examples in Section~\ref{examples:one:section}, and count the
instances of logical staples like ``every,'' ``some,'' ``or,'' and
``if \ldots then.'' The results may surprise you.

Of course, the statements of the two propositions are best modeled
with a universal quantifier: we can read Proposition I.10 as the
assertion that ``any finite straight line can be bisected'' and
Proposition I.16 begins with the words ``any triangle.'' Furthermore,
there is an existential quantifier implicit in the statement of
Proposition I.10, which, in modern terms, might be expressed ``for
every finite straight line, there is a point that bisects it.''  In
modern terms, it is the existential quantifier implicit in the
statement of Proposition I.10 that makes this proposition a
``problem'' in Proclus' terminology.  There is no such quantifier
implicit in Proposition I.16, which is therefore a ``theorem.''

Thus, in a Euclidean proposition, an explicit or implicit universal
quantifier serves to set forth the givens, and, if the proposition is
a ``problem,'' an existential statement is used to specify the
properties of the objects to be constructed. What is remarkable is
that these are the only quantifiers one finds in the text; the proof
itself is purely quantifier-free. Not only that; the proof is
virtually \emph{logic free}. A construction step introduces new
objects meeting certain specifications; for example, the third line of
the proof of Proposition I.10 introduces an equilateral triangle. We
will see that in our formalization, the specification can always be
described as a list of atomic formulas and their negations. Other
lines in a Euclidean proof simply make atomic or negated atomic
statements, like ``the base $ad$ is equal to the base $bd$,''
sometimes chained together with the word ``and.''

In other words, Euclidean proofs do little more than introduce objects
satisfying lists of atomic (or negation atomic) assertions, and then
draw further atomic (or negation atomic) conclusions from these, in a
simple linear fashion. There are two minor departures from this
pattern. Sometimes a Euclidean proof involves a case split; for
example, if $ab$ and $cd$ are unequal segments, then one is longer
than the other, and one can argue that a desired conclusion follows in
either case. The other exception is that Euclid sometimes uses a
\emph{reductio}; for example, if the supposition that $ab$ and $cd$
are unequal yields a contradiction then one can conclude that $ab$ and
$cd$ are equal. In our formal system, such case splits are always case
splits on the truth of an atomic formula, and a proof by contradiction
always establishes an atomic formula or its negation.

There is one more feature of Euclid's proofs that is worth calling
attention to, namely, that in Euclid's proofs the construction steps
generally precede the deductive conclusions. Thus, the proofs
generally split into two phases: in the construction
(\emph{kataskeue}) phase, one carries out the construction,
introducing all the objects that will be needed to reach the desired
conclusion; and then in the deduction (\emph{apodeixis}) phase one
infers metric and diagrammatic consequences (see
\cite[pp.~159--160]{Proclus} or \cite[vol.~I,
pp.~129--130]{euclid}). This division is \emph{not} required by our
formal system, which is to say, nothing goes wrong in our proof system
if one constructs some objects, draws some conclusions, and then
carries out another construction. In other words, we take the division
into the two phases to be a stylistic choice, rather than a logical
necessity.  For the most part, one can follow this stylistic
prescription within $\na{E}$, and carry out all the constructions
first. An exception to this occurs when, by $\na{E}$'s lights, some
deductive reasoning is required to ensure that prerequisites for
carrying out a construction step are met. For example, we will see in
Section~\ref{departures:section} that our formal system takes issue
with Euclid's proof of Proposition I.2: where Euclid carries out a
complex construction without further justification, our system
requires an explicit (but brief) argument, amidst the construction, to
ensure that a certain point lies inside a certain circle. But even
Euclid himself sometimes fails to maintain the division between the
two phases, and includes demonstrative arguments in the construction
phase; see, for example, our discussion of Euclid's proof of
Proposition I.44, in Section~\ref{departures:section}.  Thus, our
interpretation of the usual division of a Euclidean proof into
construction and deduction phases is supported by the text of the
\emph{Elements} itself.

\subsection{Nondegeneracy assumptions}
\label{nondegeneracy:section}

As illustrated by our examples, Euclid typically assumes his geometric
configurations are nondegenerate. For example, if $a$ and $b$ are
given as arbitrary points, Euclid assumes they are \emph{distinct}
points, and if $abc$ is a triangle, the points $a$, $b$, and $c$ are
further assumed to be noncollinear. These are also sometimes called
``genericity assumptions''; we are following Wu \cite{wu:94} in using
the term ``nondegeneracy.''

Insofar as these assumptions are implicit in Euclid, his presentation
can be criticized on two grounds:
\begin{enumerate}
\item The theorems are not always as strong as they can be, because
  the conclusions sometimes can still be shown to hold when some of
  the nondegeneracy constraints are relaxed. (Sometimes one needs to
  clarify the reading of the conclusion in a degenerate case.)
\item There are inferential gaps: when Euclid \emph{applies} a theorem
  to the diagram obtained from a construction in the proof of a later
  theorem, he does not check that the nondegeneracy assumptions hold,
  or can be assumed to hold without loss of generality.
\end{enumerate}
With respect the second criticism, Wu writes:
\begin{quote}
In the proof of a theorem, even though the configuration of the
hypothesis at the outset is located in a generic, nondegenerate
position, we are still unable to determine ahead of time whether or
not the degenerate cases will occur when applying other theorems in
the proof process. Not only is the verification of every applied
theorem cumbersome and difficult, but it is actually also impossible
to guarantee that the degenerate cases (in which the theorem is
meaningless or false) do not happen in the proof process. On the other
hand, we have no effective means to judge how much to restrict the
statement of a theorem (to be proved) in order to ensure the truth of
the theorem. These problems make it impossible for the Euclidean
method of theorem proving to meet the requirements of necessary
rigor. \cite[p.~118]{wu:94}
\end{quote}

Wu's comments refer to geometric theorem proving in general, not just
the theorems of the \emph{Elements}. With respect to the latter, we
feel that the quote overstates the case: for the most part, the
nondegeneracy requirements for theorem application in Euclid are
easily met by assuming that the construction is appropriately generic.
We discuss a mild exception in Section~\ref{departures:section},
noting that, according to $\na{E}$, Euclid should have said a few more
words in the proof of Proposition I.9. But we do not know of any
examples where substantial changes are needed.

Furthermore, the first criticism is only damning insofar as the
degenerate cases are genuinely interesting. Nonetheless, from a modern
standpoint, it is better to articulate just what is required in the
statement of a theorem. Thus, we have chosen to ``go modern'' with
$\na{E}$, in the sense that any distinctness assumptions (inequality
of points, non-incidence of points and lines) that are required have
to be stated explicitly as hypotheses. Although this marks a slight
departure from Euclid, the fact that all assumptions are made explicit
provides a more flexible framework to explore the issue as to which
assumptions are implicit in his proofs.

\subsection{Our methodology}
\label{methodological:section}

We have cast our project as an attempt to model Euclidean diagrammatic
proof, aiming to clarify its logical form, and, in particular, the
nature of diagrammatic inference. In casting our project in this way,
we are adopting a certain methodological stance. From a logical
standpoint, what makes a Euclidean proof ``diagrammatic'' is
\emph{not} the fact that we find it helpful to consult a diagram in
order to verify the correctness of the proof, or that, in the absence
of such a physical artifact, we tend to roll our eyes towards the back
of our heads and imagine such a diagram. Rather, the salient feature
of Euclidean proof is that certain sorts of inferences are admitted as
basic, and are made without further justification. When we say we are
analyzing Euclidean diagrammatic reasoning, we mean simply that we are
trying to determine which inferences have this basic character, in
contrast to the geometrically valid inferences that are spelled out in
greater detail in the text of the \emph{Elements}.

Our analysis may therefore seem somewhat unsatisfying, in the sense
that we do not attempt to explain \emph{why} the fundamental methods
of inference in the \emph{Elements} are, or can be, or should be,
taken to be basic. This is not to imply that we do not take
such questions to be important. Indeed, it is just \emph{because} they
are such obvious and important questions that we are taking pains to
emphasize the restricted character of our project.

What makes these questions difficult is that it is often not clear just
what type of answer or explanation one would like. In order to explain
why Euclidean practice is the way it is, one might reasonably invoke
historical, pedagogical, or more broadly philosophical considerations.
It may therefore help to highlight various types of analysis
that are \emph{not} subsumed by our logical approach. It does not
include, \emph{per se}, any of the following:
\begin{itemize}
\item a \emph{historical} analysis of how the \emph{Elements} came to
  be and attained the features we have described;
\item a \emph{philosophical} analysis as to what characterizes the
  inferences above as epistemically special (beyond that they
  interpret the ruler-and-compass constructions of modern geometric
  formalizations, and are sound and complete for the corresponding
  semantics), or in what sense they should be accepted as ``immediate'';
\item a \emph{psychological} or \emph{cognitive} or \emph{pedagogical}
  analysis of the human abilities that make it possible, and useful,
  to understand proofs in that form; or
\item a \emph{computational} analysis as to the most efficient data
  structures and algorithms for verifying the inferences we have
  characterized as ``Euclidean,'' complexity upper and lower bounds,
  or effective search procedures. 
\end{itemize}
We do, however, take it to be an important methodological point that
the questions we address here can be separated from these related
issues. We hope, moreover, that the understanding of Euclidean proof
that our analysis provides can support these other lines of inquiry.
We return to a discussion of these issues in
Section~\ref{conclusions:section}.

\section{The formal system $\na{E}$}
\label{formal:system:section}

\subsection{The language of $\na{E}$}
\label{language:section}

The language of $\na{E}$ is six-sorted, with sorts for points, lines,
circles, segments, angles, and areas. There are variables ranging over
the first three sorts; we use variables $a, b, c, \ldots$ to range
over points, $L, M, N, \ldots$ to range over lines, and $\alpha,
\beta, \gamma, \ldots$ to range over circles. In addition to the
equality symbol, we have the following basic relations on elements of these
sorts:
\begin{itemize}
\item $\on(a,L)$: point $a$ is on line $L$
\item $\sameside(a,b,L)$: points $a$ and $b$ are on the same side of
  line $L$
\item $\mybetween(a,b,c)$: points $a$, $b$, and $c$ are distinct and
  collinear, and $b$ is between $a$ and $c$
\item $\on(a,\alpha)$: point $a$ is on circle $\alpha$
\item $\inside(a, \alpha)$: point $a$ is inside circle $\alpha$
\item $\mycenter(a, \alpha)$: point $a$ is the center of circle
  $\alpha$
\end{itemize}
Note that $\mybetween(a,b,c)$ denotes a strict betweenness relation,
and $\sameside(a,b,L)$ entails that neither $a$ nor $b$ is on $L$.  We
also have three versions of an additional relation symbol, to keep
track of the intersection of lines and circles:
\begin{itemize}
\item $\intersects(L,M)$: line $L$ and $M$ intersect
\item $\intersects(L,\alpha)$: line $L$ intersects circle $\alpha$
\item $\intersects(\alpha, \beta)$: circles $\alpha$ and $\beta$
  intersect
\end{itemize}
In each case, by ``intersects'' we really mean ``intersects
transversally.'' In other words, two lines intersect when they have
exactly one point in common, and two lines, or a line and a circle,
intersect when they have exactly two points in common.

The objects of the last three sorts represent magnitudes. There are no
variables ranging over these sorts; instead, one obtains objects of
these sorts by applying the following functions to points:
\begin{itemize}
\item $\segment(a,b)$: the length of the line segment from $a$ to $b$,
  written $\seg{ab}$
\item $\myangle(a, b, c)$: the magnitude of the angle $abc$, written
  $\angle abc$
\item $\myarea(a, b, c)$: the area of triangle $abc$, written $\area{abc}$
\end{itemize}
In addition to the equality relation, we have an addition function,
$+$, a less-than relation, $<$, and a constant, $0$, on each magnitude
sort. Thus, for example, the expression $\seg{ab} = \seg{cd}$ denotes
that the line segment determined by $a$ and $b$ is congruent to the
line segment determined by $c$ and $d$, and $\seg{ab} < \seg{cd}$
denotes that it is strictly shorter. The symbol $0$ is included for
convenience; we could have, in a manner more faithful to Euclid, taken
magnitudes to be strictly positive, with only minor modifications to
the axioms and rules of inference described below. Finally, we also
include a constant, ``$\rightangle$,'' of the angle sort. Thus we model
the statement ``$abc$ is a right angle'' as $\angle abc =
\rightangle$.

The assertion ``$\mybetween(a,b,c)$'' is intended to denote that $b$
is \emph{strictly} between $a$ and $c$, which is to say, it implies
that $b$ is not equal to either $a$ or $c$. In
Section~\ref{completeness:section}, we will see that, in this respect,
it differs from the primitive used by Tarski in his axiomatization of
Euclidean geometry. One reason that we have chosen the strict version
is that it seems more faithful to Euclidean practice; see the
discussion in Section~\ref{nondegeneracy:section}. Another is that it
seems to have better computational properties; see
Section~\ref{implementation:section}.

The atomic formulas are defined as usual. A \emph{literal} is an
atomic formula or a negated atomic formula. We will sometimes refer to
literals as ``assertions,'' since, as we have noted, statements found
in proofs in the \emph{Elements} are generally of this form (or, at
most, conjunctions of such basic assertions). Literals involving the
relations on the first three sorts are ``diagrammatic assertions,''
and literals involving the relations on the last three sorts are
``metric assertions.''

Additional predicates can be defined in terms of the basic ones
presented here. For example, we can take the assertion $\seg{ab} \leq
\seg{cd}$ to be shorthand for $\lnot (\seg{cd} <
\seg{ab})$. Similarly, we can assert that $a$ and $b$ are on different
sides of a line $L$, written $\diffside(a,b,L)$, by making the
sequence of assertions $\lnot \on(a,L), \lnot \on(b,L), \lnot
\sameside(a,b,L)$.  Similarly, we can define $\outside(a,\alpha)$ to
be the conjunction $\lnot \inside(a,\alpha), \lnot \on(a,
\alpha)$. Definitional extensions like these are discussed in
Section~\ref{formal:language:section}.

It is worth mentioning, at this point, that diagrammatic assertions
like ours rarely appear in the text of Euclid's proofs. Rather, they
are implicitly the result of diagrammatic hypotheses and construction
steps, and they, in turn, license further construction steps and
deductive inferences. But this fact \emph{is} adequately captured by
$\na{E}$: even though raw diagrammatic assertions \emph{may} appear in
proofs, the rules are designed so that typically they do not have to.
Consider, for example, the example in
Section~\ref{generality:section}. In our system, the construction
``let $e$ be the point of intersection of $L$ and $M$'' is licensed by
the diagrammatic assertion $\intersects(L,M)$, which, in turn, is
licensed by the fact that $M$ contains two points, $c$ and $d$, that
are on opposite sides of $L$. But we will take the assertion
$\intersects(L,M)$ to be a direct consequence of diagrammatic assertions
that result from the construction, which allows this fact to
license the construction step without explicit mention. And once $e$
has been designated the point of intersection, the fact that $e$ is
between $c$ and $d$ is another direct consequence of the diagram
assertions in play, and hence can be used to license future
constructions and metric assertions. We discuss the relationship
between our formal language and the informal text of the
\emph{Elements} in more detail in Section
\ref{formal:language:section}.

\subsection{Proofs in $\na{E}$}
\label{proofs:section}

Theorems in $\na{E}$ have the following logical form:
\[
\fa {\vec a, \vec L, \vec \alpha} (\ph(\vec a, \vec L, \vec \alpha)
\limplies \ex {\vec b, \vec M, \vec \beta} \psi(\vec a, \vec b, \vec
L, \vec M, \vec \alpha, \vec \beta)),
\]
where $\ph$ is a conjunction of literals, and $\psi$ is either a
conjunction of literals or the symbol $\bot$, for ``falsity'' or
``contradiction.'' Put in words, theorems in $\na{E}$ make statements
of the following sort:
\begin{quote}
  Given a diagram consisting of some points, $\vec a$, some lines,
  $\vec L$, and some circles, $\vec \alpha$, satisfying assertions
  $\ph$, one can construct points $\vec b$, lines $\vec M$, and
  circles $\vec \beta$, such that the resulting diagram satisfies
  assertions $\psi$.
\end{quote}
If the list $\vec b, \vec M, \vec \beta$ is nonempty, the theorem is a
``problem,'' in Proclus' terminology. If that list is empty and $\psi$
is not $\bot$, we have a ``theorem,'' in Proclus' sense. If $\psi$ is
$\bot$, the theorem asserts the impossibility of the configuration
described by $\ph$.

In our proof system, we will represent a conjunction of literals by
the corresponding set of literals, and the initial universal
quantifiers will be left implicit. Thus, theorems in our system will
be modeled as sequents of the form
\[
\Gamma \Rightarrow \ex{\vec b, \vec M,
  \vec \beta.} \Delta,
\]
where $\Gamma$ and $\Delta$ are sets of literals, and $\vec b, \vec M,
\vec \beta$ do not occur in $\Gamma$. Assuming the remaining variables
in $\Gamma$ and $\Delta$ are among $\vec a, \vec L, \vec \alpha$, the
interpretation of the sequent is as above: given objects $\vec a, \vec
L, \vec \alpha$ satisfying the assertions in $\Gamma$, there are
objects $\vec b, \vec M, \vec \beta$ satisfying the assertions in
$\Delta$.

As is common in the proof theory literature, if $\Gamma$ and $\Gamma'$
are finite sets of literals and $\ph$ is a literal, we will use
$\Gamma, \Gamma'$ to abbreviate $\Gamma \cup \Gamma'$ and $\Gamma,
\ph$ to abbreviate $\Gamma \cup \{ \ph \}$. Beware, though: in the
literature it is more common to read sets on the right side of a
sequent arrow disjunctively, rather than conjunctively, as we do. Thus
the sequent above corresponds to the single-succedent sequent $\Gamma
\Rightarrow \ex{\vec b, \vec M, \vec \beta} (\bigwedge \Delta)$ in a
standard Gentzen calculus for first-order logic.

Having described the theorems in our system, we now describe the
proofs. As noted in Section~\ref{logical:section}, there are two sorts
of steps in a Euclidean proof: construction steps, which introduce new
objects into the diagram, and deduction steps, which infer facts about
objects that have already been introduced. Thus, after setting forth
the hypotheses, a typical Euclidean proof might have the following
form:
\begin{quote}
Let $a$ be a point such that \ldots \\
Let $b$ be a point such that \ldots \\
Let $L$ be a line such that \ldots \\
\ldots \\
Hence \ldots \\
Hence \ldots \\
Hence \ldots
\end{quote}
Application of a previously proved theorem fits into this framework:
if the theorem is a ``problem,'' in Proclus' terminology, applying it
is a construction step, while if it is a ``theorem,'' applying it is a
demonstration step. The linear format is occasionally broken by a
proof by cases or a proof by contradiction, which temporarily
introduces a new assumption. For example, a proof by cases might have
the following form:
\begin{quote}
Suppose $A$. \\
\hspace*{10pt} Hence \ldots \\
\hspace*{10pt} Hence \ldots \\
\hspace*{10pt} Hence $B$. \\
On the other hand, suppose not $A$. \\
\hspace*{10pt} Hence \ldots \\
\hspace*{10pt} Hence \ldots \\
\hspace*{10pt} Hence $B$. \\
Hence $B$.
\end{quote}

Proofs in $\na{E}$ can be represented as sequences of assertions in
this way, where the validity of the assertion given at any line in the
proof depends on the hypotheses of the theorem, as well as any
temporary assumptions that are in play. Below, however, we will adopt
conventional proof-theoretic notation, and take each line of the proof
to be a sequent $\Gamma \Rightarrow \ex{\vec x.}  \Delta$, where
$\Gamma$ represents all the assumptions that are operant at that stage
of the proof, $\vec x$ represent all the geometric objects that have
been introduced, and $\Delta$ represents all the conclusions that have
been drawn.

Thus, in our formal presentation of the proof system, a single
construction step involves passing from a sequent of the form $\Gamma
\Rightarrow \ex{\vec x.} \Delta$ to a sequent of the form $\Gamma
\Rightarrow \ex{\vec x, \vec y.}  \Delta, \Delta'$, where $\vec y$ are
variables for points, lines, and/or circles that do not occur in the
original sequent. That is, the step asserts the existence of the new
objects, $\vec y$, with the properties asserted by $\Delta'$. In
contrast, demonstration steps pass from a sequent of the form $\Gamma
\Rightarrow \ex{\vec x.} \Delta$ to one of the form $\Gamma
\Rightarrow \ex{\vec x.} \Delta, \Delta'$, without introducing new
objects. These include:
\begin{itemize}
\item Diagrammatic inferences: here $\Delta'$ consists of a direct
  diagrammatic consequence of diagrammatic assertions in $\Gamma,
  \Delta$;
\item Metric inferences: here $\Delta'$ consists of a direct metric
  consequence of metric assertions in $\Gamma, \Delta$; and
\item Transfer inferences: here $\Delta'$ consists of a metric or
  diagrammatic assertion that can be inferred from metric and
  diagrammatic diagrammatic assertions in $\Gamma, \Delta$.
\end{itemize}
We will describe these inferences in detail in the sections that
follow.

We have already noted that applying a previously proved theorem may or
may not introduce new objects. Suppose we have proved a theorem of the
form $\Pi \Rightarrow \ex{\vec y.} \Theta$, and we are at a stage in
our proof where we have established the sequent $\Gamma \Rightarrow
\ex{\vec x.} \Delta$. The first theorem, that is, the hypotheses in
$\Pi$, may concern a right triangle $abc$, whereas we may wish to
apply it to a right triangle $def$. Thus, the inference may require
renaming the variables of the first theorem.  Furthermore, we may wish
to extract only some of the conclusions of the theorem, and discard
the others.  Applying such a theorem, formally, involves doing the
following:
\begin{itemize}
\item renaming the variables of $\Pi \Rightarrow \ex{\vec y.} \Theta$
  to obtain a sequent $\Pi' \Rightarrow \ex{\vec y'.} \Theta'$, so
  that all the free variables of that sequent are among the variables
  of $\Gamma, \Delta$, and the variables $\vec y'$ do not occur in
  $\Gamma, \Delta$;
\item checking that every element of $\Pi'$ is a direct diagram or
  metric consequence of $\Gamma, \Delta$;
\item selecting some subset $\Delta'$ of the conclusions $\Theta'$ and
  the sublist $\vec z$ of variables from among $\vec y'$ that occur in
  $\Theta'$;
\item and then concluding the sequent $\Gamma \Rightarrow \ex{\vec x,
    \vec z.} \Delta, \Delta'$.
\end{itemize}
In words, suppose that, assuming that some geometric objects satisfy
the assertions $\Gamma$, we have constructed objects $\vec x$
satisfying $\Delta$. Suppose, further, that, by a previous theorem,
the assertions in $\Gamma$ and $\Delta$ imply the existence of new
objects $\vec z$ satisfying $\Delta'$. Then we can introduce new
objects $\vec z$, satisfying the assertions in $\Delta'$.

We also adhere to the common proof-theoretic practice of representing
our proofs as trees rather than sequences, where the sequent at each
node is inferred from sequents at the node's immediate predecessors.
For the most part, trees will be linear, in the sense that each node
has a single predecessor. The only exceptions arise in a proof by
cases or a proof by contradiction. In the first case, one can
establish a conclusion using a case split on atomic formulas. Such a
proof has the following form:
\begin{prooftree}
\AXM{\Gamma \Rightarrow \ex{\vec x.} \Delta}
\AXM{\Gamma, \Delta, \ph \Rightarrow \ex{\vec y.} \Delta'}
\AXM{\Gamma, \Delta, \lnot \ph \Rightarrow \ex{\vec y.} \Delta'}
\TIM{\Gamma \fCenter \ex{\vec x, \vec y.} \Delta,\Delta'}
\end{prooftree}
In words, suppose that, given geometric objects satisfying the
assertions $\Gamma$, we have constructed objects $\vec x$ satisfying
$\Delta$. Suppose, further, that given objects satisfying $\Gamma$ and
$\Delta$, we can construct additional objects $\vec y$ satisfying
$\Delta'$, whether or not $\ph$ holds. Then, given geometric objects
satisfying the assertions $\Gamma$, we can obtain objects $\vec x,
\vec y$ satisfying the assertions in $\Delta, \Delta'$. 

Recall that we have included the symbol $\bot$, or ``contradiction,''
among our basic atomic assertions. Since the rules described below
allow one to infer anything from a contradiction, we can use case
splits to simulate proof by contradiction, as follows. Suppose,
assuming $\lnot \ph$, we establish $\bot$. Then from $\lnot \ph$ we
can establish $\ph$. Since $\ph$ certainly follows from $\ph$, we have
shown that $\ph$ follows in any case.

Finally, we need to model two ``superposition'' inferences used by
Euclid in Propositions 4 and 8 of Book I, to establish the familiar
``side-side-side'' and ``side-angle-side'' criteria for triangle
congruences. The interpretation of these rules has been an ongoing
topic of discussion for Euclid's commentators (see
Heath~\cite[pp.224--228,249--250]{euclid}, Mancosu
\cite[pp.~28--33]{mancosu:96}, or Mueller
\cite[pp.~21--24]{mueller:81}). But the inferences have a very natural
modeling in our system, described in
Section~\ref{superposition:section} below.

A proof that ends with the sequent $\Gamma \Rightarrow \ex {\vec x'.}
\Delta'$ constitutes a proof of $\Gamma \Rightarrow \ex {\vec x.}
\Delta$ exactly when there is a map $f$ from $\vec x$ to the variables
of $\Gamma, \Delta'$ such that, under the renaming, every element of
$\Delta^f$ is contained in or a diagrammatic consequence of $\Delta'$.
In other words, we have succeeded in proving the theorem when we have
constructed the requisite objects and shown that they have the claimed
properties.\footnote{Note that the function $f$ can map an
  existentially quantified variable in $\vec x$ to one of the
  variables in $\Gamma$. This means that the theorem ``assuming
  $p$ is on $L$, there is a point $q$ on $L$'' has the trivial proof:
  ``assuming $p$ is on $L$, $p$ is on $L$.'' 
  
  We are, however, glossing over some technical details concerning the
  usual treatment of bound variables and quantifiers. For example,
  technically, we should require that no variable of $\Gamma$ conflict
  with the bound variables $\vec x$ of the sequent. It may be
  convenient to assume that we simply use separate stocks of variables
  for free (implicitly universally quantified) variables and bound
  (existentially quantified) variables. Or, better, one should
  construe all our claims as holding ``up to renaming of bound
  variables.''}
  
We claim that our formal system captures all the essential features of
the proofs found in Books I to IV of the \emph{Elements}. To be more
precise, the \emph{Elements} includes a number of more complicated
inferences that are easily modeled in terms of our basic rules.  To
start with, Euclid often uses more elaborate case splits than the ones
defined above, for example, depending on whether one segment is
shorter than, the same length as, or longer than another.  This is
easily represented in our system as a sequence of two case splits.
Also, Euclid often implicitly restricts attention to one case, without
loss of generality, where the treatment of the other is entirely
symmetric.  Furthermore, we have focused on triangles; the handling of
convex figures like rectangles and their areas can be reduced to these
by introducing defined predicates. In
Section~\ref{formal:language:section}, we describe some of the ways
that ``syntactic sugar'' could be used to make $\na{E}$'s proofs even
more like Euclid's. Thus a more precise formulation of our claim is
that if we use a suitable textual representation of proofs, then,
modulo syntactic conventions like these, proofs in our formal system
look very much like the informal proofs found in the
\emph{Elements}.\footnote{\label{case:footnote} The manner of
  presenting proofs used above, whereby suppositional reasoning is
  indicated by indenting or otherwise setting off subarguments,
  amounts to the use of what are known as ``Fitch diagrams.''

  Since the objects constructed to satisfy the conclusion of a proof
  by cases can depend on the case, we have glossed over details as to
  how our formal case splits should be represented in Fitch-style
  proofs. But see the second example in
  Section~\ref{technical:section} for one way of doing this.} Some
examples are presented in Section~\ref{examples:section} below to help
substantiate this claim.  Some ways in which proofs in our system
depart substantially from the text of the \emph{Elements} are
discussed in Section~\ref{departures:section}.

To complete our description of $\na{E}$, we now need to describe:
\begin{enumerate}
\item the construction rules, 
\item the diagrammatic inferences,
\item the metric inferences, 
\item the diagram-metric transfer inferences, and
\item the two superposition inferences.
\end{enumerate}
These are presented in
Sections~\ref{construction:rules:section}--\ref{superposition:section}.
The diagrammatic inferences, metric inferences, and diagram-metric
transfer inferences will be presented as lists of first-order axioms,
whereas what we really mean is that in a proof one is allowed to
introduce any ``direct consequence'' of those axioms. This requires us
to spell out a notion of ``direct consequence,'' which we do in
Section~\ref{direct:section}. In the meanwhile, little harm will come
of thinking of the direct consequences as being the assertions that
are first order consequences of the axioms, together with the
assertions in $\Gamma, \Delta$.

\subsection{Construction rules}
\label{construction:rules:section}

In this section, we present a list of construction rules for $E$.
Formally, these are described by sequents of the form $\Pi \Rightarrow
\ex {\vec x.} \Theta$, where the variables $\vec x$ do not appear in
$\Pi$. Applying such a construction rule means simply applying this
sequent as a theorem, as described in the last section. In other
words, one can view our construction rules as a list of ``built-in''
theorems that are available from the start. Intuitively, $\vec x$ are
the objects that are constructed by the rule; $\Pi$ are the
preconditions that guarantee that the construction is
possible,\footnote{The conditions that are prerequisite to a
  construction are called the \emph{diarismos} by Proclus; see
  \cite[Book I, p.~130]{euclid} or \cite[p.~160]{Proclus}.} and
$\Theta$ are the properties that characterize the objects that are
constructed.

We pause to comment on our terminology. What the rules below have in
common is that they serve to introduce new objects to the diagram.
Sometimes an object that is introduced is uniquely determined, as is
the case, for example, with the rule ``let $a$ be the intersection of
$L$ and $M$.'' In other cases, there is an arbitrary choice involved,
as is the case with the rule ``let $a$ be a point on $L$''. We are
referring to both as ``construction rules,'' though one might object
that picking a point is not really a ``construction.'' It might be
more accurate to describe them as ``rules that introduce new objects
into the diagram,'' but we have opted for the shorter locution. Our
choice is made reasonable by the fact that the rules are all
\emph{components} of Euclidean constructions. Insofar as picking a
point $c$ and connecting it to two points $a$ and $b$ can be seen as
``constructing a triangle on the segment $ab$,'' it is reasonable to
call the rule that allows one to pick $c$ a ``construction rule.''

For readability, the sequents are described informally. First, we
provide a natural-language description of the construction, such as
``let $a$ be a point on $L$.'' This is followed by a more precise
specification of the prerequisites to the construction (corresponding
to $\Pi$ in the sequent $\Pi \Rightarrow \ex {\vec x.} \Theta$), and
the conclusion (corresponding to $\Theta$). Furthermore, when one
constructs a point on a line, for example, one has the freedom to
choose such a point distinct from any of the other points already in
the diagram, and to specify that it does not lie on various lines and
circles. The ability to do so is indicated by the optional ``[distinct
from \ldots]'' clause; for example, assuming the lines $L$ and $M$ do
not coincide, one can say ``let $a$ be a point on $L$, distinct from
$b$, $M$, and $\alpha$.'' As noted in
Section~\ref{nondegeneracy:section}, both the ability to specify, and
the requirement of specifying, such ``distinctness'' conditions marks
a departure from Euclid. In the presentation of the construction rules
below, such conditions are abbreviated ``[distinct from \ldots].''
Similarly, the requirement that $L$ be distinct from all the other
lines mentioned is abbreviated ``[$L$ is distinct from lines
\ldots],'' and so on. So the example we just considered is an instance
of the second rule on the list that follows, and becomes
\[
L \neq M \Rightarrow \ex{a.} \on(a,L), a \neq b, \lnot \on(a, M),
\lnot \on(a, \alpha)
\]
when expressed in sequent form.

\bigskip

\noindent {\bf Points}

\begin{enumerate}
\item Let $a$ be a point [distinct from \ldots].\\
Prerequisites: none\\
Conclusion: [$a$ is distinct from\ldots]

\item Let $a$ be a point on $L$ [distinct from \ldots].\\
  Prerequisites: [$L$ is distinct from lines\ldots]\\
  Conclusion: $a$ is on $L$, [$a$ is distinct from\ldots]

\item Let $a$ be a point on $L$ between $b$ and $c$ [distinct from
  \ldots]. \\
  Prerequisites: $b$ is on $L$, $c$ is on $L$, $b \neq c$, [$L$ is
  distinct from lines \ldots] \\
  Conclusion: $a$ is on $L$, $a$ is between $b$ and $c$, [$a$ is distinct
  from\ldots] 

\item Let $a$ be a point on $L$ extending the segment from $b$ to $c$
  [with $a$ distinct from\ldots]. \\
  Prerequisites: $b$ is on $L$, $c$ is on $L$, $b \neq c$, [$L$ is
  distinct from lines \ldots] \\
  Conclusion: $a$ is on $L$, $c$ is between $b$ and $a$, [$a$ is distinct
  from\ldots]

\item Let $a$ be a point on the same side of $L$ as $b$ [distinct
  from\ldots] \\
  Prerequisite: $b$ is not on $L$ \\
  Conclusion: $a$ is on the same side of $L$ as $b$, [$a$ is distinct
  from\ldots]

\item Let $a$ be a point on the side of $L$ opposite $b$ [distinct
  from\ldots] \\
  Prerequisite: $b$ is not on $L$. \\
  Conclusion: $a$ is not on $L$, $a$ is on the same side of $L$ as
  $b$, [$a$ is distinct from\ldots]

\item Let $a$ be a point on $\alpha$ [distinct from \ldots]. \\
  Prerequisite: [$\alpha$ is distinct from other circles] \\
  Conclusion: $a$ is on $\alpha$, [$a$ is distinct from\ldots]

\item Let $a$ be a point inside $\alpha$ [distinct from \ldots]. \\
  Prerequisites: none \\
  Conclusion: $a$ is inside $\alpha$, [$a$ is distinct from\ldots]

\item Let $a$ be a point outside $\alpha$ [distinct from \ldots]. \\
  Prerequisites: none \\
  Conclusion: $a$ is outside $\alpha$, [$a$ is distinct from\ldots]

\end{enumerate}

\bigskip

\noindent {\bf Lines and circles} \nopagebreak[4]

\begin{enumerate}

\item Let $L$ be the line through $a$ and $b$. \\
  Prerequisite: $a \neq b$ \\
  Conclusion: $a$ is on $L$, $b$ is on $L$
  
\item Let $\alpha$ be the circle with center $a$ passing
  through $b$. \\
  Prerequisite: $a \neq b$ \\
  Conclusion: $a$ is the center of $\alpha$, $b$ is on $\alpha$

\end{enumerate}

\bigskip

To make sense of the next list of constructions, recall that we are
using the word ``intersect'' to refer to transversal intersection. For
example, saying that two circles intersect means that they meet in
exactly two distinct points.

\bigskip

\noindent {\bf Intersections}

\begin{enumerate}

\item Let $a$ be the intersection of $L$ and $M$. \\
  Prerequisite: $L$ and $M$ intersect \\
  Conclusion: $a$ is on $L$, $a$ is on $M$

\item Let $a$ be a point of intersection of $\alpha$ and $L$. \\
  Prerequisite: $\alpha$ and $L$ intersect \\
  Conclusion: $a$ is on $\alpha$, $a$ is on $L$

\item Let $a$ and $b$ be the two points of intersection of $\alpha$
  and $L$. \\
  Prerequisite: $\alpha$ and $L$ intersect \\
  Conclusion: $a$ is on $\alpha$, $a$ is on $L$, $b$ is on $\alpha$,
  $b$ is on $L$, $a \neq b$

\item Let $a$ be the point of intersection of $L$ and $\alpha$ between
  $b$ and $c$. \\
  Prerequisites: $b$ is inside $\alpha$, $b$ is on $L$, $c$ is not
  inside $\alpha$, $c$ is not on $\alpha$, $c$ is on $L$ \\
  Conclusion: $a$ is on $\alpha$, $a$ is on $L$, $a$ is between $b$
  and $c$

\item Let $a$ be the point of intersection of $L$ and $\alpha$
  extending the segment from $c$ to $b$. \\
  Prerequisites: $b$ is inside $\alpha$, $b$ is on $L$, $c \neq b$,
  $c$ is on $L$. \\
  Conclusion: $a$ is on $\alpha$, $a$ is on $L$, $b$ is between $a$
  and $c$

\item Let $a$ be a point on the intersection of $\alpha$ and
  $\beta$. \\
  Prerequisite: $\alpha$ and $\beta$ intersect \\
  Conclusion: $a$ is on $\alpha$, $a$ is on $\beta$

\item Let $a$ and $b$ be the two points of intersection of $\alpha$
  and $\beta$. \\
  Prerequisite: $\alpha$ and $\beta$ intersect \\
  Conclusion: $a$ is on $\alpha$, $a$ is on $\beta$, $b$ is on
  $\alpha$, $b$ is on $\beta$, $a \neq b$

\item Let $a$ be the point of intersection of $\alpha$ and $\beta$,
  on the same side of $L$ as $b$, where $L$ is the line through their
  centers, $c$ and $d$, respectively. \\
  Prerequisites: $\alpha$ and $\beta$ intersect, $c$ is the center of
  $\alpha$, $d$ is the center of $\beta$, $c$ is on $L$, $d$ is on
  $L$, $b$ is not on $L$ \\
  Conclusion: $a$ is on $\alpha$, $a$ is on $\beta$, $a$ and $b$ are
  on the same side of $L$

\item Let $a$ be the point of intersection of $\alpha$ and $\beta$,
  on the side of $L$ opposite $b$, where $L$ is the line through their
  centers, $c$ and $d$, respectively. \\
  Prerequisite: $\alpha$ and $\beta$ intersect, $c$ is the center of
  $\alpha$, $d$ is the center of $\beta$, $c$ is on $L$, $d$ is on
  $L$, $b$ is not on $L$ \\
  Conclusion: $a$ is on $\alpha$, $a$ is on $\beta$, $a$ and $b$ are
  not on the same side of $L$, $a$ is not on $L$.

\end{enumerate}

\showdiagram{
\begin{figure}
\large
\begin{center}
\psfrag{c}{$\alpha$}\psfrag{d}{$\beta$}\psfrag{p}{$a$}\psfrag{q}{$b$}\psfrag{r}{$c$}\psfrag{s}{$d$}
\includegraphics[height=2.5cm]{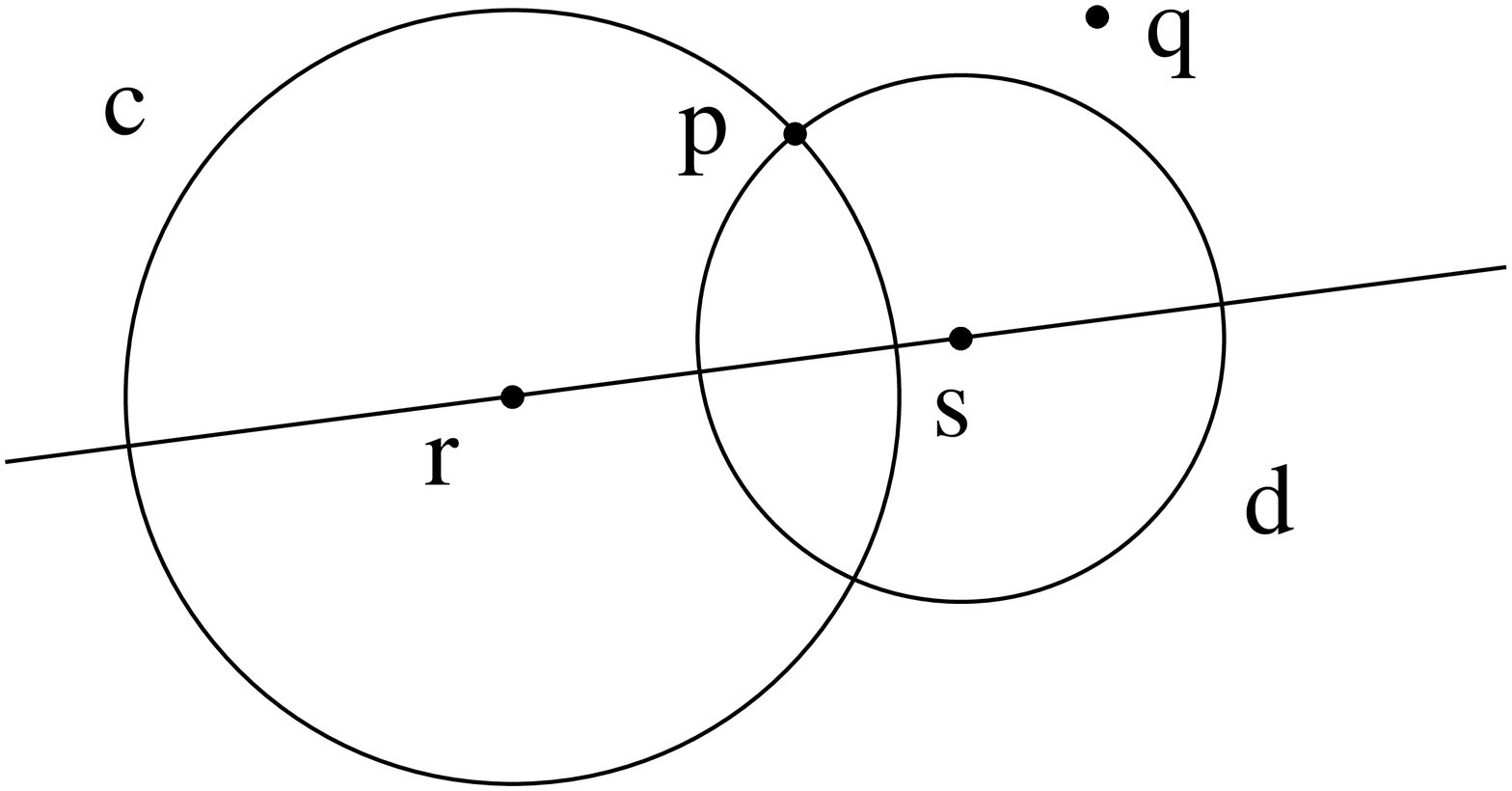}
\hspace*{0.3in}
\psfrag{c}{$\alpha$}\psfrag{d}{$\beta$}\psfrag{p}{$a$}\psfrag{q}{$b$}\psfrag{r}{$c$}\psfrag{s}{$d$}
\includegraphics[height=2.5cm]{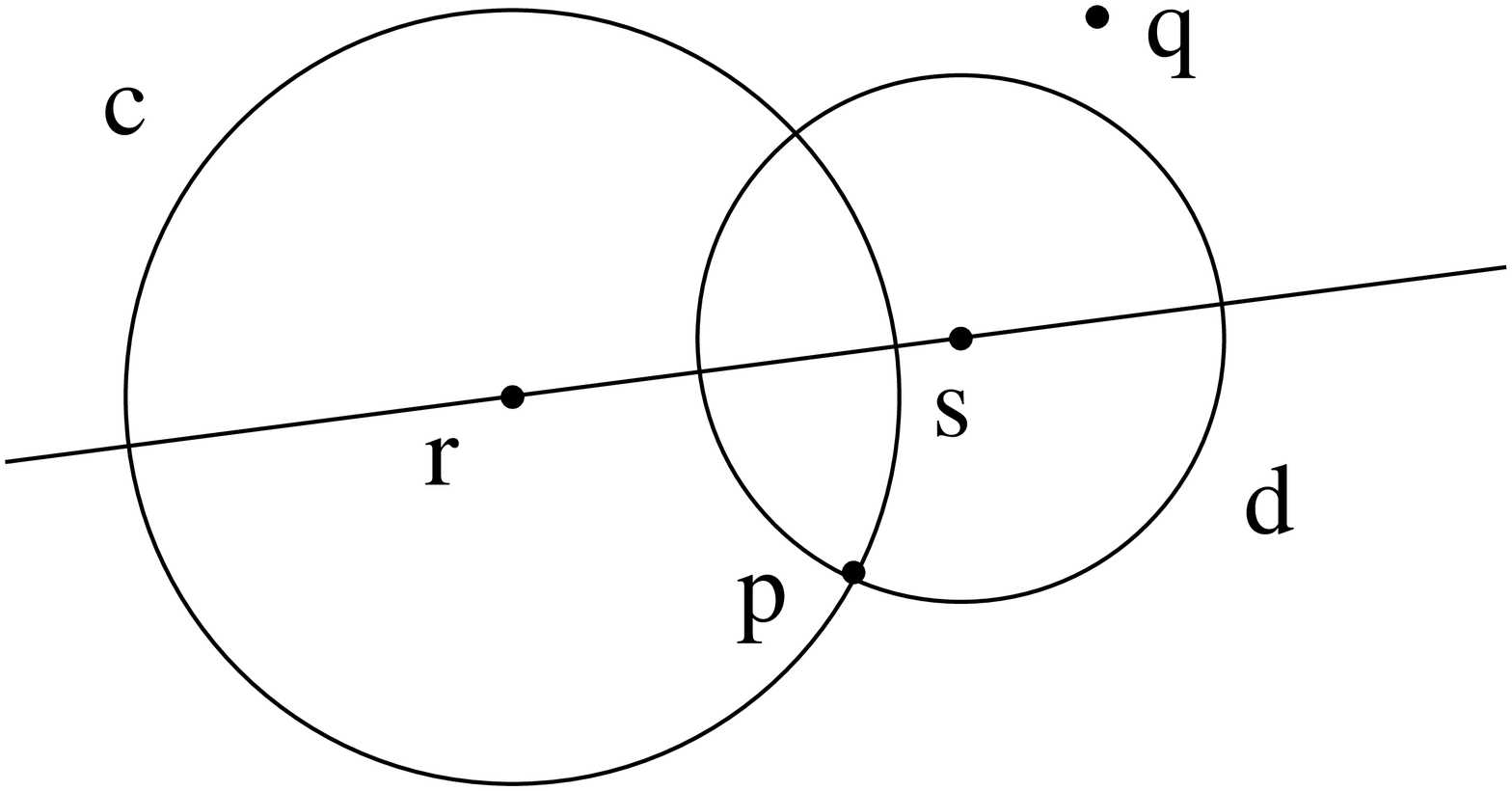}
\end{center}
\caption{Diagrams for intersection rules 8 (left) and 9 (right). In
  the first, the added intersection point $a$ is on the same side of
  $L$ as $b$; in the second, it is opposite $b$.}
\end{figure}
}

We close this section by noting that there is some redundancy in our
construction rules. For example, the circle intersection rules 8 and
9, which are somewhat complex, could be derived as \emph{theorems}
from the more basic rules. As we will see below, we have added them to
model particular construction steps in the \emph{Elements}. But there
are other constructions that can be derived in our system, that seem
no less obvious; for example, if $M$ and $N$ are distinct lines that
intersect, and $a$ is not on $N$, then one can pick a point $b$ on $M$
on the same side of $N$ as $a$. We did not include this rule only
because we did not find it in Euclid, though constructions like this
come up in our completeness proof, in
Section~\ref{completeness:section}.

This situation is somewhat unsatisfying. Our list of construction
rules was designed with two goals in mind: first, to model the
constructions in Euclid, and, second, to provide a system that is
complete, in the sense described in
Section~\ref{completeness:section}. But a smaller set of rules would
have met the second constraint, and since the constructions appearing
in Books I to IV of the \emph{Elements} constitute a finite list, the
first constraint could be met by brute-force enumeration.  What is
missing is a principled determination of what should constitute an
``obvious'' construction, as opposed to an existence assertion that
requires explicit proof.

We did, at one point, consider allowing the prover to introduce any
point satisfying constraints that are consistent with the current
diagram. Even for diagrams without circles, however, determining
whether such a list of constraints meets this criterion seems to be a
knotty combinatorial problem. And since circles can encode metric
information, in that case the proposal seems to allow users to do
things that are far from obvious. In any event, it is not clear that
this proposal comes closer to characterizing what we should take as
``obvious constructions.'' We therefore leave this task as an open
conceptual problem, maintaining only that the list of constructions we
have chosen here are (1) obviously sound, in an informal sense; (2)
sufficient to emulate the proofs in Books I to IV of the
\emph{Elements}; (3) sound for the intended semantics; and (4)
sufficient to make the system complete.

\subsection{Diagrammatic inferences}
\label{diagram:section}

We now provide a list of axioms that allow us to infer diagrammatic
assertions from the diagrammatic information available in a given
context in a proof. For the moment, these can be read as first-order
axioms; the precise sense in which they can be used to license
inferences in $\na{E}$ is spelled out in Section~\ref{direct:section}.

\bigskip

\noindent {\bf Generalities}

\nopagebreak

\begin{enumerate}
\item If $a \neq b$, $a$ is on $L$, and $b$
  is on $L$, $a$ is on $M$ and $b$ is on $M$, then $L = M$.
\item If $a$ and $b$ are both centers of $\alpha$ then $a = b$.
\item If $a$ is the center of $\alpha$ then $a$ is inside $\alpha$.
\item If $a$ is inside $\alpha$, then $a$ is not on $\alpha$. 
\end{enumerate}

\noindent The first axiom above says that two points determine a line.
It is logically equivalent to the assertion that the intersection of
two distinct lines, $L$ and $M$, is unique. The next two axioms tell
us that the center of a circle is unique, and inside the circle. The
final axiom then rules out ``degenerate'' circles.

\bigskip

\noindent {\bf Between axioms}

\nopagebreak

\begin{enumerate}
\item If $b$ is between $a$ and $c$ then $b$ is between $c$ and $a$,
  $a \neq b$, $a \neq c$, and $a$ is not between $b$ and $c$.
\item If $b$ is between $a$ and $c$, $a$ is on $L$, and $b$ is on $L$, then
  $c$ is on $L$.
\item If $b$ is between $a$ and $c$, $a$ is on $L$, and $c$ is on $L$,
  then $b$ is on $L$.
\item If $b$ is between $a$ and $c$ and $d$ is between $a$ and $b$
  then $d$ is between $a$ and $c$. 
\item If $b$ is between $a$ and $c$ and $c$ is between $b$ and $d$
  then $b$ is between $a$ and $d$.
\item If $a$, $b$, and $c$ are distinct points on a line $L$, then
  then either $b$ is between $a$ and $c$, or $a$ is between $b$ and
  $c$, or $c$ is between $a$ and $b$.
\item If $b$ is between $a$ and $c$ and $b$ is between $a$ and $d$
  then $b$ is not between $c$ and $d$. 
\end{enumerate}

Axioms 1, 4, 5, and 6 are essentially the axioms for ``between'' given
in Krantz et al.~\cite{Kran71}, with the minor difference that we are
axiomatizing a ``strict'' notion of betweenness instead of a nonstrict
one. Krantz et al.~show that a countable set satisfies these axioms if
and only if it can be embedded as a set of points on the real line. We
remark, in passing, that it would be interesting to have similar
completeness or representation theorems for other groups of the axioms
found here. Our approach has been syntactic rather than semantic,
which is to say, our goal has been to capture certain deductive
relationships rather than to characterize classes of structures; but
it would be illuminating to understand the extent to which our various
groups of axioms give rise to natural classes of structures.

The last axiom is illustrated by the following diagram:
\showdiagram{
\begin{center}
\psfrag{p}{$a$}\psfrag{q}{$b$}\psfrag{r}{$c$}\psfrag{s}{$d$}
\includegraphics[height=0.375cm]{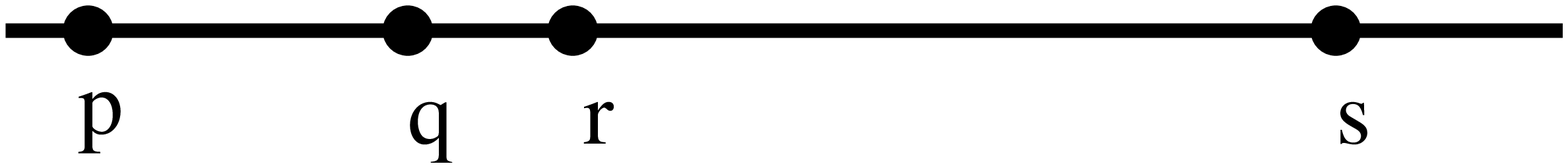}
\end{center}
} The axiom states that if $d$ and $c$ are on the same side of $b$
along a line, then $b$ does not fall between them. This axiom is, in
fact, a first-order consequence of the others; it is therefore only
useful in contexts where we consider more restrictive notions of
consequence, as we do in Section~\ref{direct:section}.

\bigskip

\noindent {\bf Same side axioms}

\begin{enumerate}
\item If $a$ is not on $L$, then $a$ and $a$ are on the same side of $L$.
\item If $a$ and $b$ are on the same side of $L$, then $b$ and $a$ are
  on the same side of $L$.
\item If $a$ and $b$ are on the same side of $L$, then $a$ is not on
  $L$. 
\item If $a$ and $b$ are on the same side of $L$, and $a$ and $c$ are
  on the same side of $L$, then $b$ and $c$ are on the same side of
  $L$. 
\item If $a$, $b$, and $c$ are not on $L$, and $a$ and $b$ are not on
  the same side of $L$, then either $a$ and $c$ are on the same side
  of $L$, or $b$ and $c$ are on the same side of $L$.
\end{enumerate}

\noindent If $L$ is a line, the axioms imply that the relation
``falling on the same side of $L$'' is an equivalence relation; and
any point $a$ not on $L$ serves to partition the points into three
classes, namely, those on $L$, those on the same side of $L$ as $a$,
and those on the opposite side of $L$ from $a$.

With the interpretation of $\diffside(p,q,L)$ described in
Section~\ref{language:section}, the axioms imply that if $a$ and $b$
are on different sides of $L$ and $a$ and $c$ are on different sides
of $L$, then $b$ and $c$ are on the same side of $L$. The axioms also
imply that if $a$ and $b$ are on the same side of $L$ and $a$ and $c$
are on different sides of $L$ then $b$ and $c$ are on different sides
of $L$.

\bigskip

\noindent {\bf Pasch axioms}

\nopagebreak

\begin{enumerate}
\item If $b$ is between $a$ and $c$ and $a$ and $c$ are on the same side of
  $L$, then $a$ and $b$ are on the same side of $L$.
\item If $b$ is between $a$ and $c$ and $a$ is on $L$ and $b$ is not
  on $L$, then $b$ and $c$ are on the same side of $L$.
\item If $b$ is between $a$ and $c$ and $b$ is on $L$ then $a$ and $c$
  are not on the same side of $L$.
\item If $b$ is the intersection of distinct lines $L$ and $M$, $a$
  and $c$ are distinct points on $M$, $a \neq b$, $c \neq b$, and $a$
  and $c$ are not on the same side of $L$, then $b$ is between $a$ and
  $c$.
\end{enumerate}

\showdiagram{
\begin{figure}

\begin{center}
\psfrag{p}{$a$}\psfrag{q}{$b$}\psfrag{r}{$c$}\psfrag{s}{$d$}\psfrag{L}{$L$}\psfrag{M}{$M$}\psfrag{N}{$N$}\psfrag{t}{$e$}
\includegraphics[height=2cm]{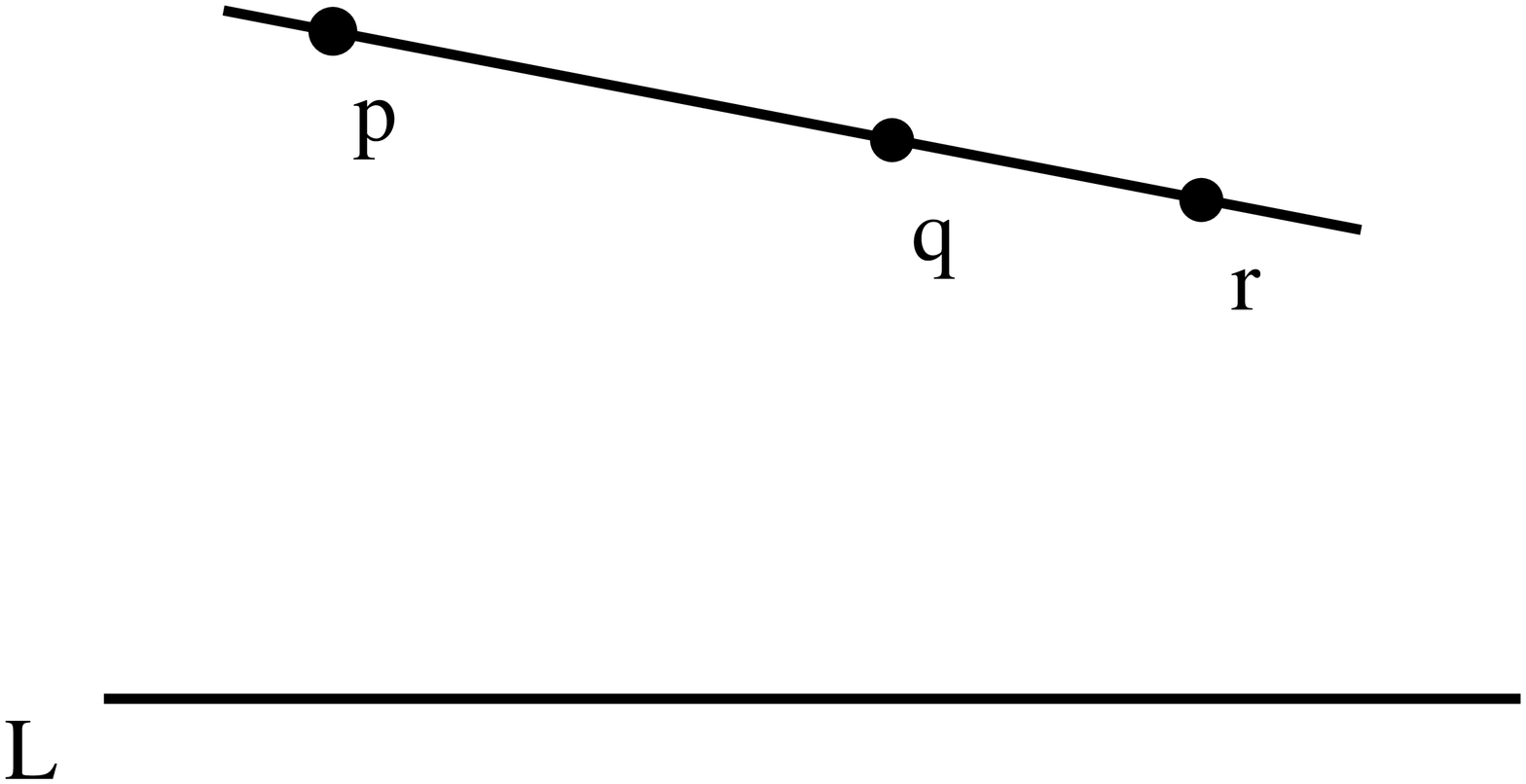}
\hspace*{0.2in}
\psfrag{p}{$a$}\psfrag{q}{$b$}\psfrag{r}{$c$}\psfrag{s}{$d$}\psfrag{L}{$L$}\psfrag{M}{$M$}\psfrag{N}{$N$}\psfrag{t}{$e$}
\includegraphics[height=2cm]{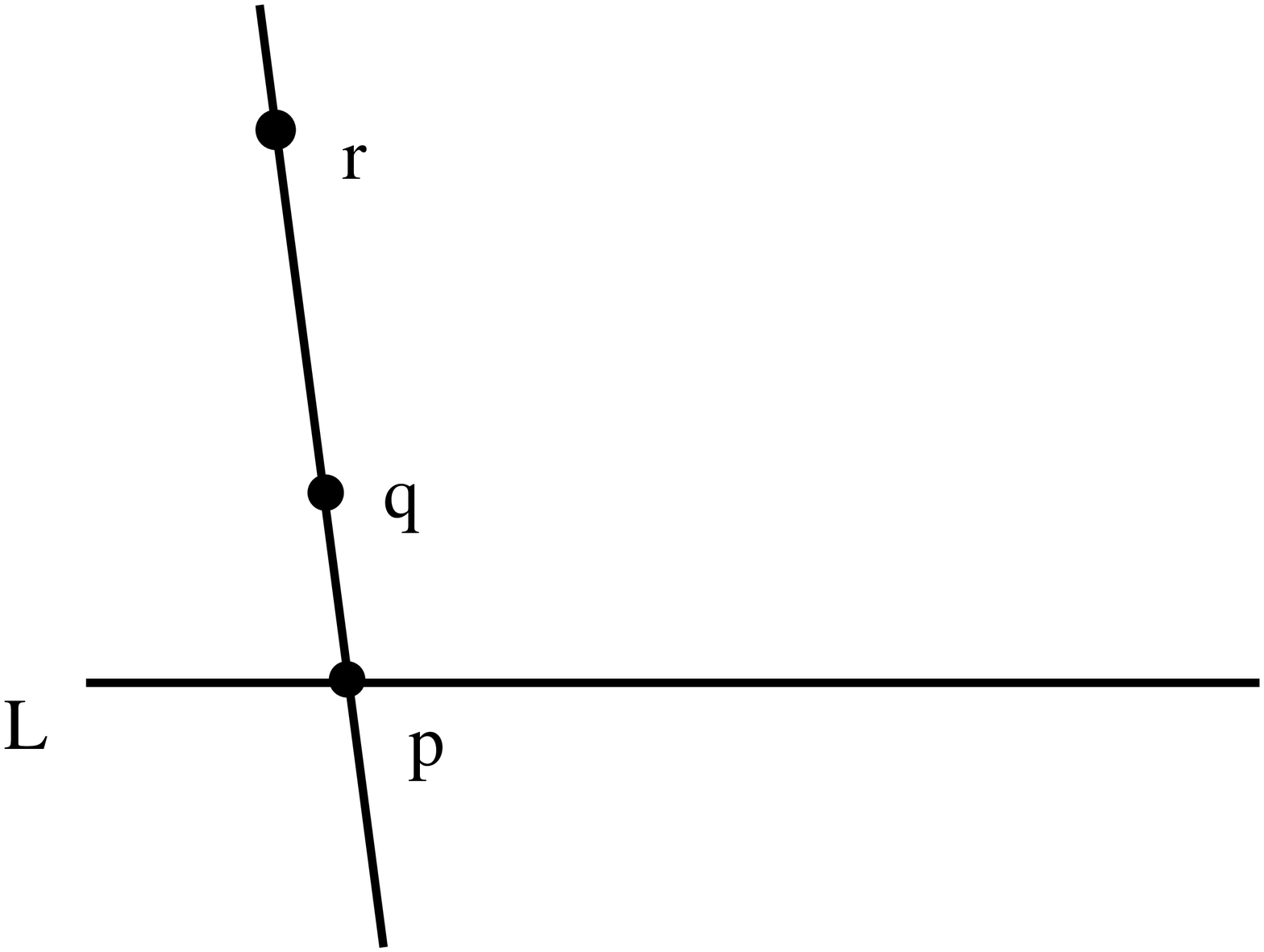}
\hspace*{0.2in}
\psfrag{p}{$a$}\psfrag{q}{$b$}\psfrag{r}{$c$}\psfrag{s}{$d$}\psfrag{L}{$L$}\psfrag{M}{$M$}\psfrag{N}{$N$}\psfrag{t}{$e$}
\includegraphics[height=2cm]{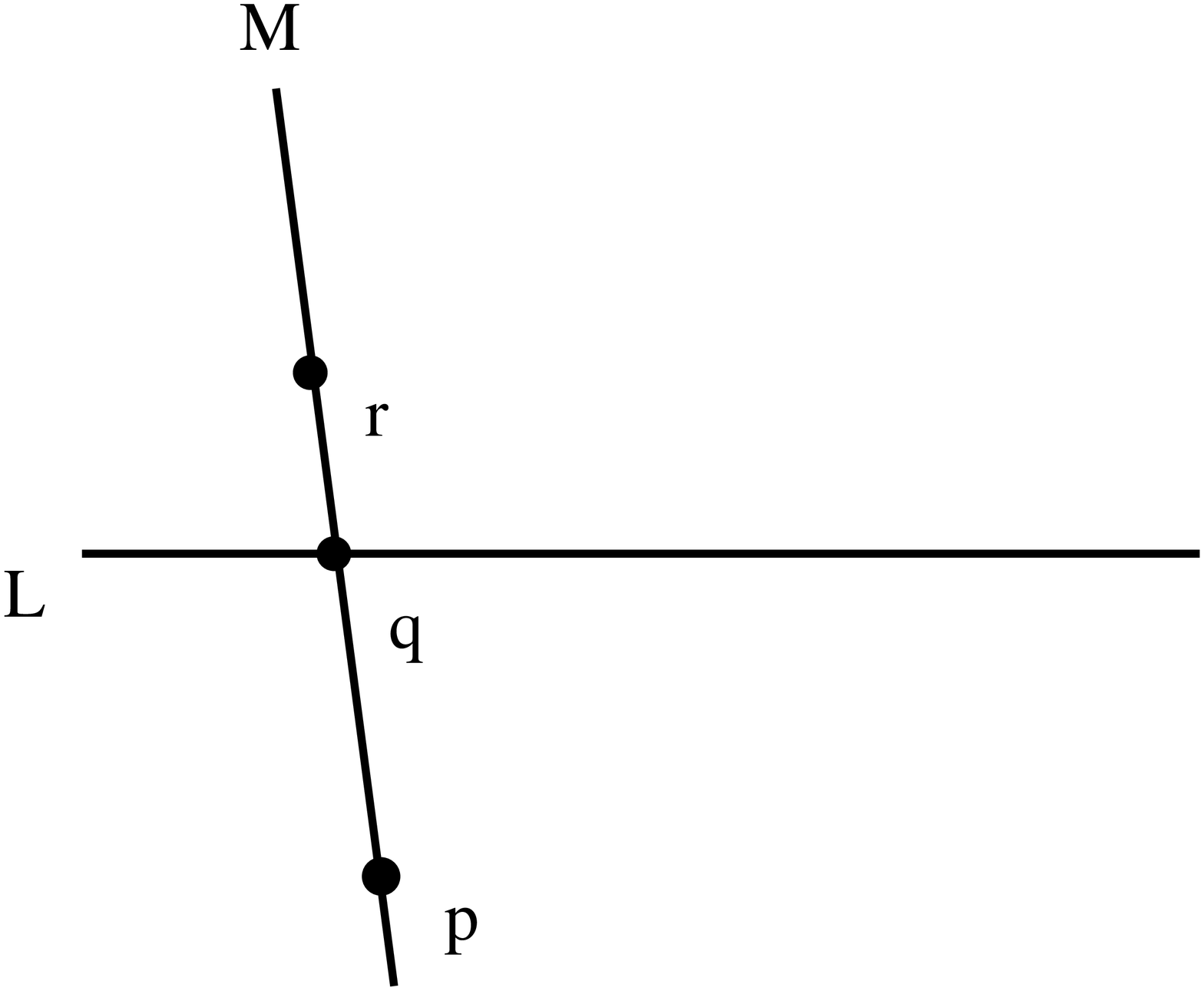}
\end{center}

\caption{Pasch rules 1 (left), 2 (center), and 3 and 4 
  (right).}
\end{figure}
}

These axioms serve to relate the ``between'' relation and the ``same
side'' relation. In the fourth axiom, ``$b$ is the intersection of
distinct lines $L$ and $M$'' should be understood as ``$L \neq M$, $b$
is on $L$, and $b$ is on $M$.''

\begin{figure}
\psfrag{a}{$a$}
\psfrag{b}{$b$}
\psfrag{c}{$c$}
\psfrag{L}{$L$}
\begin{center}
\includegraphics[height=2cm]{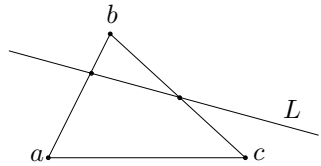}
\end{center}
\caption{the Pasch axiom}
\label{pasch:figure}
\end{figure}

In the literature, the phrase ``Pasch axiom'' is used to refer to the
assertion that a line passing through one side of a triangle
necessarily passes through one of the other two sides, or their point of
intersection (see Figure~\ref{pasch:figure}). This axiom was
indeed used by Pasch \cite{pasch:82}, and later by Hilbert
\cite{hilbert:99}, with attribution. Theorems of $\na{E}$ do not allow
disjunctive conclusions, but one can use the conclusion of Pasch's
theorem to reason disjunctively in a proof: in
Figure~\ref{pasch:figure}, either $c$ is on $L$, or on the same side
of $L$ as $a$, or on the same side of $L$ as $b$. In the second case,
where $a$ and $c$ are on the same side of $L$, our third Pasch axiom
(together with the same-side axioms) imply that $b$ and $c$ are on
opposite sides of $L$. The intersection rules below then tell us that
the line through $b$ and $c$ intersects $L$.  Our fourth Pasch axiom
then implies that this intersection is between $b$ and $c$. The third
case is handled in a similar way. We have therefore chosen the name
for this group of axioms to indicate that they provide an analysis of
the usual Pasch axiom into more basic diagrammatic rules.

\bigskip

\noindent {\bf Triple incidence axioms}

\nopagebreak

\begin{enumerate}
\item If $L$, $M$, and $N$ are lines meeting at a point $a$, and $b$,
  $c$, and $d$ are points on $L$, $M$, and $N$ respectively, and if $c$ and
  $d$ are on the same side of $L$, and $b$ and $c$ are on the same
  side of $N$, then $b$ and $d$ are not on the same side of $M$.
\item If $L$, $M$, and $N$ are lines meeting at a point $a$, and $b$,
  $c$, and $d$ are points on $L$, $M$, and $N$ respectively, and if
  $c$ and $d$ are on the same side of $L$, and $b$ and $d$ are not on
  the same side of $M$, and $d$ is not on $M$ and $b \neq a$, then $b$
  and $c$ are on the same side of $N$.
\item If $L$, $M$, and $N$ are lines meeting at a point $a$, and $b$,
  $c$, and $d$ are points on $L$, $M$, and $N$ respectively, and if $c$ and
  $d$ are on the same side of $L$, and $b$ and $c$ are on the same
  side of $N$, and $d$ and $e$ are on the same side of $M$, and $c$
  and $e$ are on the same side of $N$, then $c$ and $e$ are on the
  same side of $L$.
\end{enumerate}

\showdiagram{
\begin{figure}
\label{triple:figure}

\begin{center}
\psfrag{p}{$a$}\psfrag{q}{$b$}\psfrag{r}{$c$}\psfrag{s}{$d$}\psfrag{L}{$L$}\psfrag{M}{$M$}\psfrag{N}{$N$}\psfrag{t}{$e$}
\includegraphics[height=2.5cm]{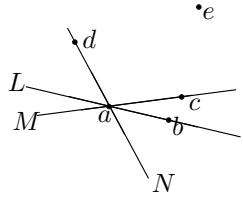}
\end{center}
\caption{Triple incidence rules. (The same diagram illustrates all
  three rules.)}
\end{figure}
}

\noindent These axioms explain how three lines
intersecting in a point divide space into regions (see
diagram~\ref{triple:figure}). 

\bigskip

\noindent {\bf Circle axioms}

\begin{enumerate}
\item If $a$, $b$, and $c$ are on $L$, $a$ is inside $\alpha$, $b$
  and $c$ are on $\alpha$, and $b \neq c$, then
  $a$ is between $b$ and $c$.
\item If $a$ and $b$ are each inside $\alpha$ or on $\alpha$, and $c$
  is between $a$ and $b$, then $c$ is inside $\alpha$.
\item If $a$ is inside $\alpha$ or on $\alpha$, $c$ is not inside
  $\alpha$, and $c$ is between $a$ and $b$, then $b$ is neither inside
  $\alpha$ nor on $\alpha$.
\item Let $\alpha$ and $\beta$ be distinct circles that intersect in 
  distinct
  points $c$ and $d$. Let $a$ be a the center of $\alpha$, let $b$ be
  the center of $\beta$, and let $L$ be the line through $a$ and $b$. 
  Then $c$ and $d$ are not on the same side of $L$.
\end{enumerate}

\showdiagram{
\begin{figure}

\begin{center}
\psfrag{c}{$\alpha$}\psfrag{d}{$\beta$}\psfrag{p}{$a$}\psfrag{q}{$b$}\psfrag{r}{$c$}\psfrag{s}{$d$}
\includegraphics[height=1.5cm]{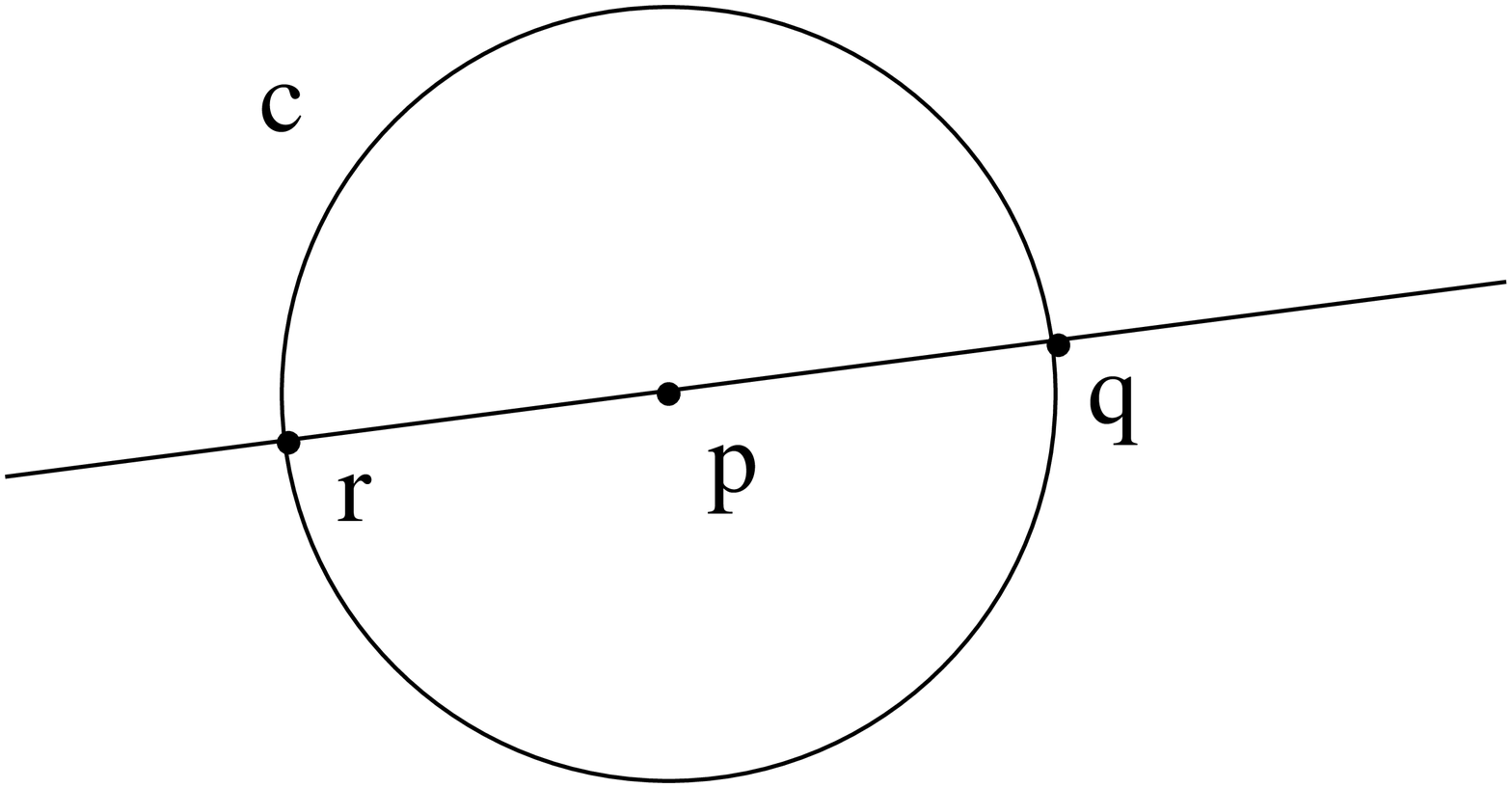}
\psfrag{c}{$\alpha$}\psfrag{d}{$\beta$}\psfrag{p}{$a$}\psfrag{q}{$b$}\psfrag{r}{$c$}\psfrag{s}{$d$}
\includegraphics[height=1.5cm]{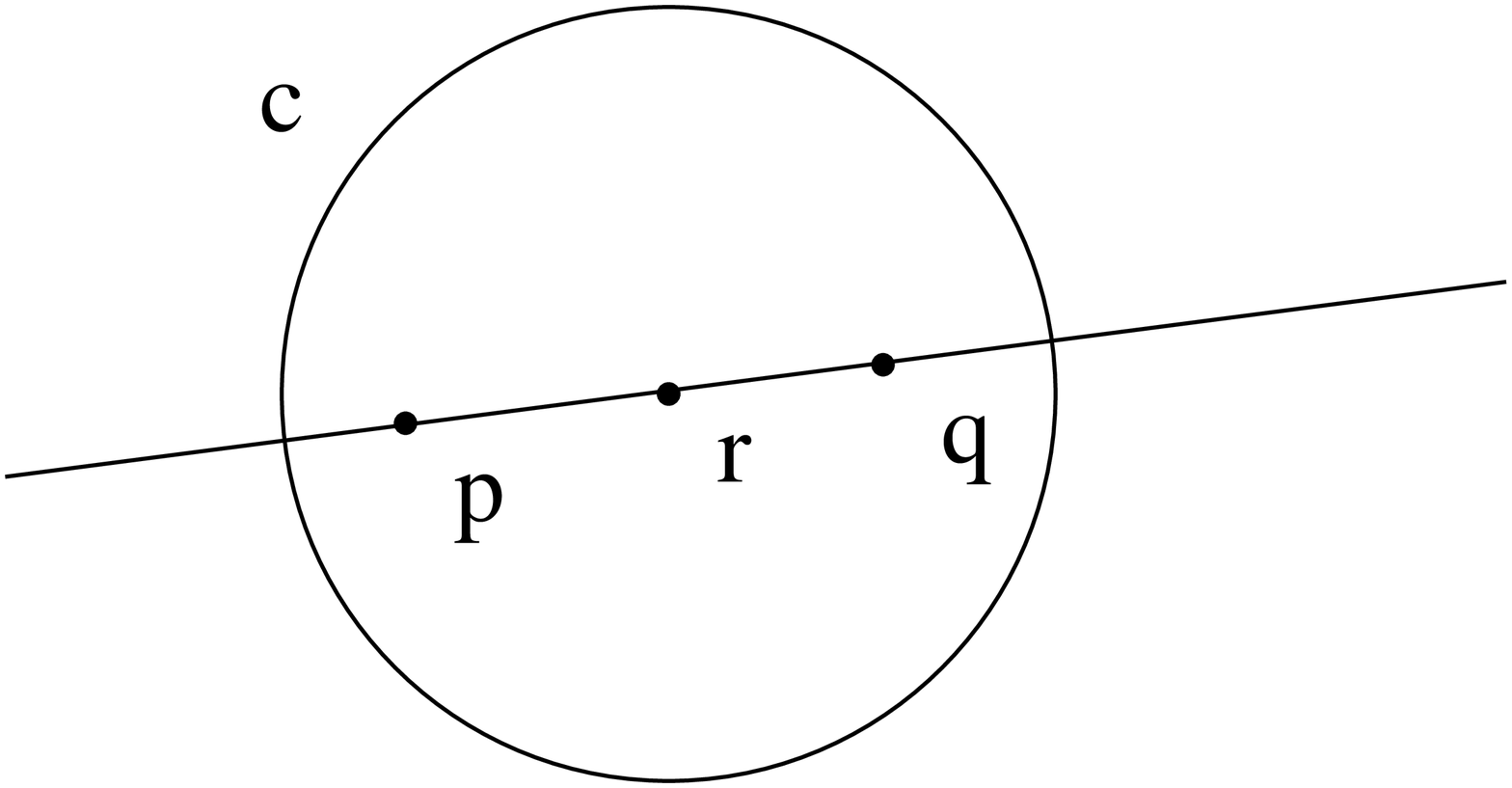}
\psfrag{c}{$\alpha$}\psfrag{d}{$\beta$}\psfrag{p}{$a$}\psfrag{q}{$b$}\psfrag{r}{$c$}\psfrag{s}{$d$}
\includegraphics[height=1.5cm]{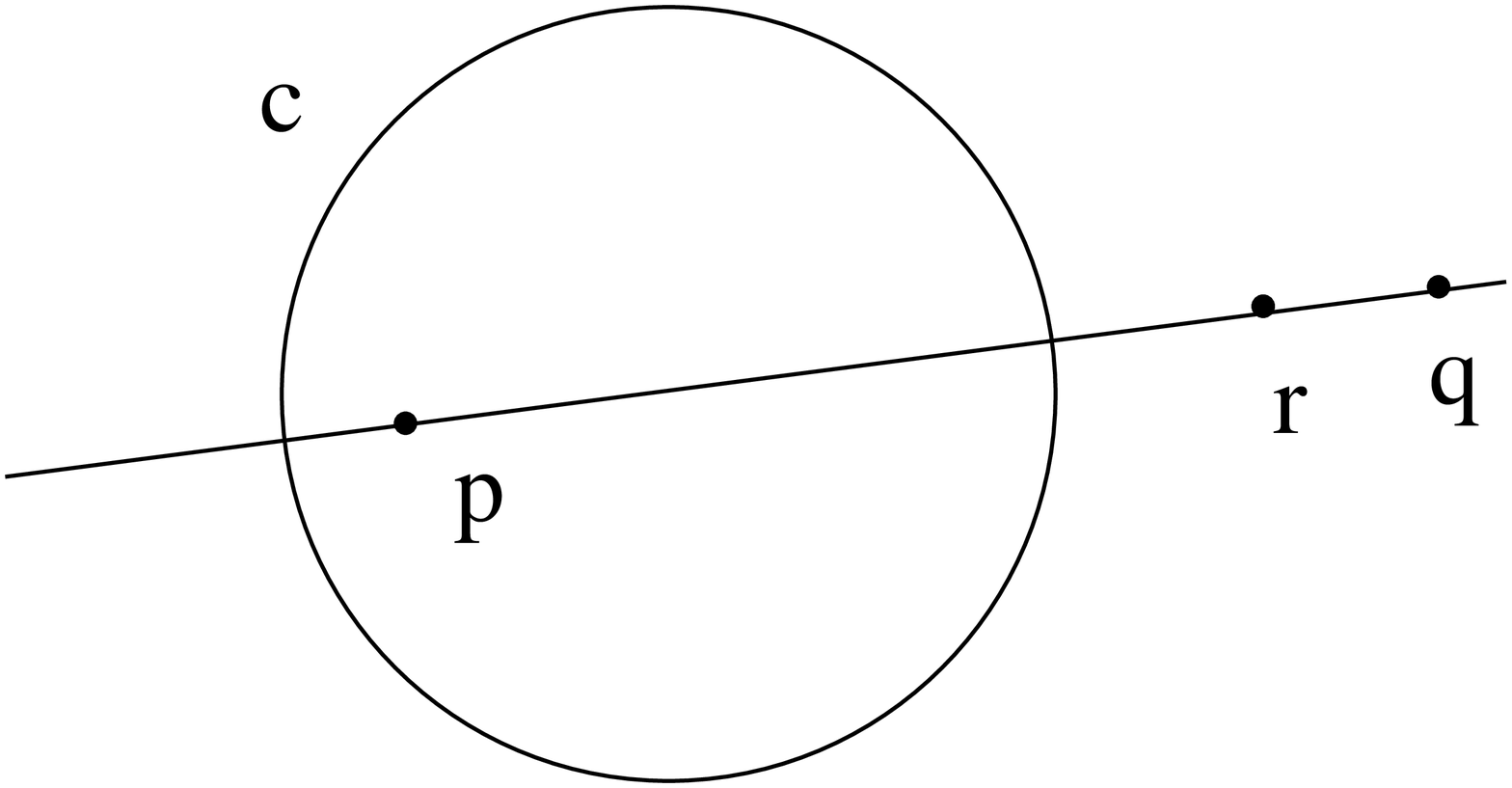}
\psfrag{c}{$\alpha$}\psfrag{d}{$\beta$}\psfrag{p}{$a$}\psfrag{q}{$b$}\psfrag{r}{$c$}\psfrag{s}{$d$}
\includegraphics[height=1.5cm]{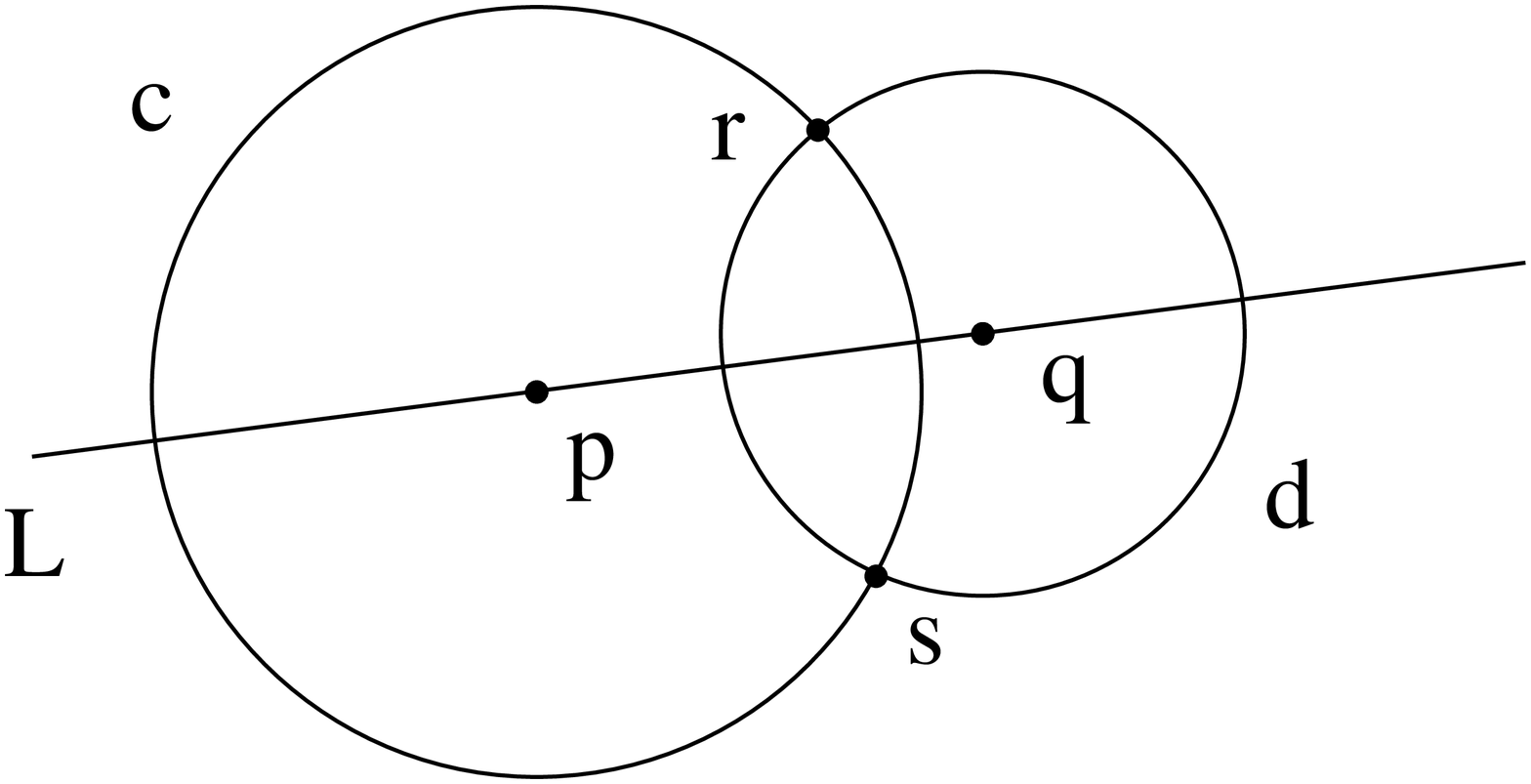}
\end{center}
\caption{Circle axioms 1--4.}
\end{figure}
}

\noindent {\bf Intersection rules}

\nopagebreak

\begin{enumerate}
\item If $a$ and $b$ are on different sides of $L$, and $M$ is the line
  through $a$ and $b$, then $L$ and $M$ intersect.
\item If $a$ is on or inside $\alpha$, $b$ is on or inside $\alpha$,
  and $a$ and $b$ are on different sides of $L$, then $L$ and $\alpha$
  intersect. 
\item If $a$ is inside $\alpha$ and on $L$, then $L$ and $\alpha$
  intersect.
\item If $a$ is on or inside $\alpha$, $b$ is on or inside $\alpha$,
  $a$ is inside $\beta$, and $b$ is outside $\beta$, then $\alpha$ and
  $\beta$ intersect.
\item If $a$ is on $\alpha$, $b$ is in $\alpha$, $a$ is in $\beta$,
  and $b$ is on $\beta$, then $\alpha$ and $\beta$ intersect.
\end{enumerate}

\noindent Recall that ``intersection'' means transversal intersection.
The first axiom says that a line passing from one side of $L$ to the
other intersects it. The second axiom says that if $\alpha$ is a
circle that straddles $L$, then $\alpha$ intersects $L$. The third
axiom says that a line that passes through a circle intersects it. The
fourth and fifth axioms are the analogous properties for circles. The
third axiom can be viewed as the assertion that a line cannot be
bounded by a circle; the others can be viewed as continuity
principles.

\bigskip

\noindent {\bf Equality axioms}

\nopagebreak

\begin{enumerate}
\item $x = x$
\item If $x = y$ and $\ph(x)$, then $\ph(y)$
\end{enumerate}

\noindent Here $x$ and $y$ can range over any of the sorts (that is,
there is an equality symbol for each sort) and $\ph$ can be any atomic
formula. These are the usual equality axioms for first-order logic,
and so may be taken to be subsumed under the notion of ``first-order
consequence.''

\subsection{Metric inferences}
\label{metric:section}

Consider the structure $\la \RR^{+}, 0, +, < \ra$, that is, the
nonnegative real numbers with zero, addition, and the less-than
relation. It is well known that the theory of this structure is
decidable. The set of universal consequences of this theory (or,
equivalently, the set of quantifier-free formulas that are true of the
structure under any assignment to the free variables) can be
axiomatized as follows:
\begin{itemize}
\item $+$ is associative and commutative, with identity $0$.
\item $<$ is a linear ordering with least element $0$.
\item For any $x$, $y$, and $z$, if $x < y$ then $x + z < y + z$.
\end{itemize}
Equivalently, these axioms describe the nonnegative part of any
linearly ordered abelian group. Happily, these are the general
properties Euclid assumes of magnitudes, that is, the segment lengths,
angles, and areas in our formalization (see Stein
\cite[p.~167]{stein:90}). To be more precise, Euclid seems to assume
that his magnitudes are strictly positive. But we have already noted
in Section~\ref{language:section} that we simply include $0$ for
convenience; we could just as well have axiomatized the strictly
positive reals. The axioms above imply that if $x + z = y + z$, then
$x = z$, which corresponds to Euclid's common notion 3, ``if equals be
subtracted from equals, the remainders are equal.'' The third axiom
implies that if $0 < y$, then $z < y + z$, which corresponds to common
notion 5, ``the whole is greater than the part.''

In addition to these, we include the following axioms, which Euclid
seems to take to be clear from the definitions (modulo the caveat, in
the last paragraph, that we include $0$ as a magnitude):
\begin{enumerate}
\item $\seg{ab} = 0$ if and only if $a = b$.
\item $\seg{ab} \geq 0$
\item $\seg{ab} = \seg{ba}$.
\item $a \neq b$ and $a \neq c$ imply $\angle abc = \angle cba$.
\item $0 \leq \angle abc$ and $\angle abc  \leq \rightangle + \rightangle$.
\item $\area{aab} = 0$.
\item $\area{abc} \geq 0$.
\item $\area{abc} = \area{cab}$ and $\area{abc} = \area{acb}$.
\item If $\seg{ab} = \seg{a'b'}$, $\seg{bc} = \seg{b'c'}$, $\seg{ca} =
  \seg{c'a'}$, $\angle abc = \angle a'b'c'$, $\angle bca = \angle
  b'c'a'$, and $\angle cab = \angle c'a'b'$, then $\area{abc} =
  \area{a'b'c'}$. 
\end{enumerate}
Note that we do not ascribe any meaning to the magnitude $\angle{abc}$
when $b = a$ or $b = c$. As axiom 6 indicates,
however, we take ``degenerate'' triangles to have area $0$. Once
Euclid has proved two triangles congruent (that is, once he has shown
that all their parts are equal), he uses the fact that they have the
same area, without comment. The last axiom simply makes this explicit.

Of course, there are further properties involving magnitudes that can
be read off from a diagram, and, conversely, metric considerations can
imply diagrammatic facts. These ``transfer inferences'' are the
subject of the next section.

\subsection{Transfer inferences}
\label{transfer:section}

We divide the transfer inferences into three groups, depending on
whether they involve segment lengths, angles, or areas.

\bigskip

\noindent {\bf Diagram-segment transfer axioms}

\nopagebreak

\begin{enumerate}
\item If $b$ is between $a$ and $c$, then $\seg{ab} + \seg{bc} = \seg{ac}$.
\item If $a$ is the center of $\alpha$ and $\beta$, $b$ is on
  $\alpha$, $c$ is on $\beta$, and $\seg{ab} = \seg{ac}$, then $\alpha
  = \beta$.
\item If $a$ is the center of $\alpha$ and $b$ is on $\alpha$, then
  $\seg{ac} = \seg{ab}$ if and only if $c$ is on $\alpha$.
\item If $a$ is the center of $\alpha$ and $b$ is on $\alpha$, and
  $\seg{ac} < \seg{ab}$ if and only if $c$ is in $\alpha$.
\end{enumerate}

\noindent The second axiom implies that a circle is determined by its
center and radius. In the discussion in
Section~\ref{departures:section}, we will explain that this is a mild
departure from Euclid's treatment of circles. (Euclid seems to rely on
a construction rule which has the same net effect.) When $\alpha =
\beta$, this axiom implies the converse direction of the equivalence
in axiom 3 (so that axiom could be stated instead as an implication).

\bigskip

\noindent {\bf Diagram-angle transfer axioms}

\begin{enumerate}
\item Suppose $a \neq b$, $a \neq c$, $a$ is on $L$, and $b$ is on
  $L$.  Then $c$ is on $L$ and $a$ is not between $b$ and $c$ if and
  only if $\angle bac = 0$.
\item Suppose $a$ is on $L$ and $M$, $b$ is on $L$, $c$ is on $M$, $a \neq
  q$, $a \neq c$, $d$ is not on $L$ or $M$, and $L \neq M$. Then 
  $\angle bac = \angle bad + \angle dac$ if and only if $b$ and $d$
  are on the same side of $M$ and $c$ and $d$ are on the same side of $L$.
\item Suppose $a$ and $b$ are points on $L$, $c$ is between $a$ and $b$,
  and $d$ is not on $L$. Then $\angle acd = \angle dcb$ if and only if
  $\angle acd$ is equal to $\rightangle$.
\item Suppose $a$, $b$, and $b'$ are on $L$, $a$, $c$, and $c'$ are on
  $M$, $b \neq a$, $b'\neq a$, $c \neq a$, $c' \neq a$, $a$ is not
  between $b$ and $b'$, and $a$ is not between $c$ and $c'$. Then
  $\angle bac = \angle b'ac'$. 
\item Suppose $a$ and $b$ are on $L$, $b$ and $c$ are on $M$, and $c$
  and $d$ are on $N$. Suppose also that $b \neq c$, $a$ and $d$ are on
  the same side of $N$, and $\angle abc + \angle bcd < \rightangle +
  \rightangle$. Then $L$ and $N$ intersect, and if $e$ is on $L$ and
  $N$, then $e$ and $a$ are on the same side of $M$.
\end{enumerate}

\showdiagram{
\begin{figure}

\begin{center}
\psfrag{L}{$L$}\psfrag{a}{$a$}\psfrag{b}{$b$}\psfrag{c}{$c$}\psfrag{d}{$d$}\psfrag{N}{$N$}\psfrag{M}{$M$}
\includegraphics[height=2cm]{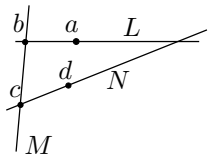}
\end{center}
\caption{diagram-angle transfer axiom 5.}
\end{figure}
}

\noindent The first axiom says that if $a$ and $b$ are distinct points
on a line $L$, then a point $c$ is on $L$ on the same side of $a$ as
$b$ if and only if $\angle bac = 0$. The right-hand side of the
equivalence in the second axiom can be read more simply as the
assertion that $d$ lies inside the angle $bac$. Thus the axiom
implies that angles sum in the expected way. The third axiom
corresponds to Euclid's definition 10, ``when a straight line set up
on a straight line makes the adjacent angles equal to one another,
each of the equal angles is called \emph{right}\ldots.'' It also, at
the same time, codifies postulate 4, ``all right angles are equal to
one another,'' using the constant, ``$\rightangle$,'' to describe the
magnitude that all right angles are equal to. The fourth axiom says
that different descriptions of the same angle are equal; more
precisely, if $ab$ and $ab'$ are the same ray, and likewise for $ac$
and $ac'$, then $abc$ and $ab'c$ are equal.

Euclid's wording may make it seem more natural to use a predicate to
assert that $abc$ forms a right angle, rather than using a constant,
``$\rightangle$,'' to denote an arbitrary right angle. But Euclid
seems to refer to an arbitrary right angle in his statement of this
parallel postulate, which, in the Heath translation, states:
\begin{quote}
That, if a straight line falling on two straight lines make the
interior angles on the same side less than two right angles, the two
straight lines, if produced indefinitely, meet on that side on which
are the angles less than the two right angles. \cite[p.~155]{euclid}
\end{quote}
Formulated in this way, a better name for the axiom might be the
``non-parallel postulate'': it asserts that if the diagram
configuration satisfies the given metric constraints on the angles,
then two of the lines are guaranteed to intersect. The postulate
translates to the last axiom above, which licenses the
construction ``let $e$ be the intersection of $L$ and $N$.''
Furthermore, assuming $e$ \emph{is} the intersection of $L$ and $N$,
the postulate specifies the side of $M$ on which $e$ lies.

\bigskip

\noindent {\bf Diagram-area transfer axioms}

\begin{enumerate}
\item If $a$ and $b$ are on $L$ and $a \neq b$, then $\area{abc} = 0$
  if and only if $c$ is on $L$.
\item If $a$, $b$, $c$ are on $L$ and distinct from one another, $d$
  is not on $L$, then $c$ is between $a$ and $b$ if and only if
  $\area{acd} + \area{dcb} = \area{adb}$.
\end{enumerate}

\noindent The second axiom implies that when a triangle is divided in
two, the areas sum in the expected way.

\subsection{Superposition}
\label{superposition:section}

We now come to the final two inferences in our system, Euclid's
notorious ``superposition inferences,'' which vexed commentators
through the ages (see the references in Section~\ref{proofs:section}).
Euclid's Proposition I.4 states the familiar ``side-angle-side''
property, namely that if two triangles $abc$ and $def$ are
such that $ab, ac$ are congruent to $de, df$ respectively, and $bac$
is congruent to angle $edf$, then the two triangles are congruent. The
proof proceeds by imagining $abc$ superimposed on $def$. In the Heath
translation:
\begin{quote}
For, if the triangle $abc$ be applied to the triangle $def$, and if
the point $a$ be placed on the point $d$ and the straight line $ab$ on
$de$, then the point $b$ will also coincide with $e$, because $ab$ is
equal to $de$.\ldots \cite[p.~247]{euclid}
\end{quote}
At issue is what it means to ``apply'' $abc$ to another triangle in
such a way. Euclid has not yet proved that one can \emph{construct} a
copy of $a'b'c'$ of $abc$ that will meet the given constraints. This
requires one to be able to copy a given angle, which is Euclid's
Proposition I.23.  The chain of reasoning leading to that proposition
includes Proposition I.4 as a component. The same issue arises in the
proof of Proposition I.8, which uses a superposition argument to
establish the ``side-side-side'' property.

\begin{figure}
\begin{center}
\psfrag{c}{$c$}\psfrag{d}{$d$}\psfrag{a}{$a$}\psfrag{b}{$b$}
\psfrag{c'}{$c'$}\psfrag{f}{$f$}\psfrag{a'}{$a'=d$}
\psfrag{b'}{$b'$}\psfrag{e}{$e$}\psfrag{g}{$g$}\psfrag{h}{$h$}
\includegraphics[height=2cm]{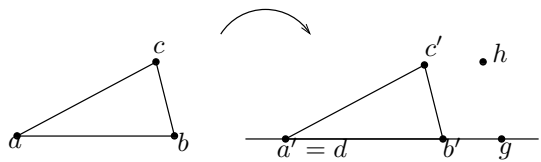}
\end{center}
\caption{superposition}
\end{figure}

How, then, shall we treat superposition? One possibility is simply to
add two new construction rules. The first would assert that given an
angle $abc$, a point $d$ on a line $L$, a point $g$ on $L$, and a
point $h$ not on $L$, one can construct points $a'$, $b'$, $c'$ such
that $a' = d$, $\angle a'b'c' = \angle abc$, $b'$ lies on $L$ in the
direction determined by $g$, and $c'$ lies on the same side of $L$ as
$h$. The second says that given a triangle $abc$, a point $d$ on a
line $L$, a point $g$ on $L$, and a point $h$ not on $L$, once can
find points $a'$, $b'$, $c'$ as above with $ab, bc, ca$ congruent to
$a'b', b'c', c'a'$, respectively.  These new construction rules would
certainly allow us to carry out the proofs of Propositions I.4 and
I.8, but the solution is not at all satisfying: Euclid takes great
pains to \emph{derive} the fact that one can carry out constructions
like these, using Propositions I.4 and I.8 along the way.

A second possibility is simply to add Propositions I.4 and I.8, the
SAS and SSS properties, as axioms. But, once again, this is not a
satisfactory solution, since it fails to explain why Euclid takes the
trouble to prove them.

Our formulation of $\na{E}$ provides a third, more elegant solution.
What superposition allows one to do is to act \emph{as though} one has
the result of doing the constructions above, but only for the sake of
proving things about objects that are already present in the diagram.
In proof-theoretic parlance, superposition is used as an
\emph{elimination} rule: if you can derive a conclusion assuming the
existence of some new objects, you can infer that the conclusion holds
without the additional assumption. In Euclid's case, one is barred,
however, from using the assumption to construct new objects.

This has a straightforward formulation as a sequent inference. Suppose
$\Gamma, \Delta$ includes assertions to the effect that $abc$ are
distinct and noncollinear, and $g$, $L$, and $h$ are as above. Let
$\Pi_1$ be the set
\[
\{ a' = d, \angle a'b'c' = \angle abc, \on(b',L), \lnot \mybetween(b',d,g),
\sameside(c',h,L) \}
\]
corresponding to the result of SAS superposition, and let $\Pi_2$
be the set
\[
\{ a' = d, \seg{ab} = \seg{a'b'}, \seg{bc} = \seg{b'c'}, \seg{ca} =
\seg{c'a'}, \on(b',L), \lnot \mybetween(b',d,g), \sameside(c',h,L) \}
\]
corresponding to the results of SSS superposition. Then the rules
can be expressed as
 \begin{prooftree}
\AXM{\Gamma \fCenter \ex{\vec x.} \Delta}
\AXM{\Gamma, \Delta, \Pi_i \fCenter \Delta'}
\BIM{\Gamma \fCenter \ex{\vec x.} \Delta, \Delta'}
\end{prooftree}
where $i$ is equal to 1, 2, respectively.

\subsection{The notion of a ``direct consequence''}
\label{direct:section} 

We have characterized ``the diagram'' in a Euclidean proof as the
collection of diagrammatic facts that have been established, either by
construction or by inference, at a given point in the proof; and we
have characterized the ``diagrammatic inferences'' as those
diagrammatic facts that are ``direct consequences'' of those. The goal
of this section is to complete the description of $\na{E}$ by spelling
out an adequate notion of ``direct consequence.''

Our attempts to define such a notion are constrained by a number of
desiderata. The first is fidelity to Euclid:
\begin{itemize}
\item The direct consequences of a set of diagrammatic hypotheses
  should provide an adequate model of the diagrammatic facts that
  Euclid makes use of in a proof, either explicitly or in licensing a
  construction or a metric conclusion, without explicit justification.
\end{itemize}
The next two are more mathematical:
\begin{itemize}
\item Any direct consequence should be, in particular, a
  first-order consequence of the diagrammatic axioms and diagrammatic
  facts in $\Gamma, \Delta$.
\item Conversely, any diagrammatic assertion that is a first-order
  consequence of the diagrammatic axioms should be derivable in
  $\na{E}$, though not necessarily in one step.
\end{itemize}
The first constraint says that direct consequences of a set of
diagrammatic assertions should be \emph{sound} with respect to the set
of first-order consequences of the diagrammatic axioms. The second
constraint says that together with the other methods of proof
provided by $\na{E}$, they should be \emph{complete} as well. We will
see that there is a lot of ground between these two constraints. For
example, they can be met by taking the direct consequences to be
\emph{all} first-order consequences. But this overshoots our first
desideratum, since it would let us make direct inferences that
Euclid spells out more explicitly. Nor does it sit well with the
notion of ``directness.'' Since we are dealing with a universal theory
in a language with no function symbols, the set of literals that are
consequences of a given set $\Gamma$ of literals is decidable: one
only need extract all instances of the axioms among the variables in
$\Gamma$, and use a decision procedure for propositional logic. But
this is unlikely to be computationally feasible,\footnote{We do not,
  however, have a lower bound on the computational complexity of the
  decision problem associated with our particular set of axioms.} and
we expect a ``direct'' inference to be more tame than that. Thus
our third desiderata is of a computational nature:
\begin{itemize}
\item The problem of determining whether a literal is a direct
  consequence of some diagrammatic facts should be, in some sense,
  computationally tractable.
\end{itemize}
The notion of tractability should be taken with a grain of salt. It is
loosely related to the practical question as to whether one can
implement a proof checker for our formal system which performs
reasonably on formalized proofs of statements in the \emph{Elements},
a question we address in Section~\ref{implementation:section}. But it
is worth keeping in mind that even our theoretical characterization is
only intended to be compelling at the level of complexity found in
proofs in the \emph{Elements}. When a diagram has millions of points,
lines, and circles, we may be faulted for sanctioning ``direct''
inferences that cannot be carried out with our limited cognitive
apparatus. But even propositional logic, as a model of logical
inference, is subject to the same criticisms: can we really
``recognize'' an instance of modus ponens when the formulas involved
have more than $2^{100}$ symbols?

To develop a notion of direct consequence, let us begin by noting
that most of our axioms are naturally expressed as rules; in other
words, they have the form
\begin{quote}
if $\ph_1, \ph_2, \ldots, \ph_n$ then $\psi$
\end{quote}
where $\ph_1, \ldots, \ph_n, \psi$ are literals. The example in
Section~\ref{diagrammatic:nature:section} suggests that we should be
able to chain such rules; that is, whenever we know $\ph_1, \ldots,
\ph_n$, we also know $\psi$, and can use $\psi$ to secure
further knowledge.  Occasionally, our diagrammatic axioms are not
quite in rule form, with either a disjunction among the hypothesis or
a conjunction in the conclusion. But this can be viewed as a
notational convenience; the rule ``if $\ph_1, \ph_2, \ldots, \ph_n$
then $\psi$ and $\theta$'' is equivalent to the pair of rules ``if
$\ph_1, \ph_2, \ldots, \ph_n$ then $\psi$'' and ``if $\ph_1, \ph_2,
\ldots, \ph_n$ then $\theta$,'' and the rule ``if $\ph_1, \ph_2,
\ldots, \ph_n$ and either $\theta$ or $\eta$ then $\psi$'' is
equivalent to the pair of rules ``if $\ph_1, \ph_2, \ldots, \ph_n$ and
$\theta$ then $\psi$'' and ``if $\ph_1, \ph_2, \ldots, \ph_n$ and
$\eta$ then $\psi$.''

A moment's reflection, however, shows that we should also allow
``contrapositive'' variants of our rules. For example, consider the
first Pasch axiom:
\begin{quote}
  if $b$ is between $a$ and $c$ and $a$ and $c$ are on the same side of
  $L$, then $a$ and $b$ are on the same side of $L$ 
\end{quote}
Certainly, if we know that $b$ is between $a$ and $c$ and that $a$ and
$c$ are on the same side of $L$, we should be allowed to infer that
$a$ and $b$ are on the same side of $L$. But suppose we know that $b$
is between $a$ and $c$ but that the conclusion fails, that is, $a$ and
$b$ are not on the same side of $L$. Drawing a picture or imagining
the situation in our mind's eye enables us to see, straightforwardly,
that the second hypothesis fails, that is, $a$ and $c$ are not on the
same side of $L$. In other words, we should include the rule
\begin{quote}
  if $b$ is between $a$ and $c$ and $a$ and $b$ are not on the same
  side of $L$ then $a$ and $c$ are not on the same side of
  $L$
\end{quote}
as a variant of the above. More generally, we should read the rule
``if $\ph_1, \ph_2, \ldots, \ph_n$ then $\psi$'' as the disjunction
\begin{quote}
  either not $\ph_1$, or not $\ph_2$, or \ldots, or not $\ph_n$, or
  $\psi$
\end{quote}
and infer any disjunct once we know that the others are false. This is
exactly the notion of direct consequence that we adopt: we take the
set of direct consequences of a set of diagrammatic assertions to be
the set obtained by closing the set under the inferences just
described.

Let us spell out the details more precisely. For simplicity, we
initially restrict our attention to propositional logic. A
\emph{clause} is simply a finite set of propositional literals; think
of each clause as representing the associated disjunction. Let $S$ be
a set of propositional clauses and let $\Gamma$ be a set of
propositional literals. Take negation as an operation mapping literals
to literals, that is, identify $\lnot \lnot p$ with $p$. We define the
\emph{set of direct consequences of $\Gamma$ under $S$} to be the
smallest set $\Gamma'$ of literals that includes $\Gamma$ and is
closed under the following rule: if $\{ \ph_1, \ldots, \ph_n \}$ is a
clause in $S$ and $\lnot \ph_1, \ldots, \lnot \ph_{n-1}$ are all in
$\Gamma'$, then $\ph_n$ is in $\Gamma'$. In other words, $\Gamma'$ is
obtained by starting with the literals in $\Gamma$ and applying the
rule above to add literals, one at a time, until no more literals can
be added. We adopt the understanding, however, that if $\Gamma'$
contains an atomic formula and its negation, then it contains every
literal; in other words, everything is a consequence of a
contradiction.

We now provide an alternative characterization of the set $\Gamma'$.
Consider a sequent calculus formulation of intuitionistic logic
\cite{buss:98e,troelstra:schwichtenberg:00}, with sequents of the form
$\Pi \Rightarrow \ph$, intended to denote that the set of hypotheses
in $\Pi$ entails $\ph$. Take the ``contrapositive variants'' of any
clause $\{ \ph_1, \ldots, \ph_n \}$ to be the sequents of the form $\{
\lnot \ph_1, \ldots, \lnot \ph_{n-1} \} \Rightarrow \ph_n$, again with
the understanding that if $A$ is atomic then $\lnot \lnot A$ is
replaced by $A$.

\begin{proposition}
\label{direct:equiv:prop}
Let $S$ be a set of clauses, and let $\Gamma, \theta$ be a set of
propositional literals. The following are equivalent:
\begin{enumerate}
\item $\theta$ is a direct consequence of $\Gamma$ under $S$.
\item There is an intuitionistic proof of the sequent $\Rightarrow
  \theta$ from initial sequents that are either contrapositive
  variants of the clauses in $S$ or of the form $\Rightarrow \psi$,
  where $\psi$ is a formula in $\Gamma$.
\end{enumerate}
\end{proposition}

\begin{proof}
  The implication from 1 to 2 is straightforward, since adding to
  $\Gamma'$ the result of applying our rule of inference with one of
  the clauses in $S$ is equivalent to inferring the consequence of the
  implication given by a contrapositive variant of that clause. The
  fact that as soon as $\Gamma'$ contains an atomic formula and its
  negation we take every literal to be a direct consequence follows
  from the fact that $\bot$, and hence every formula, is an
  intuitionistic consequence of an atomic formula and its negation.

  Conversely, suppose there is an intuitionistic proof of $\Rightarrow
  \psi$ from the initial sequents described in 2. Then by a version of
  cut-elimination theorem for the intuitionistic sequent calculus with
  axioms and additional rules (\cite[Theorem 2.4.5]{buss:98e} or
  \cite[Section 4.5.1]{troelstra:schwichtenberg:00}), there is a proof
  in which every cut formula is a literal. Since there are no other
  logical connectives in the initial sequents or conclusion, the only
  other rules used are the rules for negation and the ``ex falso''
  rule $\Pi, \bot \Rightarrow \eta$.

  We can therefore obtain the desired conclusion by proving the
  following claim:
\begin{quote}
  Suppose $d$ is a proof of a sequent $\{ \theta_1, \ldots, \theta_n
  \} \Rightarrow \eta$ from the initial sequents described in 2, using
  only the negation rules, \emph{ex falso}, and the cut rule
  restricted to literals. Then for any $\Gamma'' \supseteq \Gamma$,
\begin{enumerate}
\item if $\theta_1, \ldots, \theta_n$ are in $\Gamma''$, then $\eta$
  is in the closure of $\Gamma''$ under $S$; and
\item if $\eta$ is $\bot$ and $\theta_1, \ldots, \theta_{n-1}$ are in
  $\Gamma''$, then $\lnot \theta_n$ is in the closure of $\Gamma''$
  under $S$.
\end{enumerate}
\end{quote}
  This can be proved by a straightforward induction on $d$. Suppose the
  the last inference of $d$ is the cut rule,
\begin{prooftree}
\AXN{\theta_1, \ldots, \theta_n \fCenter \alpha}
\AXN{\theta_1, \ldots, \theta_n, \alpha \fCenter \eta}
\BIN{\theta_1, \ldots, \theta_n \fCenter \eta}
\end{prooftree}
If $\eta$ is not $\bot$, applying the inductive hypothesis to the left
subproof yields that for any $\Gamma'' \supseteq \Gamma$, if
$\theta_1, \ldots, \theta_n$ are in $\Gamma''$, then $\alpha$ is in
the closure of $\Gamma''$ under $S$. Applying the inductive hypothesis
to the right subproof and $\Gamma'', \alpha$ yields that $\eta$ is in
the closure of $\Gamma'', \alpha$ under $S$, and hence in the closure
of $\Gamma''$ under $S$, as required. The case where $\eta$ is $\bot$
is similar.

Handling the other rules is straightforward. For example, if the last
inference of $d$ is a left negation introduction, it is of the
following form:
\begin{prooftree}
\AXN{\theta_1, \ldots, \theta_{n-1} \fCenter \alpha}
\UIN{\theta_1, \ldots, \theta_{n-1}, \lnot \alpha \fCenter \eta}
\end{prooftree}
In that case, the desired conclusions are obtained by applying the
inductive hypothesis to the immediate subproof.
\end{proof}

In the statement of the last proposition, instead of taking all
contrapositive variants of the clauses in $S$, one can equivalently
take any \emph{one} contrapositive variant of each clause in $S$, if
we also add the following rule of double-negation elimination for
atomic formulas:
\begin{prooftree}
\AXN{\Pi, \lnot A \fCenter \bot}
\UIN{\Pi \fCenter A}
\end{prooftree}  
This has the net effect of making $\lnot \lnot A$ equivalent to $A$.
But it is important to recognize that this is \emph{not} the same as
adding the law of the excluded middle, $A \lor \lnot A$, for atomic
formulas.  Indeed, this is exactly what is missing from the notion of
a direct consequence. For example, suppose $S$ has rules ``if $A$ and
$B$ then $C$'' and ``if $A$ and not $B$ then $C$.'' Then $C$ is
certainly a classical propositional consequence of $\{ A \}$ under
these rules, since $C$ follows from both $B$ and from $\lnot B$. But
it is not a \emph{direct} consequence. This distinction is what makes
the notion of a direct consequence well-suited to modeling the
diagrammatic inferences in the \emph{Elements}. Euclid \emph{does}
explicitly introduce case splits when they are needed, and so any
inference that requires considering different diagrammatic
configurations, in an essential way, should not count as ``reading off
from the diagram.'' These case splits make all the difference: the
next two propositions show that, in the propositional setting, they mark the
difference between the complexity classes P and NP.

\begin{proposition}
Let $\Gamma$ be a set of literals and let $S$ be a set of clauses. The
question ``is $\theta$ a direct consequence of $\Gamma$ under
$S$?'' can be decided in time polynomial in the size of $\Gamma$ and
$S$. 
\end{proposition}

\begin{proof}
  If the encoding of $\Gamma$ and $S$ have length $n$, they contain at
  most $n$ propositional variables. Starting with the literals in
  $\Gamma$, iteratively apply the closure rule using clauses in $S$,
  until $\theta$ is added, or the set becomes inconsistent, or no
  further rules can be applied. Each step of the iteration amounts to
  scanning through the clauses in $S$ and matching against literals
  already in $\Gamma'$ to see whether a new literal can be added, and
  can be carried out in time polynomial in $n$. At each step, at least
  one literal is added the set $\Gamma'$ of consequences, so the
  process terminates in at most $n + 1$ steps.
\end{proof}

\begin{proposition}
Suppose one augments intuitionistic logic with the following rule:
\begin{prooftree}
\AXN{\Pi, A \fCenter \eta}
\AXN{\Pi, \lnot A \fCenter \eta}
\BIN{\Pi \fCenter \eta}
\end{prooftree}
where $A$ is an atomic formula and $\Pi, \eta$ is a set of literals.
Then a sequent $\Rightarrow \theta$ is provable from the initial
sequents described in Proposition~\ref{direct:equiv:prop} if and only
if $\theta$ is a classical consequence of $\Gamma$ together with the
clauses in $S$. Hence, in the presence of such case splits, the
problem of determining whether a literal is a consequence of $S$ is
NP complete.
\end{proposition}

\begin{proof}
  Since the rule for case splits is classically valid, it is clear
  that if $\Rightarrow \theta$ is provable from the initial
  sequents~\ref{direct:equiv:prop}, it is a classical consequence
  of $\Gamma$ together with the clauses in $S$.

  Conversely, given $\Rightarrow \theta$, we can work backwards and
  apply case splits until at each node we have a sequent $\Pi
  \Rightarrow \theta$ such that for every propositional variable $p$
  occurring in $\Gamma$ and $S$, either $p$ or $\lnot p$ is in $\Pi$.
  If each such sequent is classically inconsistent with $\Gamma$ and
  the clauses in $S$, we obtain a proof of $\Rightarrow \theta$.
  Otherwise, at least one such $\Pi$ describes a truth assignment
  which is consistent with $\Gamma$ and $S$ but makes $\theta$ false,
  showing that $\theta$ is not a classical consequence of $\Gamma$
  together with the clauses in $S$.

  To prove the final claim in the lemma, let $S$ be any set of
  propositional clauses, and let $p$ be a new propositional
  variable. Then $S$ is satisfiable if and only if $p$ is not a
  classical consequence of $S$. The claim follows from the fact that
  the satisfiability of a set of propositional clauses is NP complete.
\end{proof}

We now turn to the first-order setting. Suppose $S$ is a set of
clauses, where now a clause is a finite set of first-order literals.
Interpret these as universal axioms; that is, a clause $\{ \ph_1,
\ldots, \ph_n \}$ represents the universal closure of the associated
disjunction. If $\Gamma$ is a set of literals, define the set
$\Gamma'$ of direct consequences of $\Gamma$ under $S$ as before, but
now using arbitrary substitution instances of the clauses in $S$.

Focusing on $\na{E}$ in particular, we take the direct consequences of
a set of diagrammatic assertions, $\Gamma$, to be the set of direct
consequences of $\Gamma$ under the set of rules given in
Section~\ref{diagram:section}. Note that the language of $\na{E}$ has
no function symbols. Since there are a fixed number of relation
symbols, given $n$ variables ranging over points, lines, and circles,
one can bound the number of literals involving these variables with a
polynomial in $n$. The preceding propositions then show that our
notion of direct consequence has the following desirable
properties.

\begin{theorem}
  Every direct consequence of a set of diagrammatic assertions is a
  first-order consequence of these assertions and the diagrammatic
  axioms.
\end{theorem}

\begin{theorem}
  Any literal that is a classical consequence of a set of diagrammatic
  assertions and diagrammatic axioms can proved from those
  diagrammatic assertions in $\na{E}$.
\end{theorem}

\begin{theorem}
  Let $\Gamma$ be a set of diagrammatic assertions involving at most
  $n$ points, lines, and circles. Whether or not a particular literal
  is a direct diagrammatic consequence of $\Gamma$ can be
  determined in time polynomial in $n$.
\end{theorem}

Note that ``polynomial-time computable'' need not mean feasible in
practice. Since ``$\mybetween$'' is a ternary relation, with ten
points, for example, we have to keep track of a thousand potential
betweenness assertions. On the other hand, experiments described in
Section~\ref{implementation:section} suggest that even the full set of
quantifier-free consequences can be feasibly obtained for reasonable
diagrams, so that our system should be practically implementable as
well.

We should also provide an account of what it means to be a direct
metric consequence. It would be perhaps most faithful to Euclid to add
a finite list of variants extending the list of axioms given in
Section~\ref{metric:section}, allowing one to add equal segments to a
segment in either order, and so on. But recognizing $\seg{ab}$ and
$\seg{ba}$ as the same quantity, or $\seg{ab} + \seg{cd}$ and
$\seg{cd} + \seg{ab}$ as the same quantity, should not need explicit
justification; in general, a prover should be allowed to identify
terms up to associativity, commutativity, and symmetric
transformations without further comment. There are very simple
computational devices that make this easy to implement in practice
\cite{dershowitz:plaisted:01}, and it is the kind of thing we (like
Euclid) take for granted, and so we take these to be built into
$\na{E}$.

In fact, we would not be doing too much damage to Euclid if we allowed
\emph{any} metric consequence of previous metric facts to be inferred
in one step. This, too, has an easy computational implementation. As
noted above, the theory is just the universal fragment of the theory
of linearly ordered groups. Decision procedures for this theory have
been studied extensively, and at the level of complexity one finds in
Euclid's proofs, even the naive ``Fourier-Motzkin'' algorithm performs
quite well in practice. (See \cite{bockmayr:weispfenning:01} for an
overview of such methods.)

Finally, to handle the transfer axioms, we allow the prover to assert,
in one step, the conclusion of any single rule where the hypotheses
are all direct diagrammatic or metric consequences of the available
data, i.e.~the diagrammatic and metric assertions in $\Gamma, \Delta$.
Note that almost all these axioms can be described by clauses where
exactly one of the literals is a metric assertion. (The exception is
the third diagram-angle transfer axiom, which characterizes the notion
of a ``right angle'' by stating an equivalence between two metric
assertions in the context of some diagrammatic information. But this
could be replaced by the Euclidean theorem that if a line is cut by a
transversal, the adjacent angles add up to two right angles.)
Sometimes Euclid takes certain metric information to be so clear from
the diagram that he uses it without asserting it explicitly; these
include, for example, our diagram-angle axiom 4, which asserts that
different descriptions of the same angle have the same magnitude. In
cases like that, one could modify our definition of ``metric
consequence'' so that consequences of the diagram like these are added
to the ``store'' of available metric hypotheses automatically.

This concludes our presentation of $\na{E}$. The fact that there is
room to tinker with our notion of ``direct consequence'' by expanding
or contracting the allowable inferences should help clarify the nature
of our project. In order to show, in
Section~\ref{completeness:section}, that $\na{E}$ is sound and
complete with respect to the relevant ``ruler and compass'' semantics,
our one-step inferences have to be sound, and the full proof system
has to be complete. This gives us a lot of latitude in defining the
``one-step'' inferences. The fact that soundness and completeness do
so little to constrain our choice shows that we are trying to capture
something more fine-grained than the entailment relation for Euclidean
geometry.  Rather, we are trying to understand Euclidean \emph{proof},
which requires an understanding of the sorts of inferences that are
taken to be basic in the \emph{Elements}. So, where Euclid draws an
immediate conclusion from the data available in a proof, it should be
possible to carry out that inference in one step, or at most a few
steps, in our formal system.  On the other hand, in cases where Euclid
invokes a chain of steps to reach a conclusion, our system should
\emph{not} sanction that inference as ``direct.'' The extent to which
our system meets these constraints is the subject of the next section.

Ziegler \cite{ziegler} has shown that the notion of validity for
ruler-and-compass semantics is undecidable. (His proof shows that the
set of $\forall\exists\forall$ consequences of any finitely
axiomatized fragment of the theory of real closed fields is
undecidable. It is, however, still an open question whether the set of
$\forall\exists$ consequences, which correspond to the geometric
assertions that can be expressed in $\na{E}$, is decidable.) It is
therefore interesting to note that, in principle, one can expand our
notion of ``direct consequence'' dramatically and maintain
decidability:

\begin{theorem}
  The question as to whether a given literal is a first-order
  consequence of a finite set of literals and the set of all our
  diagrammatic, metric, and transfer axioms is decidable.
\end{theorem}

\begin{proof}
  The problem is equivalent to determining whether a finite set
  $\Gamma$ of literals is consistent with the diagrammatic, metric,
  and transfer axioms. Write $\Gamma = \Pi \cup \Theta$ where $\Pi$
  consists of the diagrammatic literals and $\Theta$ consists of the
  metric literals. By splitting on cases, we can assume without loss
  of generality that for every diagrammatic atomic formula $\ph$
  involving the variables occurring in $\Gamma$, either $\ph$ or
  $\lnot \ph$ is in $\Pi$. There are, moreover, only finitely many
  substitution instances of the axioms in question with the variables
  occurring in $\Gamma$. Modulo $\Pi$, all these axioms are equivalent
  to quantifier-free formulas over the metric sorts. We can then use a
  decision procedure for linear arithmetic to decide whether the
  resulting set of formulas, together with $\Theta$, is satisfiable.
\end{proof}

This means that if decidability, soundness, and completeness for
ruler-and-compass semantics were the only constraints, we could take
proofs in $\na{E}$ to be nothing more than a sequence of construction
steps, followed by ``Q.E.D.'' (or ``Q.E.F.''). Due to the case splits,
however, this naive algorithm runs in exponential time, and will be
infeasible in practice.

\section{Comparison with the \emph{Elements}}
\label{comparison:section}

In this section, we argue that $\na{E}$ provides an adequate modeling
of the proofs in Books I--IV of the \emph{Elements}, according to the
criteria presented in Section~\ref{characterizing:section}. In
Section~\ref{formal:language:section} we focus on the language of the
\emph{Elements}, and in Section~\ref{examples:section} we present some
examples to illustrate how Euclid's proofs are represented in
$\na{E}$. In Section~\ref{departures:section}, we explore some of the
ways in which proofs in $\na{E}$ differ from Euclid's, and in
Section~\ref{postulates:section} we compare our axiomatic basis to
his. Finally, Section~\ref{technical:section} provides a few more
examples of proofs, some of a technical nature, that will be needed in
our completeness proof in Section~\ref{completeness:section}.

\subsection{Language}
\label{formal:language:section}

We begin with a discussion of the language of the
\emph{Elements}. Since we have chosen a fairly minimal language for
$\na{E}$, we need to fix some conventions for interpreting the less
regimented and more expansive language in Euclid. For example, in the
\emph{Elements}, Euclid takes lines to be line segments, although
postulate 2 (``to produce a finite straight line continuously in a
straight line'') allows any segment to be extended
indefinitely. Distinguishing between finite segments and their
extensions to lines makes it clear that at any given point in a proof,
the diagrammatic information is limited to a bounded portion of the
plane. But, otherwise, little is lost by taking entire lines to be
basic objects of the formal system. So where Euclid writes, for
example, ``let $a$ and $b$ be points, and extend segment $ab$ to
$c$,'' we would write ``let $a$ and $b$ be distinct points, let $L$ be
the line through $a$ and $b$, and let $c$ be a point on $L$ extending
the segment from $a$ to $b$.'' Insofar as there is a fairly
straightforward translation between Euclid's terminology and ours,
we take such differences to be relatively minor.

Our basic diagrammatic terms include words like ``on,'' ``between,''
``inside,'' and ``same side.'' It is worth noting that such words
rarely occur explicitly in the \emph{Elements}. Diagrammatic
assertions are sometimes implicitly present in the result of a
construction; in the example of the last paragraph, we use ``$b$ is
between $a$ and $c$'' to represent one of the outcomes of the
diagrammatic construction. Euclid also sometimes uses the physical
diagram to convey a diagrammatic assertion. For example, in the first
proof in Section~\ref{examples:one:section}, the diagram shows that
point $d$ is on $ab$. Diagrammatic information is also implicit in
some of Euclid's more complicated locutions; for example, we need to
analyze the Euclidean assertion ``$abc$ is a triangle'' in terms of
our more basic primitives. But, overall, it is remarkable how
\emph{little} diagrammatic information needs to be asserted in the
text. One striking exception occurs in conveying the diagrammatic
notion of being parallel (which we model with the diagrammatic
predicate ``does not intersect''): there is no way to represent the
\emph{non}intersection of two lines in a diagram, and so Euclid uses
the term ``parallel'' explicitly in Propositions 27--47 of Book I to
make the assertion.

Modeling Euclid's limited use of explicit diagrammatic assertions has
been a central goal in the design of $\na{E}$. Although one is allowed
to enter diagrammatic assertions like ``$a$ is between $b$ and $c$''
and ``$a$ and $b$ are on the same side of $L$'' in proofs in $\na{E}$,
the point is that often one does not need to. For example, if the fact
that $b$ is between $a$ and $c$ is a direct consequence of
diagrammatic assertions in the hypotheses of the theorem and previous
construction steps, then, using a transfer axiom, one can simply
assert that $\seg{ab} + \seg{bc} = \seg{ac}$, without further
justification. Thus our choice of diagrammatic primitives was
designed, primarily, to function internally, and keep track of the
information that is required to license construction steps and
explicit metric inferences.

(We remind you that, in contrast to Tarski's and Hilbert's axiomatizations
of geometry, we use $\mybetween(a,b,c)$ to denote that $b$ is
\emph{strictly} between $a$ and $c$. This choice makes our
translation, in Section~\ref{completeness:section}, to a formal system
based on Tarski's axioms slightly more complicated. On the other hand,
it does seem to correspond more closely to Euclidean practice; see the
discussion in Section~\ref{nondegeneracy:section}. Interestingly, as
noted in Section~\ref{implementation:section} below, it also seems to
provide better performance in implementations.)

\newcommand{\f}{\mathit{f}}

Having discussed our choice of diagrammatic primitives, we comment
briefly on our modeling of metric assertions. In the Heath translation
of Euclid, one finds phrases like ``the base $ab$ is equal to the base
$de$,'' ``angle $abc$ is greater than angle $de\f$,'' and ``angles
$abc$, $cbd$ are equal to two right angles.'' We model these in our
formal system with the metric assertions $\seg{ab} = \seg{de}$,
$\angle abc > \angle de\f$, and $\angle abc + \angle cbd = \rightangle
+ \rightangle$. In reasoning about such quantities, Euclid uses basic
properties of an ordered group. For example, in the middle of the text of
Proposition I.13, we find:
\begin{quote}
  \ldots since the angle $dba$ is equal to the two angles $dbe, eba$,
  let the angle $abc$ be added to each; therefore the angles $dba,
  abc$ are equal to the three angles $dbe, eba, abc$. But the angles
  $cbe, ebd$ were proved equal to the same three angles; and things
  which are equal to the same thing are equal to one another;
  therefore the angles $cbe, ebd$ are also equal to the angles $dba,
  abc$. \cite[p.~275]{euclid}
\end{quote}
In our system, this sequence of assertions would be represented as
follows:
\begin{quote}
$\angle dba = \angle dbe + \angle eba$ \\
$\angle dba + \angle abc = \angle dbe + \angle eba + \angle abc$ \\
$\angle cbe + \angle ebd = \angle dbe + \angle eba + \angle abc$ \\
$\angle cbe + \angle ebd = \angle dba + \angle abc$
\end{quote}
In the example, the first assertion is a metric consequence of
diagrammatic information, namely that the point $e$ is in the interior
of the angle $dba$. The third assertion is echoed from earlier in the
proof, and the other two are obtained using axioms of
equality. Even though Euclid does not use a symbol for addition or the
word ``sum,'' it is clear from the text that his usage of magnitudes
``taken together'' is modeled well by the modern notions.

Other locutions found in Euclid can be modeled as ``definitional
extensions'' of $\na{E}$. For example, consider the phrase ``let $abc$
be a triangle.'' Assuming we take this to mean a nondegenerate
triangle, we parse this as saying that $a$, $b$, and $c$ are points, and
there are lines $L$, $M$, and $N$, such that $a$ and $b$ are on $L$
but $c$ is not, $b$ and $c$ are on $M$ but $a$ is not, and $c$ and $a$
are on $N$ but $b$ is not. Furthermore, the Euclidean phrase ``let
$ab$ be produced to $d$'' involves picking a point $d$ on $L$
extending the segment from $a$ to $b$, and so on. Adequate modeling of
Euclidean talk of triangles thus involves introducing mild forms of
``syntactic sugar'' to $\na{E}$.

When it comes to areas, we have only introduced a primitive for the
area of a triangle. Books I to IV also deal with areas of
parallelograms (including squares and rectangles) and, in the proof of
Proposition I.35, a trapezoid. One could introduce a new primitive to
denote the area of a convex quadrilateral (convexity can be defined in
the language of $\na{E}$), with appropriate axioms. Alternatively, one
can define the area of a convex quadrilateral $abcd$ to be the sum of
the areas of triangle $abc$ and $acd$, and then introduce the
requisite properties as ``derived rules.'' Extending $\na{E}$ to
handle the area of arbitrary convex polygons (that is, convex polygons
with an arbitrary number of sides) would require a more dramatic
extension, but this notion never arises in the \emph{Elements}.

One can prove in $\na{E}$ that one can pick an arbitrary point in a
triangle, say, or in a rectangle, but these facts require proof, even
though they are diagrammatically obvious. To our knowledge, however,
Euclid never does this. To model subsequent developments in
geometry, one would probably need to extend $\na{E}$ with a uniform
treatment of convex figures.

There are a number of concepts found in later books of the
\emph{Elements} that we have not incorporated into $\na{E}$.  For
example, Book V introduces the notion of multiples and ratios;
propositions in Book VI refer to arbitrary polygons; and Book VII,
which introduces elementary number theory, refers to arbitrary
(finite) collections of numbers. It would be interesting to extend
$\na{E}$ to model the Euclidean treatment of such concepts as well.

In our formulation of $\na{E}$, one is allowed to carry out arguments
by case splits on an atomic formula. Case splits in Euclid can be
slightly more expressive; for example, knowing that angles $abc$ and
$abd$ do not coincide, Euclid may consider the two cases $abc < abd$
and $abc > abd$. We would model this by first splitting on the
assertion $\angle abc < \angle abd$; then in the case $\angle abc
\not< \angle abd$, we would employ a second case split on the
predicate $\angle abc = \angle abd$, the positive instance which has
already been ruled out. We maintain that all case arguments occurring
in the first four books of the \emph{Elements} can be obtained in this
way, using a sequence of atomic splits to obtain an exhaustive list of
possibilities (e.g.~if $a$ is a point not on a line $L$, then another
point $b$ is either on the same side of $L$ as $a$, on $L$, or on the
opposite side of $L$), some of which are ruled out immediately
(implying $\bot$, and hence the desired conclusion right away). Once
again, mild forms of ``syntactic sugar'' would allow one to express
these case splits more compactly, resulting in proofs in $\na{E}$ that
more closely model the ones in Euclid.

When different diagrammatic configurations are possible, Euclid will
sometimes prove only one case. Often this case is truly ``without loss
of generality,'' which is to say, the other case (or cases) are
entirely symmetric. In $\na{E}$, strictly speaking, we would have to
repeat the proof; but one could introduce a syntactic term,
``similarly,'' to denote such a repetition. However, as Heath points
out repeatedly, Euclid often proves only the most difficult case of a
proposition and omits the others, even though they may require a
different argument; indeed, much of Proclus' commentary is devoted to
supplying proofs of the additional cases (see, for example, the notes
to Propositions 2, 7, 25, and 35 in \cite[Book I]{euclid}).  Of
course, in cases like this $\na{E}$ requires the full argument.  There
is no reasonable syntactic account of the phrase ``left to reader,''
and we do not purport to provide one.

\subsection{Examples of proofs in $\na{E}$}
\label{examples:section}

In this section, we provide some examples of proofs in our formal
system $\na{E}$, assuming the kinds of ``syntactic sugar'' described
in the last section. We include diagrams to render the proofs
intelligible, but we emphasize that they play no role in the formal
system. To improve readability, we use both the words ``Have'' and
``Hence'' to introduce assertions, generally using ``Have'' to
introduce new metric assertions that are inferred from the diagram,
and``Hence'' to introduce assertions that follow from previous metric
assertions. But these words play no role in the logical system; all
that matters are the actual assertions that follow. For the sake of
intelligibility, we also sometimes add comments, in brackets. Once
again, these play no role in the formal proof. Since the point of this
exercise is to demonstrate that proofs in $\na{E}$ are faithful to the
text of the \emph{Elements}, we recommend comparing our versions with
Euclid's.

Proposition 1 of Book I requires one, ``on a given straight line,
to construct an equilateral triangle.''

\newtheorem*{Prop1}{Proposition I.1}

\begin{Prop1}
\label{prop_I.1}
\ \\
Assume $a$ and $b$ are distinct points.\\
Construct point $c$ such that $\seg{ab} = \seg{bc}$ and $\seg{bc} =
\seg{ca}$.
\end{Prop1}

\begin{center}
\psfrag{c}{$c$}\psfrag{z}{$\alpha$}\psfrag{a}{$a$}\psfrag{b}{$b$}\psfrag{y}{$\beta$}
\includegraphics[height=2cm]{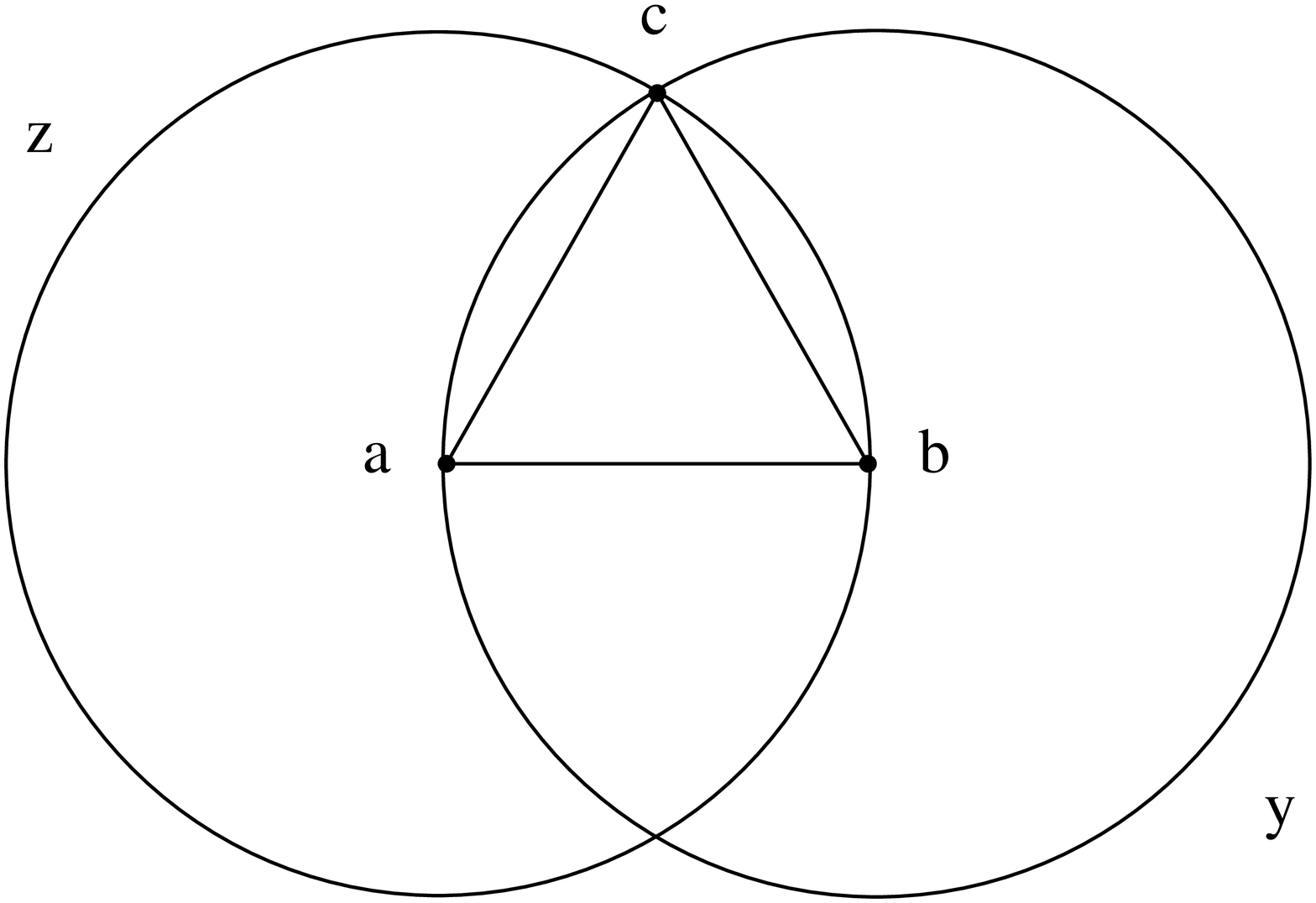}
\end{center}

\begin{proof}
Let $\alpha$ be the circle with center $a$ passing through $b$.\\
Let $\beta$ be the circle with center $b$ passing through $a$.\\
Let $c$ be a point on the intersection of $\alpha$ and $\beta$.\\
Have $\seg{ab} = \seg{ac}$ [since they are radii of $\alpha$].\\
Have $\seg{ba} = \seg{bc}$ [since they are radii of $\beta$].\\
Hence $\seg{ab} = \seg{bc}$ and $\seg{bc} = \seg{ca}$.\\
Q.E.F.
\end{proof}

The hypotheses tell us only that $a$ and $b$ are distinct points, but
this is enough to license the construction of $\alpha$ and $\beta$, by
rule 2 of the construction rules for lines and circles. Rule 5 of
diagram rules for intersections gives us the diagrammatic fact that
$\alpha$ and $\beta$ intersect. Rule 6 of the construction rules for
intersection then allows us to pick a point of intersection. Rule 3 of
the diagram-segment transfer axioms then allows us to conclude that
the given segments are equal, since they are radii of the two circles.
Using metric inferences (the symmetry of line segments and
transitivity of equality) gives us that $ab = bc = ca$.

Our proof does not establish, per se, that $c$ is distinct from $a$
and $b$, and this \emph{is} an assumption that Euclid uses freely when
applying the theorem. Fortunately, this is an easy metric consequence.

\newtheorem*{AuxE1}{Auxiliary to Proposition I.1}

\begin{AuxE1}
\ \\
  Assume $a$ and $b$ are distinct points, $\seg{ab} = \seg{bc}$, and
  $\seg{bc} = \seg{ca}$.\\
  Then $c \neq a$ and $c \neq b$.
\end{AuxE1}

\begin{proof}
Suppose $c = a$.\\
\hspace*{10pt} Hence $a = b$.\\
\hspace*{10pt} Contradiction.\\
Hence $c \neq a$.\\
Suppose $c = b$.\\
\hspace*{10pt} Hence $a = b$.\\
\hspace*{10pt} Contradiction.\\
Hence $c \neq b$.\\
Q.E.D.
\end{proof}

To show that $c$ is distinct from $a$, we suppose, to the contrary,
that $c = a$. Then direct metric inferences give us $\seg{ac} = 0$,
$\seg{ab} = 0$, and $a = b$, which is a contradiction. (We use the word
``Contradiction'' for ``Hence False.'') The fact that $c$ and $b$ are
distinct is proved in the same way.

A more faithful rendering of the proposition might assume ``Let $a$
and $b$ be distinct points on a line, $L$,'' and then also construct
the remaining lines $M$ and $N$ of the triangle. If one uses
Proposition I.1 as we initially stated it, one can simply construct
$M$ and $N$ afterwards. Euclid also, however, sometimes needs the fact
that $c$ is not on the line determined by $a$ and $b$. Once again, by
$\na{E}$'s lights, this requires a short argument.

\begin{AuxE1}
\label{prop_I.1_aux}
\ \\
Assume $a$ and $b$ are distinct points, $a$ is on $L$, $b$ is on $L$,
and $\seg{ab} = \seg{bc}$ and $\seg{bc} = \seg{ca}$. 
Then $c$ is not on $L$. 
\end{AuxE1}

\begin{proof}
  \ \\
  Suppose $c$ is on $L$.\\  
  \hspace*{10pt} Suppose $a$ is between $c$ and $b$.\\ 
  \hspace*{20pt} Hence $\seg{ca} < \seg{bc}$. Contradiction.\\
  \hspace*{10pt} Suppose $c = a$.\\
  \hspace*{20pt} Hence $a = b$. Contradiction.\\
  \hspace*{10pt} Suppose $c$ is between $a$ and $b$.\\ 
  \hspace*{20pt} Hence $\seg{ca} < \seg{ab}$. Contradiction.\\
  \hspace*{10pt} Suppose $c = b$.\\
  \hspace*{20pt} Hence $a = b$. Contradiction.\\
  \hspace*{10pt} Suppose $b$ is between $a$ and $c$.\\
  \hspace*{20pt} Hence $\seg{ab} < \seg{bc}$. Contradiction.\\
  Contradiction.\\
  Q.E.D.
\end{proof}

If $a$ and $b$ are distinct points on a line, Euclid often splits
implicitly or explicitly on cases depending on the position of a point
$c$ relative to $a$ and $b$. Strictly speaking, the proof above could
be expressed as a sequence of four nested case splits on atomic
formulas. As noted in the previous section, we can take the proof
above to rely on notational conventions, for readability.

When it is easy to rule out some cases, Euclid often does not say
anything at all, where our rules may require a line or two. The fact
that Euclid doesn't say anything to justify the nondegeneracy of the
triangle constructed in Proposition I.1, where $\na{E}$ requires some
(easy but) explicit metric considerations, is a more dramatic
difference, and is discussed in Section~\ref{departures:section}.
There, in fact, we note that in the proof of Proposition I.9, Euclid
seems to need a slight strengthening of our Proposition I.1, which
asserts that $c$ can be chosen on either side of the $L$ through $a$
and $b$. This is easily obtained using rule 8 instead of rule 6 of the
construction rules for intersections; one only needs to take the
trouble to make the stronger assertion.

Proposition 2 in Book I of the \emph{Elements} is surprisingly
complicated given that it occurs so early. It is a construction,
requiring one ``to place at a given point a straight line equal to a
given straight line,'' that is, to copy a segment to a given point.
This time, we leave it to you to check that the assertions are
justified by our rules and our notion of direct inference,
providing some hints in the bracketed comments. To simplify the
exposition, we appeal to a version of Proposition I.1 with the
additional distinctness claim.

\newtheorem*{Prop2}{Proposition I.2}

\begin{Prop2}
        \label{prop_I.2}
\ \\
Assume $L$ is a line, $b$ and $c$ are distinct points on $L$, and
$a$ is a point distinct from $b$ and $c$.\\
Construct point $f$ such that $\seg{af} = \seg{bc}$.
\end{Prop2}

\begin{center}
\psfrag{c}{$c$}\psfrag{d}{$d$}\psfrag{a}{$a$}\psfrag{b}{$b$}\psfrag{M}{$M$}\psfrag{N}{$N$}\psfrag{g}{$g$}\psfrag{f}{$f$}\psfrag{z}{$\alpha$}\psfrag{y}{$\beta$}
\includegraphics[height=3cm]{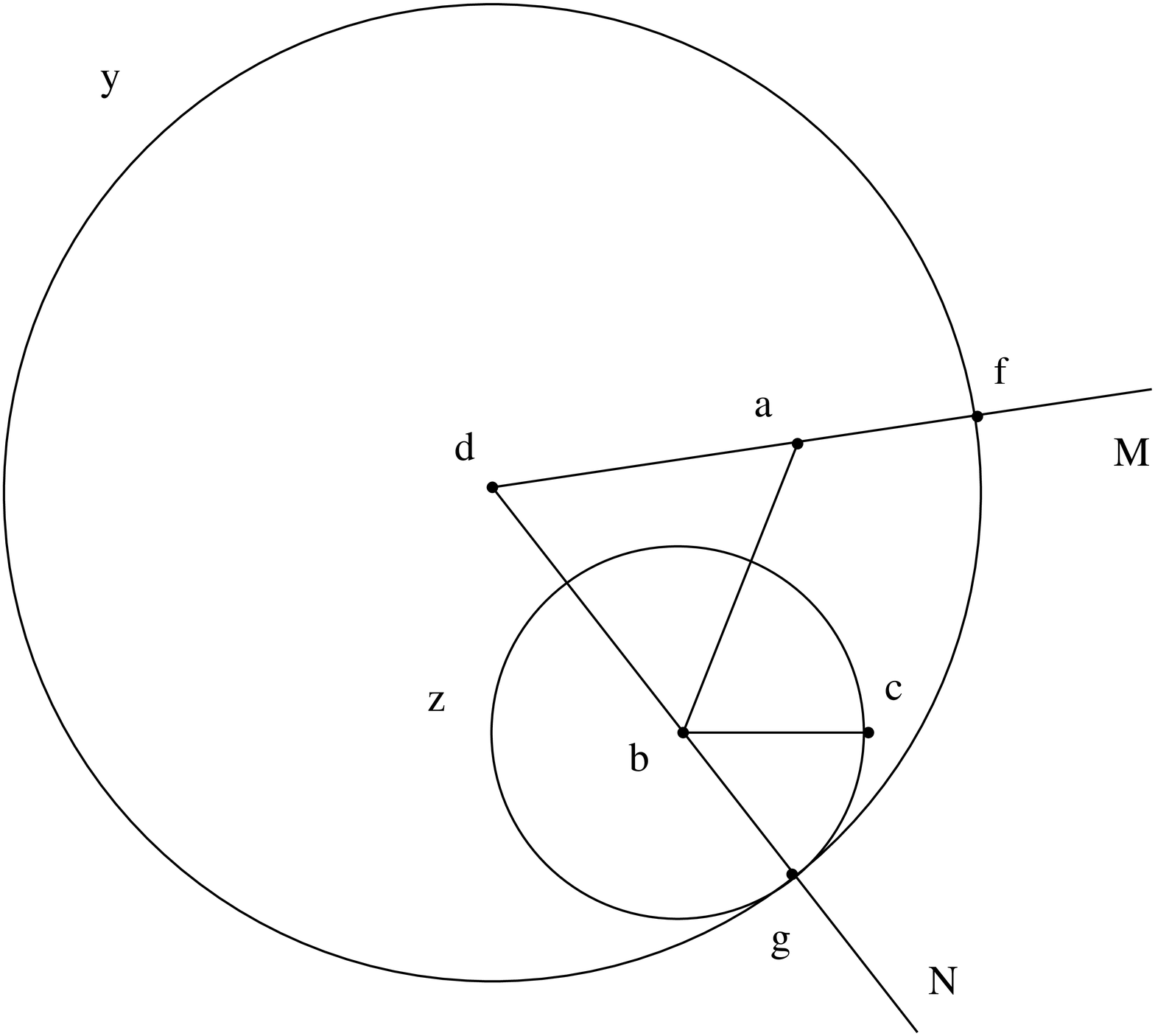}
\end{center}

\begin{proof}
  By Proposition I.1 applied to $a$ and $b$, let $d$ be a point such
  that $d$ is distinct from $a$ and $b$ and $\seg{ab} = \seg{bd}$ and 
    $\seg{bd} = \seg{da}$.\\
  Let $M$ be the line through $d$ and $a$.\\
  Let $N$ be the line through $d$ and $b$.\\
  Let $\alpha$ be the circle with center $b$ passing through $c$.\\
  Let $g$ be the point of intersection of $N$ and $\alpha$
  extending the segment from $d$ to $b$.\\
  Have $\seg{dg} = \seg{db} + \seg{bg}$.\\
  Hence $\seg{dg} = \seg{da} + \seg{bg}$ [since $\seg{da} = \seg{db}$].\\
  Hence $\seg{da} < \seg{dg}$.\\
  Let $\beta$ be the circle with center $d$ passing through $g$.\\
  Hence $a$ is inside $\beta$ [since $d$ is the center and
  $\seg{da} < \seg{dg}$].\\
  Let $f$ be the intersection of $\beta$ and $M$ extending the segment
  from $d$ to $a$.\\
  Have $\seg{df} = \seg{da} + \seg{af}$.\\
  Have $\seg{df} = \seg{dg}$ [since they are both radii of $\beta$].\\
  Hence $\seg{da} + \seg{af} = \seg{da} + \seg{bg}$.\\
  Hence $\seg{af} = \seg{bg}$.\\
  Have $\seg{bg} = \seg{bc}$ [since they are both radii of $\alpha$]. \\
  Hence $\seg{af} = \seg{bc}$.\\
  Q.E.F.
\end{proof}

Notice that the last construction step requires
knowing that $a$ is inside $\beta$. We obtain this, in our proof,
using simple metric considerations. We discuss this fact in the next
section.

Let us consider one more example. You may wish to compare the
following rendering of Proposition I.10 to the one given in
Section~\ref{examples:one:section}. Once again, to simplify the
exposition, we appeal to a version of Proposition I.1 with the
additional noncollinearity claim. The proof also appeals to
Proposition I.9, which asserts that an angle $acb$ can be bisected. We
take this to be the assertion that there is a point $e$ such that
$\angle ace = \angle bce$; with the further property that if $M$ is
the line through $c$ and $a$, and $N$ is the line through $c$ and $b$,
then $e$ and $b$ are on the same side of $M$, and $e$ and $a$ are on
the same side of $N$.  The last requirement could be expressed more
naturally with the words ``$e$ is inside the angle $acb$,'' though
that locution does not make $M$ and $N$ explicit. This requirement
rules out choices of $e$ on the other side of $c$ which satisfy the
same metric conditions.

\begin{Prop10}
        \label{prop_I.10}
        \ \\
        Assume $a$ and $b$ are distinct points on a line $L$.\\
        Construct a point $d$ such that $d$ is between $a$ and $b$ and
        $\seg{ad} = \seg{db}$.
\end{Prop10}

\begin{center}
\psfrag{a}{$a$}\psfrag{b}{$b$}
\psfrag{c}{$c$}\psfrag{d}{$d$}\psfrag{e}{$e$}
\psfrag{L}{$L$}\psfrag{M}{$M$}\psfrag{N}{$N$}\psfrag{K}{$K$}
\includegraphics[height=3cm]{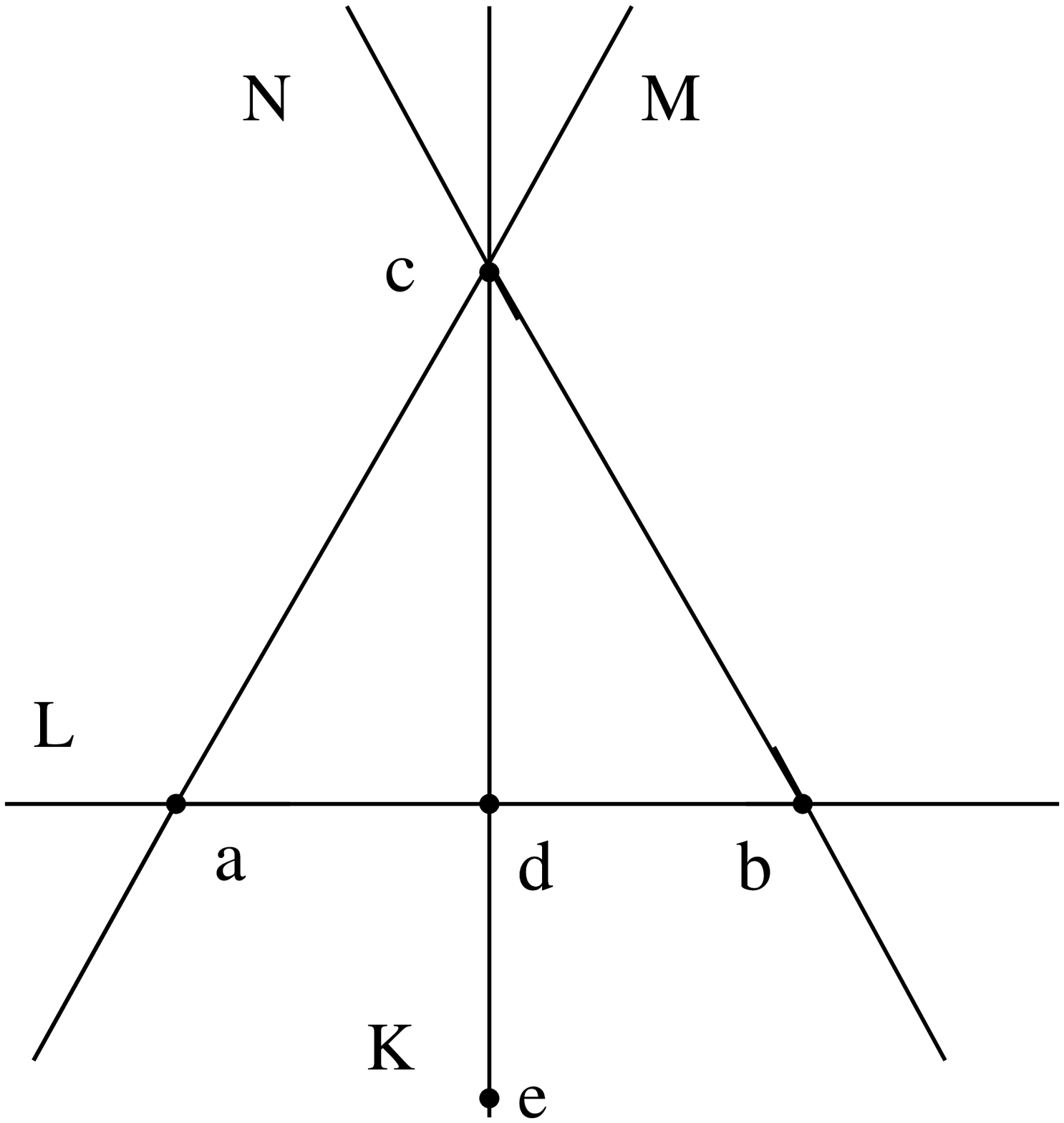}
\end{center}

\begin{proof}
  By Proposition I.1 applied to $a$ and $b$, let $c$ be a point such
  that $\seg{ab} = \seg{bc}$ and $\seg{bc} = \seg{ca}$ and $c$ is not
  on $L$.\\
  Let $M$ be the line through $c$ and $a$.\\
  Let $N$ be the line through $c$ and $b$.\\
  By Proposition I.9 applied to $a$, $c$, $b$, $M$, and $N$, let
    $e$ be a point such that $\angle ace = \angle bce$, $b$ and
    $e$ are on the same side of $M$, and $a$ and $e$ are on the same
    side of $N$.\\
  Let $K$ be the line through $c$ and $e$.\\
  Let $d$ be the intersection of $K$ and $L$.\\
  Have $\angle ace = \angle acd$.\\
  Have $\angle bce = \angle bcd$.\\
  By Proposition I.4 applied to $a$, $c$, $d$, $b$, $c$, and $d$ have
    $\seg{ad} = \seg{bd}$.\\
  Q.E.F.
\end{proof}

As noted in Section~\ref{examples:one:section},
when applying Proposition I.9, Euclid immediately takes $d$ to be the
point of intersection; we need to bisect the angle and then choose the
intersection explicitly. A direct diagrammatic inference yields the
fact that the two lines intersect: the triple incidence axioms imply
that points $a$ and $b$ are on opposite sides of $K$, which serves as
the hypothesis to intersection rule 1. We also need to note that the
angles $acd$ and $bcd$ are then the same as angles $ace$ and $bce$,
which is justified by metric rule 6. The fact
that $d$ is between $a$ and $b$ is again the result of a direct
diagrammatic inference, using Pasch inference 4.

There are some cases where the extent to which formal proofs in
$\na{E}$ match Euclid's is particularly impressive. For example,
Proposition 1 of Book III is ``to find the center of a given circle.''
This may seem strange, since Euclid's definitions seem to suggest that
every circle comes ``equipped'' with its center;\footnote{We are
  grateful to Henry Mendell for pointing this out.} but the
proposition makes it clear that we can be ``given'' a circle on its
own. The fact that we use a relation symbol rather than a function
symbol to pick out the center of a circle makes our formalization of
Proposition III.1 as $\ex a. \mycenter(a,\gamma)$ perfectly natural,
and the proof is essentially Euclid's.

For another example, Proposition 2 of Book III shows that circles are
convex --- more precisely, that the chord of a circle lies inside the
circle. This, too, is somewhat surprising, since that fact seems to be
as obvious as anything one is allowed to ``read off'' from a
diagram. But in $\na{E}$, one needs a proof using metric
considerations, as in Euclid. Thus $\na{E}$ can help ``explain'' some
puzzling features of the \emph{Elements}.

\subsection{Departures from the \emph{Elements}}
\label{departures:section}

In this section, we discuss some instances where proofs in the
\emph{Elements} do not accord as well with the rules of $\na{E}$.
Perhaps unsurprisingly, the most common type of departure involves
cases where Euclid's arguments are not detailed enough, by the
standards of $\na{E}$. Among these cases, two situations are typical:
first, Euclid is sometimes content to consider only one case when
$\na{E}$ demands a case analysis, and, second, Euclid sometimes reads
directly from the diagram a geometric relation which in $\na{E}$ must
be licensed by a transfer rule. We will consider examples of each, in
turn.

\begin{figure}
\begin{center}
\psfrag{e}{$e$}\psfrag{d}{$d$}\psfrag{a}{$a$}\psfrag{f}{$f$}
\includegraphics[height=2.2cm]{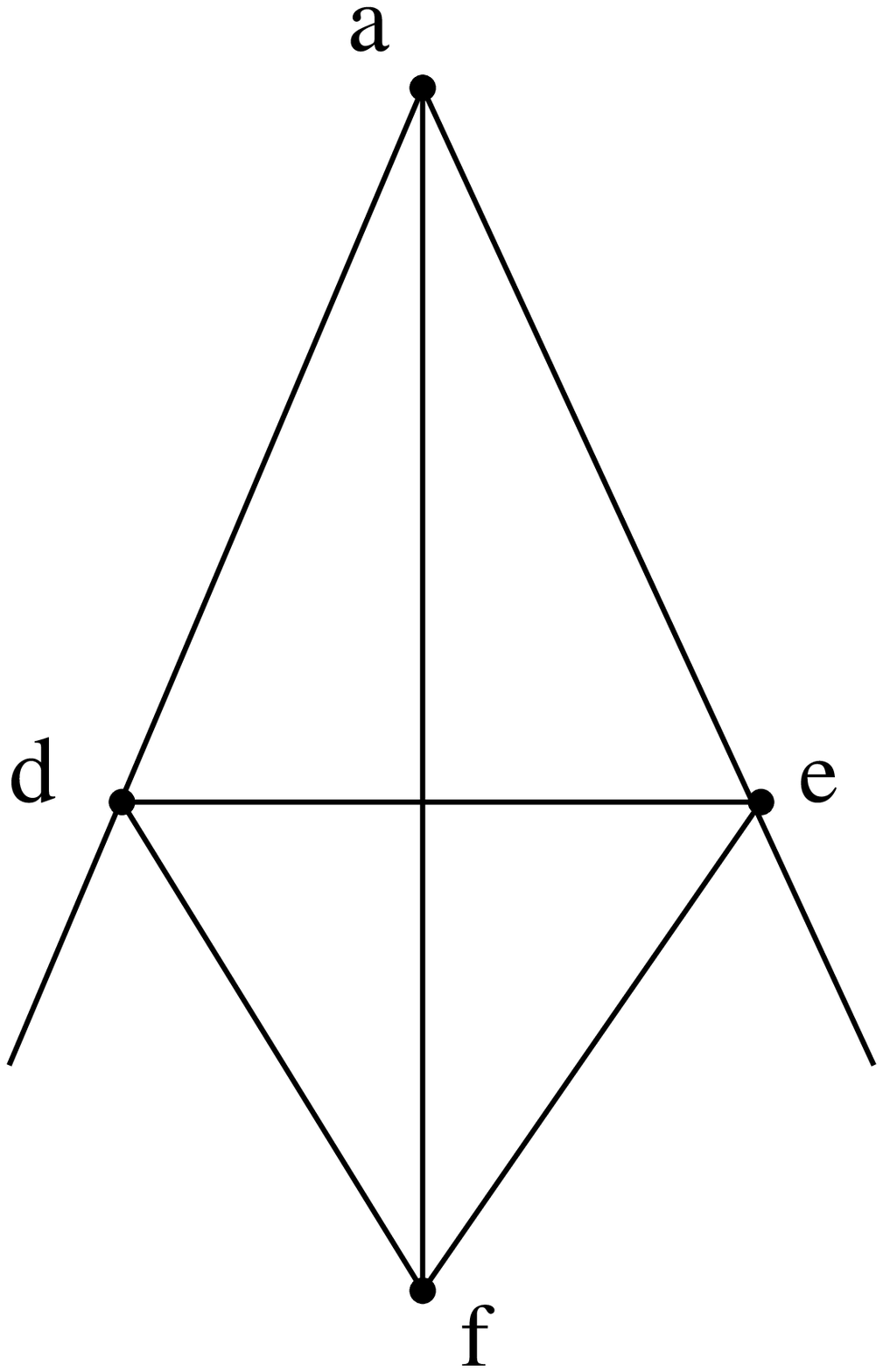} \hspace{2cm} \includegraphics[height=2.2cm]{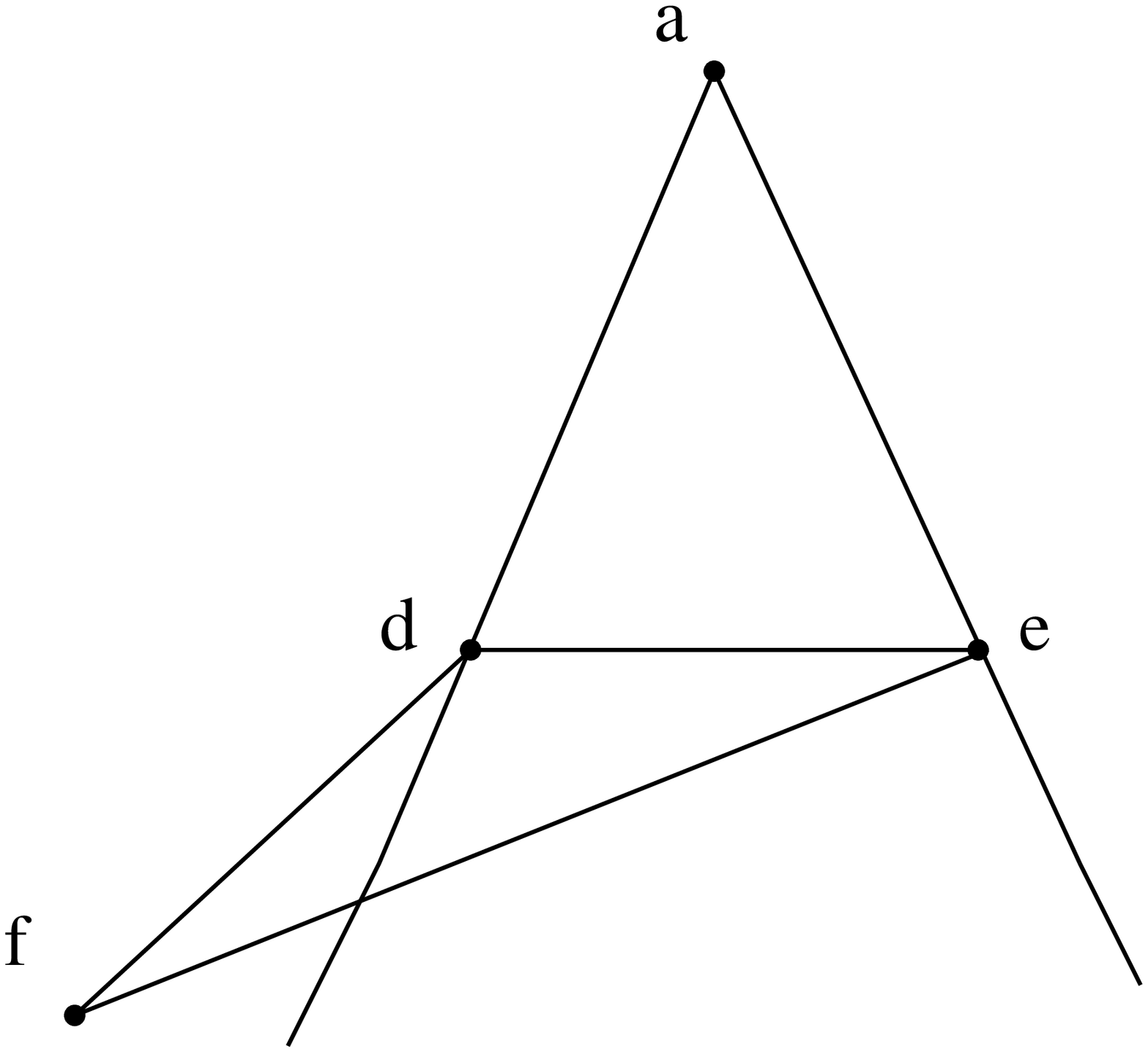} 
\end{center}
\caption{Two cases for Proposition I.9 considered in $\na{E}$. }
\label{prop:nine:figure}
\end{figure}

As pointed out in Section~\ref{proofs:section}, Euclid occasionally
reasons by cases to establish a proposition. When Euclid carries out
such a case analysis, $\na{E}$ typically provides a natural account of
the proof. But when $\na{E}$ demands a case analysis, Euclid does not
always provide one. For an example, consider Euclid's proof of
Proposition 9 in book I.  The proposition is a problem which demands
the construction of an angle bisector (see
Figure~\ref{prop:nine:figure}).  After constructing equal segments
$ad$ and $ae$ on the two sides of the given angle (with vertex $a$),
Euclid joins $d$ and $e$ and constructs on the segment the equilateral
triangle $dfe$.  The vertex $f$ of the triangle is then joined with
the vertex $a$ of the angle, and it is then argued that this segment
bisects the angle.  Euclid takes it as given that the point $f$ falls
within the angle. In $\na{E}$, however, one cannot. Though one may
stipulate that $f$ falls on the side of the segment $de$ opposite the
point $a$, one cannot assume anything about $a$'s position with
respect to the sides of the angle.  One must consider the cases where
$f$ falls on or outside the angle, and show that they are
impossible.\footnote{Vaughan Pratt has pointed out to us the
  contrapositive of Proposition 7 shows that if $ad$ is equal to $ae$,
  $df$ is equal to $ef$, and $d$ and $e$ are distinct, then $d$ and
  $e$ cannot lie on the same side of $af$. This immediately rules out
  two of the cases. But Euclid typically carries out an explicit
  reductio when he needs the contrapositive form of a prior
  proposition. Thus, if that is the proof one has in mind, $\na{E}$
  requires one to do the case split and apply Proposition 7
  explicitly.}

\begin{figure}
\begin{center}
\vspace*{0.5cm}
\includegraphics[height=1.7cm]{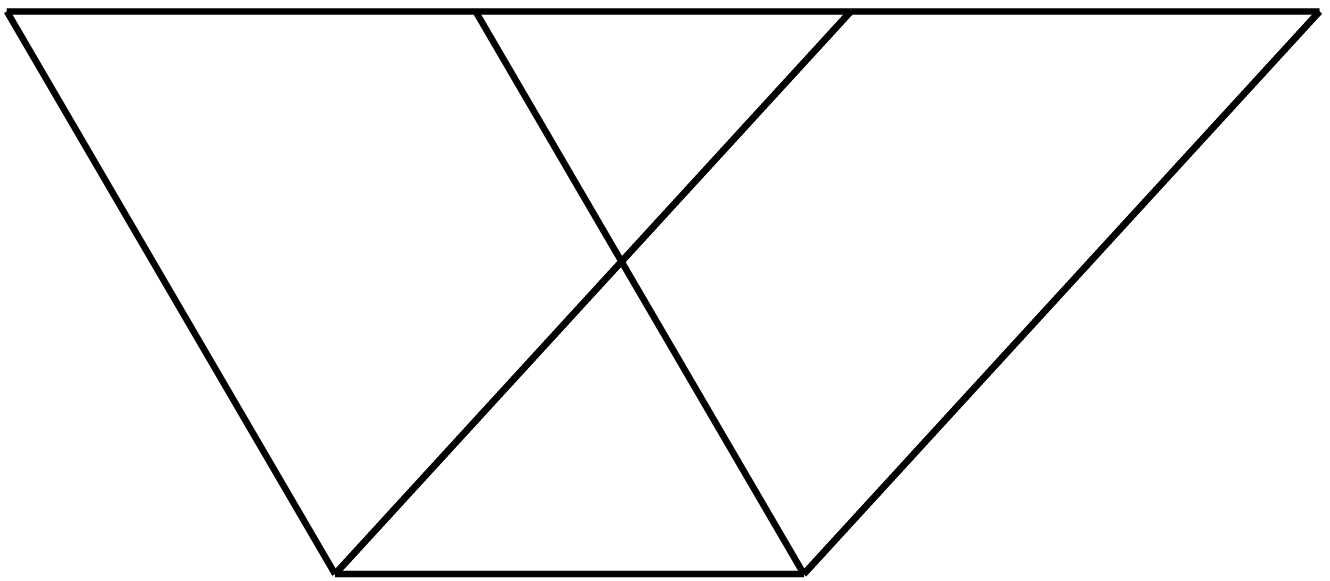}\hspace{1cm}\includegraphics[height=1.7cm]{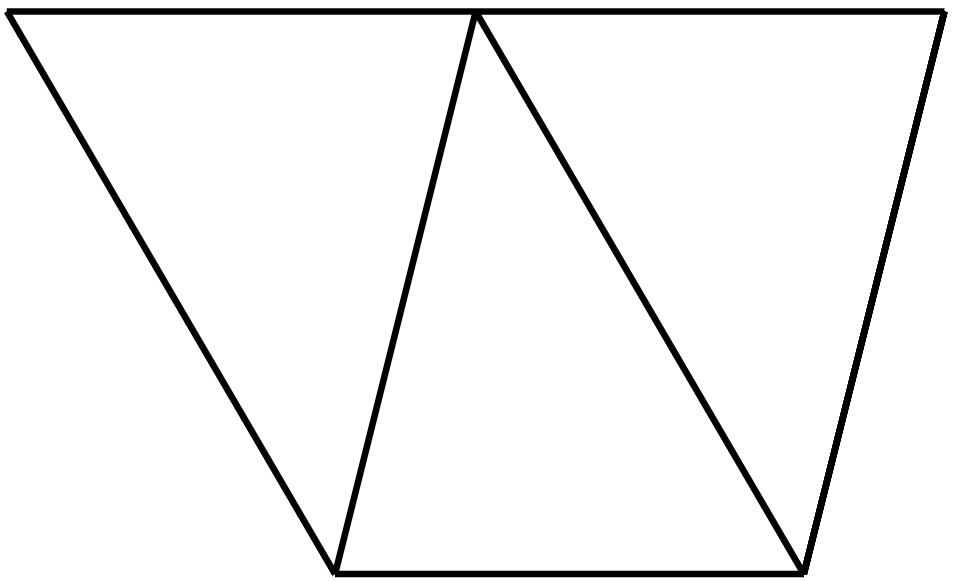}\hspace{1cm}\includegraphics[height=1.7cm]{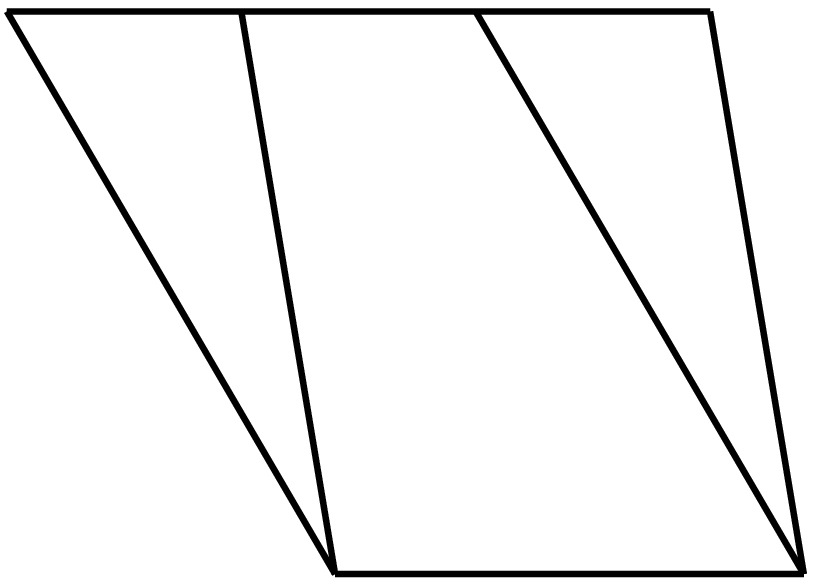} 
\end{center}
\caption{The three cases for Proposition I.35. }
\label{prop:thirty:five:figure}
\end{figure}

Another place where Euclid falls short of meeting $\na{E}$'s standards
for case analysis is Proposition I.35. Whereas with Proposition I.9
the need for a case analysis arises within the construction, with
Proposition I.35 one must start the proof with a case analysis (see
Figure~\ref{prop:thirty:five:figure}).  Euclid's statement of the
proposition is too general for the proof which follows. The
proposition underlies the familiar formula that the area of a
parallelogram is the product of its base and height. It asserts,
specifically, that two parallelograms which have the same base and are
bounded by the same parallel lines have the same area. The proof in
the \emph{Elements}, however, establishes a weaker result, in which
the parallelograms satisfy another condition: the nonintersection of
the sides opposite the common base of the parallelograms. Euclid
groups together into one case the different ways the sides opposite
the base can relate to one another positionally. But the containment
relations which license Euclid's steps in his proof do not generalize
to the other cases, which really require separate proofs.

Proclus, in fact, commented on Euclid's cavalier attitude toward cases
in Propositions I.9 and I.35, and furnished proofs for some of the
cases Euclid neglected. Thus $\na{E}$ is better understood as a
codification of the more critical attitude towards cases found in
Proclus's commentary. It is an interesting question as to why Euclid
is less rigorous in cases like these. One possible explanation is
given by Heath's observation that Euclid only worries about the most
difficult case. Another, which would apply to I.9 but not I.35, is
that the norms governing the physical construction of diagrams
automatically rules out certain possibilities for
Euclid.\footnote{Such norms would enforce what Manders terms
  \emph{diagram discipline}. The idea is as follows. Though physical
  rulers and compasses cannot produce perfectly straight lines and
  circles, a geometer trained in diagram discipline can be trusted to
  produce approximately straight lines and circles in his diagrams.
  For $f$ to lie on or outside the angle $dae$ in I.9, however, one or
  more of the circles used in the construction of $f$ would have to be
  dramatically non-circular.  Euclid would thus be justified in
  disregarding the case as a possibility. See \cite[section 3.1,
  p.~131]{manders:08}, and also the discussion of case branching in
  \cite[Section 1.4, p.~95]{manders:08a}.}

As with $\na{E}$'s rules for case analysis, its transfer
rules can be understood as the articulation of standards observed
intermittently in the \emph{Elements}.  In some constructions, the
possibility of a certain step depends on metric facts assumed of the
configuration. On such occasions, $\na{E}$ requires that a
metric-to-diagram rule be invoked.  Euclid sometimes recognizes the
need for such justifications, and sometimes does not.
 
One place where he does not is in Proposition 2 of Book I. In terms of
the $\na{E}$ proof given in Section~\ref{examples:section}, Euclid
does not provide any argument that the point $a$ \emph{has} to lie
within the circle $\beta$.  The diagrammatic information in the proof
regarding $a$ with respect to $\beta$, however, does not alone imply
it.  The metric fact that $da < dg$ must be added to the proof for the
position of $a$ inside $\beta$ to be forced.  The $\na{E}$ proof of
Proposition 2 thus contains a few lines not present in Euclid's proof.
 
Euclid does explicitly state one metric-to-diagram rule: the famous
parallel postulate.  The postulate allows Euclid to speak of an
intersection point between two lines---a diagrammatic piece of
data---given metric data about a configuration in which the lines are
embedded.  Accordingly, in Propositions I.44 and II.10 Euclid invokes
it to justify the introduction of certain intersection points.
Strangely, however, a similar justification is needed for intersection
points appearing in Euclid's proofs of Propositions I.42 and I.45, but
Euclid does not provide it.  He simply takes the intersection points
to exist without mentioning the parallel postulate.  The reasons for
this inconsistency are not immediately apparent.  The arguments which
are lacking in I.42 and I.45 are more complicated than those included
in I.44 and II.10.  Perhaps Euclid did not want to complicate his
exposition, or perhaps it was just an oversight. In any case, in
$\na{E}$, one must invoke the parallel postulate in the proofs of all
four propositions.

We close this section with a discussion of another interesting
difference between $\na{E}$ and Euclid. This time, it is an instance
where, by $\na{E}$'s lights, Euclid does too much. At issue are the
identity conditions of circles.  Euclid's definition reads as follows:
\begin{quote} 
  A \emph{circle} is a plane figure contained by one line such that
  all the straight lines falling upon it from one point among those
  lying within the figure equal one another; and the point is called
  the \emph{center} of the circle. \cite[pp.~153--154]{euclid}
\end{quote} 
In $\na{E}$ this definition translates into diagram-segment transfer
Rules 2, 3, and 4.  The function of the Rule 2 is to fix the
construction of a circle from a given length as unique.  In fixing it
as a rule in $\na{E}$, we take it to express Euclid's definition
directly.  Euclid, however, feels that it is at least conceivable that
two distinct circles with equal radii be constructed from the same
center, for in Proposition III.5 he proves that such a configuration
is impossible. From this result Rule 2 then follows immediately.
 
Thus, with Proposition III.5 Euclid requires a proof for something
which one can assume without proof in $\na{E}$.  There is nothing,
however, about the general structure of $\na{E}$ which forces this
difference; we could have replaced our Rule 2 with a rule that
licenses the key diagrammatic inference in Euclid's proof of III.5.
Such a rule, however, would be complicated, and rather than assume it
we have decided to treat circles as uniquely defined by a center and a
length. Instead, our Rule 2 conforms better to the modern conception
of a circle as the set of points which lie a fixed distance from a
given center.

\subsection{Euclid's postulates and common notions}
\label{postulates:section}

Since the \emph{Elements} is presented as an axiomatic development, it
it is worth considering Euclid's postulates and common notions, to see
how they line up with the fundamental rules of $\na{E}$. In the Heath
translation \cite[p.~154--155]{euclid}, the postulates are as follows:
\begin{enumerate}
\item To draw a straight line from any point to any point.
\item To produce a finite straight line continuously in a straight
  line.
\item To describe a circle with any centre and distance.
\item That all right angles are equal to one another.
\item {}[The parallel postulate; see Section~\ref{metric:section}.]
\end{enumerate}
Postulates 1 and 3 are the construction rules of $\na{E}$ for
lines and circles. Postulate 2 does not have a direct translation
in our system, given that we take all our lines to be ``indefinitely
extended''; but since Euclid will use this, say, to extend a segment
$ab$ to a point $c$, it essentially corresponds to construction 4
for points. Our remaining construction rules let us choose
``arbitrary points'' or label points of intersection. Euclid doesn't
say anything more about this; he just does it. As noted in
Section~\ref{transfer:section}, Euclid's Postulate 4 essentially
corresponds to our diagram-angle transfer axiom 3. Similarly,
Postulate 5 is our diagram-angle transfer axiom 5.

Euclid's common notions are as follows \cite[p.~155]{euclid}:
\begin{enumerate}
\item Things which are equal to the same thing are also equal to one
  another.
\item If equals be added to equals, the remainders are equal.
\item If equals be subtracted from equals, the remainders are equal. 
\item Things which coincide with one another are equal to one another.
\item The whole is greater than the part.
\end{enumerate}
These, for the most part, govern magnitudes; in our formulation, they
are therefore subsumed by the laws that govern the metric sorts,
together with the transfer axioms that relate the diagrammatic notions
of ``adding,'' ``subtracting,'' and ``being a part of'' to the
operations on magnitudes. For example, common notions 1 and 2 are
equality rules, and common notion 3 is the cancellation axiom, modulo
what it means to combine magnitudes in diagrammatic terms. Our first
diagram-segment transfer axiom explains what it means to add
adjacent segments; our second diagram-angle transfer axiom explains
what it means to add adjacent angles; our second diagram-area
transfer axiom explains what it means to combine the areas of
adjacent triangles. In each case, one can take the diagrammatic
configurations representing the component magnitudes to be the
``parts'' of the diagram configurations representing the sum. In that
case, the last common notion, 5, corresponds to the fact that
nontrivial segments, angles, and areas are positive, as given by the
corresponding transfer axioms.

Thus, Euclid's postulates correspond to some of our construction rules
and transfer inferences, and the common notions correspond to metric
inferences and other transfer inferences. The remainder of our
construction rules, and \emph{all} our diagram inferences, are then
subsumed under what Euclid takes to be implicit in the definitions and
the meanings of the undefined terms. It is, perhaps, regrettable that
there is not a cleaner mapping from our axioms to Euclid's. But, as
the discussion above indicates, even a simple principle like ``the
whole is greater than the part'' assumes an understanding of how
wholes and parts can be recognized in a diagram, and it is this
implicit understanding that we have tried to spell out with the rules
of $\na{E}$.

\subsection{Additional proofs}
\label{technical:section}

In this section, we provide three additional theorems of $\na{E}$,
which are needed for the completeness proof in the next section.  The
first is Euclid's Proposition I.12. Here, the phrase ``$M$ is
perpendicular to $L$'' masks implicit references to points $p, d, a$
such that $p$ is on $M$, $d$ is on both $M$ and $L$, $a$ is on $L$,
and angle $pda$ is a right angle.

\newtheorem*{Prop12}{Proposition I.12}

\begin{Prop12}
        \label{prop_I.12}
\ \\
        Assume point $p$ is not on line $L$.\\
        Construct a line $M$ through $p$ which is perpendicular to $L$.
\end{Prop12}

\begin{center}
\psfrag{p}{$p$}\psfrag{d}{$d$}\psfrag{a}{$a$}\psfrag{b}{$b$}\psfrag{q}{$q$}\psfrag{A}{$\alpha$}\psfrag{M}{$M$}\psfrag{L}{$L$}
\includegraphics[height=3cm]{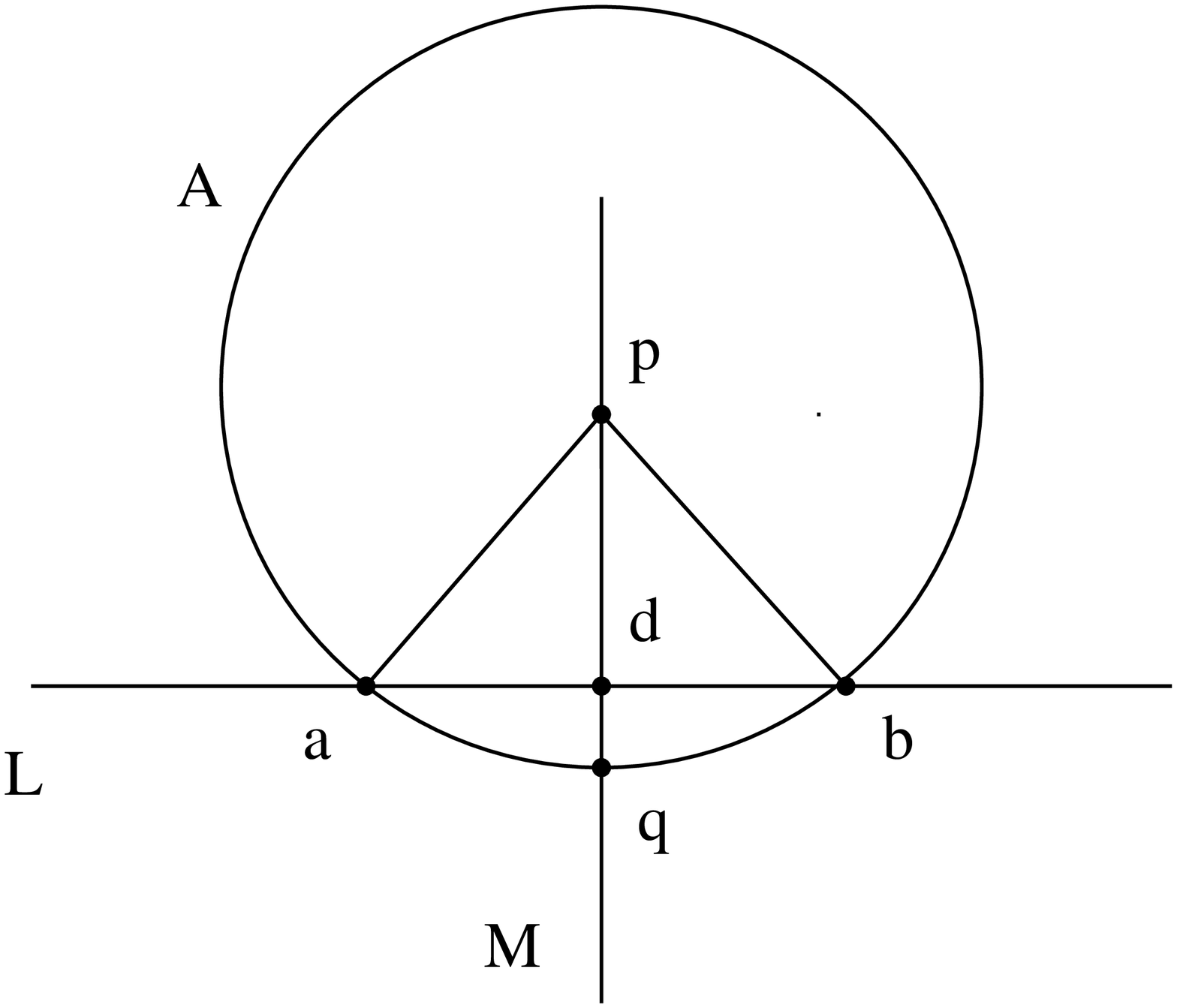}
\end{center}

\begin{proof}
  Let $q$ be a point on the opposite side of $L$ from $p$.\\
  Let $\alpha$ be the circle through $q$ with center $p$.\\
  Let $a$ and $b$ be the points of intersection of $L$ and
  $\alpha$.\\
  By Proposition I.10, let $d$ bisect segment $ab$.\\
  Let $M$ be the line through $p$ and $d$.\\
  By Proposition I.8 applied
  to triangles $pda$ and $pdb$, we
  have $\angle pda = \angle pdb$.\\
  Hence $\angle pda$ is a right angle.\\
  Q.E.F.
\end{proof}

The proof is almost identical to Euclid's. Notice that it is the
fourth diagram intersection rule that licenses the assertion that $L$
and $\alpha$ intersect.

The next two propositions are of a purely technical nature.  The first
shows how a construction in $\na{E}$ can depend on a case split (see
footnote~\ref{case:footnote}). Once again, we have taken some
liberties with the wording. Reference to the ``line through $p$ and
$s$,'' for example, masks a reference to a variable for a line on
which $p$ and $s$ both lie.

\newtheorem*{Tech1}{Technical Proposition 1}
\newtheorem*{Tech2}{Technical Proposition 2}

\begin{Tech1}
\label{same_to_points}
\ \\
Assume $p\ne q$ are on the same side of line $L$.\\
Construct points $r,s,t$ such that
\begin{enumerate}
\item $s,t$ are on $L$,
\item $r$ is the intersection of the line through $p$ and $s$ and the
  line through $q$ and $t$.
\end{enumerate}
\end{Tech1}

\begin{center}
\psfrag{p}{$p$}\psfrag{q}{$q$}\psfrag{r}{$r$}\psfrag{s}{$s = f$}
\psfrag{t}{$t = e$}\psfrag{L}{$L$}\psfrag{M}{$M$}\psfrag{N}{$N$}
\psfrag{O}{$O$}\psfrag{P}{$P$}
\includegraphics[height=3.5cm]{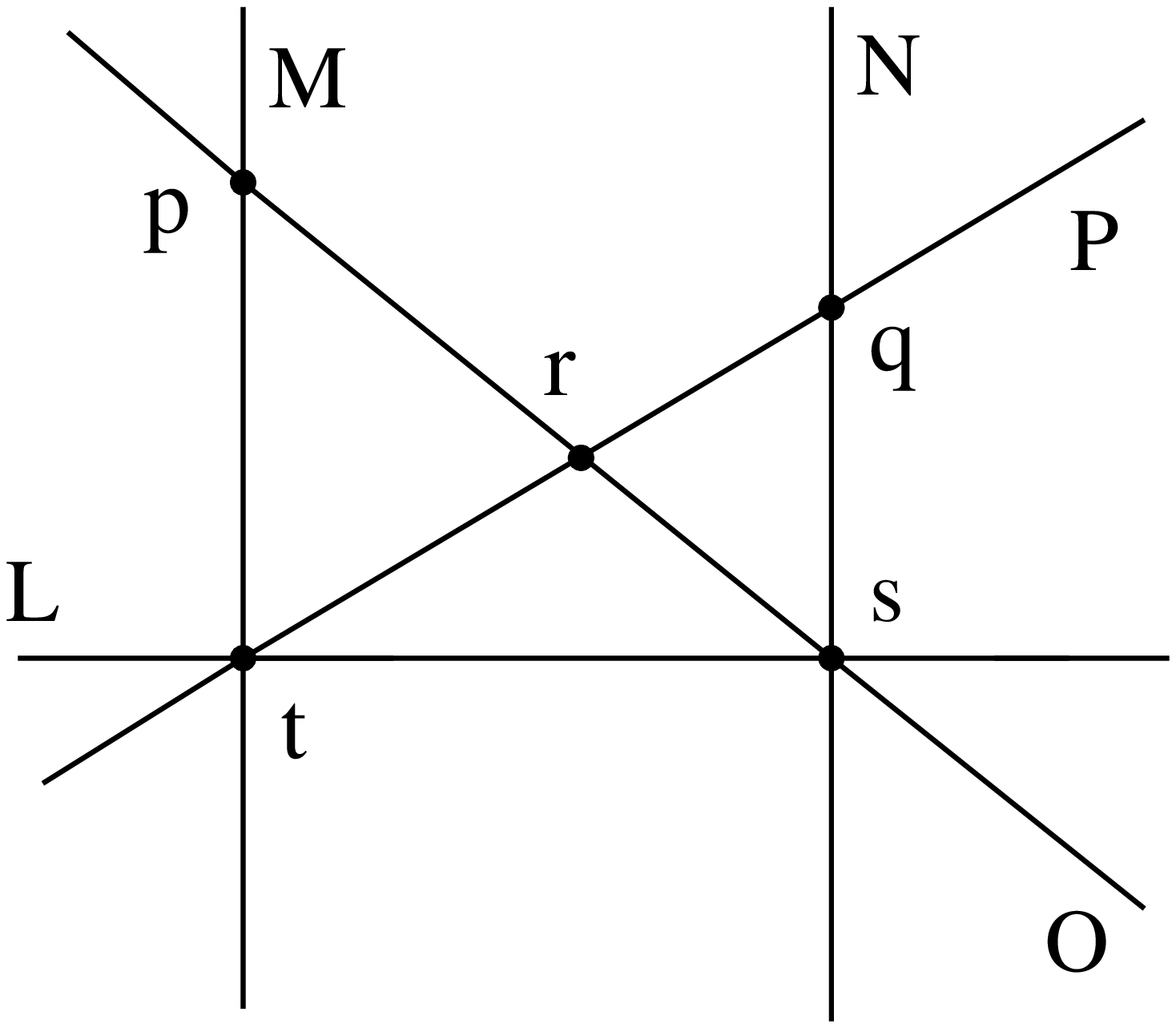}
\end{center}

\begin{proof}
  By Proposition I.12, let $M$ be a line through $p$ perpendicular
  to $L$, intersecting $L$ at $e$.\\
  By Proposition I.12, let $N$ be a line through $q$ perpendicular
  to $L$, intersecting $L$ at $f$.\\
  Suppose $e \ne f$.\\
  \hspace*{10pt} Hence $M$ and $N$ are parallel.\\
  \hspace*{10pt} Let $s = f$.\\
  \hspace*{10pt} Let $t = e$.\\
  \hspace*{10pt} Let $O$ be the line through $p$ and $s$.\\
  \hspace*{10pt} Let $P$ be the line through $q$ and $t$.\\ 
  \hspace*{10pt} Let $r$ be the intersection of $O$ and $P$.\\
  \hspace*{10pt} Then $r, s, t$ satisfy 1 and 2.\\
  Suppose $e = f$.\\
  \hspace*{10pt} Let $s$ be a point on $L$ distinct from $e$.\\
  \hspace*{10pt} Let $t$ be a point on $L$ extending the segment from
  $s$ to $e$.\\
  \hspace*{10pt} Let $O$ be the line through $p$ and $s$.\\
  \hspace*{10pt} Let $P$ be the line through $q$ and $t$.\\ 
  \hspace*{10pt} Let $r$ be the intersection of $O$ and $P$.\\
  \hspace*{10pt} Then $r, s, t$ satisfy 1 and 2.\\
  Q.E.F.
\end{proof}

In the first case, a diagram inference tells us that $p$ and $t$ are
on the same side of $M$ (since otherwise $N$ and $M$ would intersect).
A triple-incidence rule, applied to $L$, $M$, and $N$ then tells us
that $q$ and $t$ are on opposite sides of $O$, which licenses the fact
that $O$ and $P$ intersect. The second case actually requires a case
distinction on the position of $p$ and $q$ along the perpendicular, at
which point, the Pasch rules provide enough information to license the
fact that $O$ and $P$ intersect.

\begin{Tech2}
\label{points_to_same}
\ \\
Assume line $L$ and points $p,q,r,s,t$ satisfy
the conclusions of the previous proposition.\\
Then $p$ and $q$ are on the same side of $L$.
\end{Tech2}

In fact, this is a direct diagrammatic inference, using the Pasch
rules. 

\section{Completeness}
\label{completeness:section}

In this section, we sketch a proof that $\na{E}$ is complete for a modern
semantics appropriate to the \emph{Elements}. This semantics is
presented in Section~\ref{semantics:section}, and the completeness
proof is presented in
Sections~\ref{tarski:section}--\ref{last:section}. 

\subsection{The semantics of ruler-and-compass constructions}
\label{semantics:section}

Thanks to Descartes, Euclid's points, lines, and circles can be
interpreted, in modern terms, as points, lines, and circles of the
Euclidean plane, $\RR \times \RR$. It is straightforward to show that
all the constructions and inference rules of $\na{E}$ are valid for
this semantics. $\na{E}$ is not, however, complete for this semantics:
all of Euclid's constructions, and hence all constructions of
$\na{E}$, can be carried out with a ruler and compass, and Galois
theory tells us that no ruler-and-compass construction can trisect a
sixty degree angle \cite[p.~240]{hungerford:74}. In particular,
$\na{E}$ cannot prove that there exists an equilateral triangle and a
trisection of one of its angles. The negation of this statement is a
universal statement, and so can also be expressed in $\na{E}$. This
shows that there is an existential statement that can neither be
proved nor refuted in $\na{E}$, showing that $\na{E}$ is not
syntactically complete, either.

Fortunately, there is a better semantics for the \emph{Elements}. An
ordered field is said to be \emph{Euclidean} if every nonnegative
element has a square root. Taking square roots essentially allows one
to construct the intersection of a line and a circle, and
conversely. Say that a sequent of $\na{E}$ is \emph{valid for ruler
  and compass constructions} if its universal closure is true in every
plane $F \times F$, where $F$ is a Euclidean field, under the usual
cartesian interpretation of the primitives of $\na{E}$. Our goal in
this section is to outline a proof of the following:

\begin{theorem}
\label{completeness:theorem}
  A sequent $\Gamma \Rightarrow \ex{\vec x.}
  \Delta$ is valid for ruler-and-compass constructions if and only if
  it is provable in $\na{E}$.
\end{theorem}

Once again, the ``if'' direction, asserting that $\na{E}$ is sound for
ruler-and-compass constructions, is straightforward. We will therefore
focus on establishing completeness. A direct proof would involve
assuming that a given sequent is not provable in $\na{E}$, and then
constructing a Euclidean field in which that sequent is false. But
given $\na{E}$'s restricted logic, the details would be tricky, and
our job will be much easier if we build on previous work.
Tarski~\cite{tarski-whatis} gave a sound and complete axiomatization
not only of the full Euclidean plane, but also of the fragment that is
valid for ruler-and-compass constructions. It is therefore sufficient
to show that $\na{E}$ is complete with respect to Tarski's
axiomatization of the latter.

There are, however, obstacles to this approach. For one thing,
Tarski's axiomatization of geometry uses only one sort, namely points,
and two primitives, for betweenness and equidistance, as described
below. So interpreting statements of $\na{E}$ in Tarski's system and
vice-versa involves a change of language. A more serious obstacle is
that Tarski uses full first-order logic, in contrast to the very
meager fragment that is allowed in $\na{E}$. So knowing that a
statement is provable in Tarski's system is not \emph{a priori}
helpful, since there will generally be no line-by-line interpretation
of this proof in $\na{E}$.

Below, however, we will show that with a modicum of tinkering,
Tarski's axioms can be expressed in a restricted form, namely, as a
system of \emph{geometric} rules. We will then invoke a cut
elimination theorem, due to Sara Negri, that shows that if a sequent
of suitably restricted complexity is provable in the system, there is
a proof in which every intermediate sequent is also of restricted
complexity. This will allow us to translate proofs in Tarski's system
to proofs in $\na{E}$.

\newcommand{\Eu}{\na{E}}
\newcommand{\Tarski}{\na{T}}

More precisely, we will craft a slight variant, $\na{T}$, of Tarski's
system, which is sound and complete for ruler-and-compass
constructions, and enjoys some nice proof-theoretic properties. We
will define a translation $\pi$ from sequents of $\na{E}$ to sequents
of $\na{T}$, and a re-translation $\rho$ in the other direction.
Ultimately, we will show that the systems and translations involved
have the following properties:
\begin{enumerate}
\item If $\Gamma\fCenter \ex{\vec{x}.} \Delta$ is valid for ruler and
  compass constructions, then $\Tarski$ proves $\pi(\Gamma\fCenter
  \ex{\vec{x}.} \Delta)$.
\item If $T$ proves $\pi(\Gamma\fCenter \ex{\vec{x}.} \Delta)$, then
  $\Eu$ proves $\rho(\pi(\Gamma\fCenter \ex{\vec{x}.} \Delta))$.
\item If $\Eu$ proves $\rho(\pi(\Gamma\fCenter \ex{\vec{x}.}
  \Delta))$, then $\Eu$ proves $\Gamma\fCenter \ex{\vec{x}.} \Delta$.
\end{enumerate}
This yields the desired completeness result. Since many of the details
are straightforward, we will be somewhat sketchy; additional
information can be found in Dean's MS thesis \cite{dean:08}.

In fact, we will not interpret the area (``$\area$'') function of
$\na{E}$ or the functions and relations on the area sort; so we only
establish completeness for theorems that do not involve areas. Defining an
adequate notion of area in Tarski's system requires a fair amount of
work, although by now the mechanisms for doing so are well understood
(see, for example, Hilbert \cite[Chapter IV]{hilbert:99}). We are
confident that the methods described here extend straightforwardly to
cover areas as well, but spelling out the details would require more
effort.

\subsection{Tarski's system}
\label{tarski:section}

Tarski's axiomatization of the ruler-and-compass fragment of Euclidean
geometry employs the language, $\mathcal{L}$, whose only nonlogical
predicates are a ternary predicate, $B$, where $B(abc)$ is intended to
denote that $a$, $b$, and $c$ are collinear and $b$ is between $a$ and
$c$; and a four-place relation, $\equiv$, where $ab \equiv cd$ is
intended to denote that segment $ab$ is congruent to segment $cd$. (In
contrast to the ``$\mybetween$'' predicate of $\na{E}$, Tarski's $B$
denotes nonstrict betweenness.) The axioms consist of (the universal
closures of) the following (see, e.g. \cite{tarski-givant}):
\begin{enumerate}
\item Equidistance axiom (\Esymm): $ab \equiv ba$
\item Equidistance axiom (\Etrans): $(ab\equiv pq) \land
  (ab\equiv rs) \limplies (pq\equiv rs)$
\item Equidistance axiom (\Eid): $(ab\equiv cc) \limplies a=b$
\item Betweenness axiom (\Betw): $B(abd) \land B(bcd) \limplies B(abc)$
\item Segment Construction Axiom (\Seg): $\ex x (B(qax) \land
  (ax\equiv bc))$
\item Five-Segment Axiom (\Five):
\[
\begin{split} [\neg(a= b) \wedge B(abc) \wedge B(pqr) \wedge (ab\equiv
  pq) \wedge (bc\equiv qr) \wedge \\
(ad\equiv ps) \wedge (bd\equiv qs)]
  \rightarrow (cd\equiv rs)
\end{split}
\]
\item Pasch Axiom (\Pasch): $B(apc) \wedge B(qcb) \rightarrow
                \ex x (B(axq) \wedge B(bpx))$
\item Lower 2-Dimension Axiom (\LD): $\ex {a,b,c} [\neg B(abc) \wedge
  \neg B(bca) \wedge \neg B(cab)]$
\item Upper 2-Dimension Axiom (\UD): $\neg(a=b) \wedge
  \bigwedge_{i=1}^3 x_ia\equiv x_ib \rightarrow (B(x_1x_2x_3)\vee
  B(x_2x_3x_1)\vee B(x_3x_1x_2))$
\item Parallel Postulate (\Euc): $B(adt) \wedge B(bdc) \wedge
  \neg(a=d) \rightarrow \ex {x,y} (B(abx) \wedge B(acy) \wedge B(ytx))$
\item Intersection Axiom (\Int): $(ax\equiv ax') \wedge (az\equiv az')
  \wedge B(axz) \wedge B(xyz) \rightarrow \ex {y'} ((ay\equiv ay')
  \wedge B(x'y'z'))$
\end{enumerate}
Intuitively, the last axiom says that any line through a point lying
inside a circle intersects the circle. Tarski showed that when one
replaces this axiom with the Continuity Axiom Scheme,
\[
\ex a \fa {x,y} (\varphi(x)\wedge\psi(y) \rightarrow B(axy))
\rightarrow \ex b \fa {x,y} (\varphi(x)\wedge\psi(y) \rightarrow
B(xby))
\]
the result is complete for the semantics of the full Euclidean
plane. But he also showed that axioms 1--11 are complete for
ruler-and-compass constructions, and it is this result that is
important for our purposes.\footnote{\label{ziegler_footnote}Note that
  the system for ruler-and-compass constructions is finitely
  axiomatized, in contrast to the stronger system with the Continuity
  Axiom Scheme.  Ziegler \cite{ziegler} proved that any finitely
  axiomatizable theory of fields that has among its models an
  algebraically closed field, a real closed field or a field of
  $p$-adic numbers, is an undecidable theory.  It is clear from the
  present result that the formal system for ruler-and-compass
  constructions has a real closed field among its models (since a real
  closed field is, \emph{a fortiori}, Euclidean). Thus the system is
  undecidable.}

\begin{theorem}[Tarski]
\label{tarski:completeness:theorem}
If $\ph$ is valid for ruler-and-compass constructions, then $\ph$ is a
first-order consequence of the axioms above.
\end{theorem}

We will now fashion a variant of this system with better
proof-theoretic properties.  A theory is called \emph{geometric} if
all of its axioms are sentences of the following form:
\[ 
(\star) \quad\quad \fa{\vec{x}}\left[\bigwedge_{i=1}^m B_i(\vec{x})
  \rightarrow \bigvee_{j=1}^n \left(\exists \vec{y}_j
    \bigwedge_{k=1}^{\ell_j}
    A_{j,k}(\vec{x},\vec{y}_j)\right)\right],
\]
where the $A$'s and $B$'s are atomic formulas (including $\top$ and
$\bot$), and each of $\vec{x}$, $\vec{y}$ or the antecedent of the
conditional could be empty.  Formulas of the form $(\star)$ are called
\emph{geometric}. Those geometric formulas with only a single disjunct
in the consequent (i.e. geometric formulas in which $\vee$ does not
appear) are called \emph{regular}. Note that, on our modeling,
Euclid's propositions are almost of this latter form, the difference
being that arbitrary literals (negated atomic formulas as well as
atomic formulas) are allowed in the antecedent and consequent.

Sara Negri \cite{negri-main}, building on earlier joint work with Jan
von Plato \cite{negri-plato-main}, has established a cut-elimination
theorem for geometric theories that we can put to use in our
completeness proof.  Suppose we have a geometric theory formulated in
a standard two-sided sequent calculus (see, for example
\cite{buss:98e,troelstra:schwichtenberg:00}). Then the theory can be
recast equivalently by replacing each of its geometric axioms like the
one above with a corresponding inference rule, called a
\emph{geometric rule scheme} (GRS):
\begin{center}
\begin{prooftree}
  \AXN{\vec{A}_{1,\cdot}(\vec{x},\vec{y}_1),\Pi \fCenter \Theta}
  \AXM{\cdots}
  \AXN{\vec{A}_{n,\cdot}(\vec{x},\vec{y}_n),\Pi \fCenter
    \Theta}
  \TIN{\vec{B}(\vec{x}),\Pi \fCenter \Theta}
\end{prooftree}
\end{center}
Here we assume that the variables among the $\vec{y}_j$'s do not
appear free in $\vec{B}$, $\Pi$ or $\Theta$.\footnote{If one
  represents sequents using sequences or multisets of formulas, as
  Negri does, the rules must be presented with the $\vec{B}(\vec{x})$
  repeated in the premises in order for Negri to prove the
  admissibility of the structural rules of contraction and weakening,
  along with cut-elimination. Taking $\Pi$ and $\Theta$ to be sets is
  notationally simpler and suffices for our purposes.}  Negri's
principal result is the following theorem, whose corollary we will
apply later.
\begin{theorem}
\label{negri:theorem}
Any sequent provable in a sequent calculus with geometric rule schemes
has a cut-free proof.
\end{theorem}
Since the cut rule is the only rule that removes formulas, this shows
that if a sequent $\Pi \Rightarrow \Theta$ is provable in such a
system, there is a proof that mentions only subformulas of formulas in
$\Pi$ and $\Theta$, and possibly some other \emph{atomic} formulas.

Say a sequent $\Pi \Rightarrow \Theta$ is \emph{geometric} if $\Pi$ is
a set of atomic formulas and $\Theta$ is a finite set of existentially
quantified conjunctions of atomic formulas. In other words, a
geometric sequent is a representation of a geometric formula where the
implication is replaced by the sequent arrow and the outer universal
quantifiers are left implicit. Say a geometric sequent is
\emph{regular} if $\Theta$ consists of at most one formula.
Theorem~\ref{negri:theorem} implies that if we are working in a
sequent calculus with geometric rule schemes, then any provable
geometric sequent has a proof in which every sequent is geometric;
and, similarly, any provable regular sequent has a proof in which
every sequent is regular.

Tarski's axiomatization for the ruler-and-compass constructions is
\emph{nearly} geometric.  The only stumbling block is that in
$(\star)$ the conjunctions are required to be conjunctions of atomic
formulas, not literals.  Thus, for instance, the lower
2-dimensional axiom
\[ 
\ex {a,b,c} (\neg B(abc) \wedge \neg B(bca) \wedge \neg B(cab)) 
\]
is not geometric.  We remedy this situation by introducing explicit
predicates for the negations of $=$ and $B$ and $\equiv$; that is, we
expand our language to one called $\mathcal{L}(\Tarski)$ by adding
predicates $\ne$ and $\overline{B}$ and $\not\equiv$; and we add the
(geometric) axioms
\begin{itemize}
\item $\fa {x,y} ((x=y)\vee(x\ne y))$
\item $\fa {x,y} ((x=y)\wedge(x\ne y) \rightarrow {\perp})$
\end{itemize}
as well as analogous ones for $B,\overline{B}$ and
$\equiv,\not\equiv$. We will call these ``negativity axioms'' below.
Also, we replace any negated instances of $=$ or $B$ (there are no
such negated instances of $\equiv$) from Tarski's original
axiomatization with the new corresponding predicate, thus obtaining a
geometrically axiomatized theory.  

Notice that there is an obvious translation from the language
$\mathcal L(T)$ of $\na{T}$ to the language of Tarski's
system, which maps, e.g., occurrences of $\overline{B}(xyz)$ to $\neg
B(xyz)$, and so on. This translation preserves provability, since the
negativity axioms imply that the new predicates behave like
negations. We now go further and put the nonlogical axioms of
$\Tarski$ into the form of geometric rule schemes.  First of all, the
negativity axioms look like this:
\begin{center}
\fontsize{10}{10}\selectfont
\begin{prooftree}
  \AXN{(x=y),\Pi \fCenter \Theta}
  \AXN{(x\ne y),\Pi \fCenter \Theta}
  \RL{\NNeg}
  \BIN{\Pi \fCenter \Theta}
\end{prooftree}

\begin{prooftree}
  \AXN{{\perp},\Pi \fCenter \Theta}
  \RL{\NNeg}
  \UIN{(x=y),(x\ne y),\Pi \fCenter \Theta}
\end{prooftree}
\end{center}
and similarly for the other predicates. The remaining rules are as follows
(and note that variables appearing in parentheses next to the rule names
are those which are not allowed to appear free in the conclusion):
\begin{center}
\begin{prooftree}
  \AXN{ab\equiv ba,\Pi \fCenter \Theta}
  \RL{\Esymm}
  \UIN{\Pi \fCenter \Theta}
\end{prooftree}

\begin{prooftree}
  \AXN{(pq\equiv rs),\Pi \fCenter
    \Theta}
  \RL{\Etrans}
  \UIN{(ab\equiv pq),(ab\equiv rs),\Pi \fCenter \Theta}
\end{prooftree}

\begin{prooftree}
  \AXN{(a=b),\Pi \fCenter \Theta}
  \RL{\Eid}
  \UIN{(ab\equiv cc),\Pi \fCenter \Theta}
\end{prooftree}

\begin{prooftree}
  \AXN{B(abc),\Pi \fCenter \Theta}
  \RL{\Betw}
  \UIN{B(abd),B(bcd),\Pi \fCenter \Theta}
\end{prooftree}

\begin{prooftree}
  \AXN{B(qax),(ax\equiv bc),\Pi \fCenter \Theta}
  \RL{\Seg(x)}
  \UIN{\Pi \fCenter \Theta}
\end{prooftree}

\begin{prooftree}
  \AXN{(cd\equiv rs),\Pi \fCenter \Theta}
  \RL{\Five}
  \UIN{a\ne b,B(abc),B(pqr),(ab\equiv pq),(bc\equiv qr),(ad\equiv
    ps),(bd\equiv qs),\Pi \fCenter \Theta}
\end{prooftree}

\begin{prooftree}
  \AXN{B(axq),B(bpx),\Pi \fCenter \Theta}
  \RL{\Pasch(x)}
  \UIN{B(apc),B(qcb),\Pi \fCenter \Theta}
\end{prooftree}

\begin{prooftree}
  \AXN{\overline{B}(abc),\overline{B}(bca),\overline{B}(cab),\Pi \fCenter \Theta}
  \RL{\LD(a,b,c)}
  \UIN{\Pi \fCenter \Theta}
\end{prooftree}

\begin{prooftree}
\AXN{B(x_1x_2x_3),\Pi \fCenter \Theta}
\AXN{B(x_2x_3x_1),\Pi \fCenter \Theta}
\AXN{B(x_3x_1x_2),\Pi \fCenter \Theta}
\RL{\UD}
\TIN{a\ne b,(x_1a\equiv x_1b),(x_2a\equiv x_2b),(x_3a\equiv x_3b),\Pi \fCenter \Theta}
\end{prooftree}

\begin{prooftree}
\AXN{B(abx),B(acy),B(ytx),\Pi \fCenter \Theta}
\RL{\Euc(x,y)}
\UIN{B(adt),B(bdc),a\ne d,\Pi \fCenter \Theta}
\end{prooftree}

\begin{prooftree}
\AXN{(ay\equiv ay'),B(x'y'z'),\Pi \fCenter \Theta}
\RL{\Int(y')}
\UIN{(ax\equiv ax'),(az\equiv az'),B(axz),B(xyz),\Pi \fCenter
  \Theta}
\end{prooftree}

\end{center}

Since the resulting system is just a reworking of Tarski's
axiomatization, combining Theorem~\ref{tarski:completeness:theorem}
with Negri's Theorem~\ref{negri:theorem} yields the following:

\begin{lemma}
\label{tarski:variant:complete}
Let $\Pi \Rightarrow \Theta$ be a geometric sequent in the language
of $\na{T}$ that is valid for ruler-and-compass constructions. Then
$\Pi \Rightarrow \Theta$ has a cut-free proof in $\na{T}$.
\end{lemma}

\subsection{Translating $\na{E}$ to $\na{T}$}
\label{translating:section}
Our goal now is to provide a translation $\pi$ that maps any sequent
$\Gamma \fCenter \ex{\vec x.} \Delta$ of $\na{E}$ to a geometric (in
fact, regular) sequent $\Pi \fCenter \Theta$ of $\na{T}$, with
the following properties:
\begin{itemize}
\item The translation preserves ruler-and-compass semantics, so that if
  $\Gamma \fCenter \ex{\vec x.} \Delta$ is valid for ruler-and-compass
  constructions, so is $\Pi \fCenter \Theta$.
\item Conversely, the existence of a cut-free proof of $\Pi
  \fCenter \Theta$ in $\na{T}$ implies the existence of a
  proof of $\Gamma \fCenter \ex{\vec x.} \Delta$ in $\na{E}$. 
\end{itemize}
In this section we will define the translation and show that it
satisfies the first property. The second property is then established
in Section~\ref{last:section} below.

In carrying out the translation, we will represent each line $L$ of
$\na{E}$ by distinct points $\mathfrak{c}_1^L,\mathfrak{c}_2^L$ that
are assumed to lie on $L$. Similarly, we will represent each circle
$\gamma$ of $\na{E}$ by its center, $\mathfrak{c}_1^\gamma$, and a point,
$\mathfrak{c}_2^\gamma$, that is assumed to lie on $\gamma$. More
precisely, given any sequent $\Gamma \Rightarrow \ex{\vec x.} \Delta$
of $\na{E}$, we will choose fresh variables
$\mathfrak{c}_1^L,\mathfrak{c}_2^L$ for each line variable $L$
occurring in the sequent, and fresh variables
$\mathfrak{c}_1^\gamma,\mathfrak{c}_2^\gamma$ for each circle variable
$\gamma$. Let $\hat \Delta$ consist of the assumptions
\[
\{ \mathfrak{c}_1^L \neq \mathfrak{c}_2^L, \on(\mathfrak{c}_1^L,L), 
\on(\mathfrak{c}_2^L,L) \}
\]
for each line variable $L$ among $\vec
x$, and the assumptions 
\[
\{ \mycenter(\mathfrak{c}_1^\gamma,\gamma), \on(\mathfrak{c}_2^\gamma,
\gamma) \}
\]
for each circle variable $\gamma$ among $\vec x$. (Note that, in
$\na{E}$, $\mathfrak{c}_1^\gamma \neq \mathfrak{c}_2^\gamma$ is a
consequence of the latter set of assertions.) Let $\hat \Gamma$
consist of the assumptions corresponding to the remaining line and
circle variables in the sequent. Then clearly $\Gamma \fCenter \ex
{\vec x.} \Delta$ is provable in $\na{E}$ if and only if $\Gamma, \hat
\Gamma \fCenter \ex {\vec x, \vec{\mathfrak{c}}.} \Delta, \hat \Delta$
is; and one is valid if and only if the other is valid as well.  When
we translate $\Gamma \fCenter \ex {\vec x.} \Delta$ to the language of
$\na{T}$, we will use these new variables, and the translations will
make sense as long as we assume $\mathfrak{c}_1^L \neq
\mathfrak{c}_2^L$ and $\mathfrak{c}_1^\gamma \neq
\mathfrak{c}_2^\gamma$ for the relevant constants. When we translate
back, we will add the assumptions in $\hat \Gamma, \hat \Delta$, which
will make it possible for $\na{E}$ to show that the result is
equivalent to the original sequent.

To define $\pi$, first, for each $\Eu$-literal $A$ we will define a
corresponding $\mathcal{L}(\Tarski)$-formula $\overline{\pi}(A)$ of
the following form:
\[ \exists \vec{z} \left(\bigwedge_k M_k(\vec{z})\right) \] where the
$M_k$'s are atomic. (Formulas of this form are sometimes referred to
as \emph{positive primitive} formulas.) We will occasionally abuse
notation below and write $\overline{\pi}(A)$ for the conjunction
$\bigwedge_k M_k(\vec{z})$ without the existential quantifiers out
front.  Furthermore, if we have a set of literals $A_1,\dots,A_m$ and
\[ \overline{\pi}(A_i) = \exists \vec{z}_i \left(\bigwedge_{k=1}^{n_i}
  M_{i,k}(\vec{z}_i)\right)
\]
for each $i$, we will sometimes write $\overline{\pi}(A_1,\dots,A_m)$
to refer to
\[ \exists \vec{z}_1,\dots,\vec{z}_m \bigwedge_{i=1}^m
\bigwedge_{k=1}^{n_i} M_{i,k}(\vec{z}_i). \] We do so for the sake of
perspicuity and simple readability.  When making such abuses, we will
call attention to the fact that we are doing so, and no confusion
should arise.

\showdiagram{
\begin{figure}

\begin{center}
\psfrag{a}{$a$}\psfrag{b}{$b$}\psfrag{N}{$N$}\psfrag{p}{$p$}\psfrag{e}{$\mathfrak{c}_1^N$}\psfrag{d}{$\mathfrak{c}_2^N$}

\includegraphics[height=2.5cm]{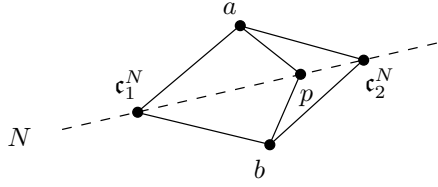}
\caption{The translation of $\on(p,N)$}
\label{rho:figure}
\end{center}

\end{figure}
}

In each case, our translation provides a natural way of expressing the
corresponding literal of $\na{E}$ as a formula of the desired form,
though some thought (and a diagram) is often needed to make sense of
it. For example, the translation of $\on(p,N)$ is illustrated by
Figure~\ref{rho:figure}. For the diagrammatic assertions, the clauses
of the translation are as follows.
\begin{itemize}
\item $\OOn(p,N) \,\,\mapsto\,\, \exists a,b
  (\underbrace{a\ne b \,\wedge\, \mathfrak{c}_1^Na\equiv\mathfrak{c}_1^Nb \,\wedge\,
    \mathfrak{c}_2^Na\equiv\mathfrak{c}_2^Nb \,\wedge\, pa\equiv
    pb}_{=: \,\, \zeta(\mathfrak{c}_1^N, \mathfrak{c}_2^N,p,a,b)})$.

\item $\neg\OOn(p,N) \,\,\mapsto\,\,
  \underbrace{\overline{B}(\mathfrak{c}_1^N \mathfrak{c}_2^N p) \wedge
    \overline{B}(\mathfrak{c}_1^N p \mathfrak{c}_2^N) \wedge
    \overline{B}(p \mathfrak{c}_1^N \mathfrak{c}_2^N)}_{=: \,\,
    \chi(\mathfrak{c}_1^N, \mathfrak{c}_2^N,p)}$.

\item $\SSide(p,q,N) \,\,\mapsto\,\,$ {\fontsize{10}{10}\selectfont
    \[ \exists r,s,t,a,b
    (\zeta(\mathfrak{c}_1^N,\mathfrak{c}_2^N,s,a,b) \wedge
    \zeta(\mathfrak{c}_1^N,\mathfrak{c}_2^N,t,a,b) \wedge
    \chi(\mathfrak{c}_1^N,\mathfrak{c}_2^N,r) \wedge B(prs) \wedge
    B(qrt)). \] }

\item $\neg\SSide(p,q,N) \,\,\mapsto\,\, \exists r,a,b
  (\zeta(\mathfrak{c}_1^N,\mathfrak{c}_2^N,r,a,b) \wedge B(prq))$.

\item $\Between(p,q,r) \,\,\mapsto\,\, B(pqr) \wedge p\ne q \wedge
  q\ne r \wedge p\ne r$.

\item $\neg\Between(p,q,r) \,\,\mapsto\,\,$
  {\fontsize{10}{10}\selectfont
    \[ \exists a,b,f,g,h,x,y,z \left[
      \begin{array}{l}
        \chi(a,b,q) \wedge a\ne p \wedge a\ne q \wedge a\ne r \wedge b\ne p \wedge b\ne q \wedge b\ne r \wedge \\
        B(apx) \wedge B(aqy) \wedge B(arz) \wedge p\ne x \wedge q\ne y \wedge r\ne z \wedge \\
        B(bpf) \wedge B(bqg) \wedge B(brh) \wedge p\ne f \wedge q\ne g \wedge r\ne h \wedge \\
        \overline{B}(xyz) \wedge \overline{B}(fgh)
      \end{array}
    \right]
    \]}

\item $\OnCirc(p,\gamma) \,\,\mapsto\,\, \mathfrak{c}_1^{\gamma} p
  \equiv \mathfrak{c}_1^{\gamma} \mathfrak{c}_2^{\gamma}$.

\item $\neg\OnCirc(p,\gamma) \,\,\mapsto\,\, \mathfrak{c}_1^{\gamma} p
  \not\equiv \mathfrak{c}_1^{\gamma} \mathfrak{c}_2^{\gamma}$.

\item $\Inside(p,\gamma) \,\,\mapsto\,\, \ex x
  (B(\mathfrak{c}_1^{\gamma} p x) \,\wedge\, p\ne x \,\wedge\,
  (\mathfrak{c}_1^{\gamma} x \equiv \mathfrak{c}_1^{\gamma}
  \mathfrak{c}_2^{\gamma}))$.

\item $\neg\Inside(p,\gamma) \,\,\mapsto\,\, \ex x
  (B(\mathfrak{c}_1^{\gamma} x p) \wedge (\mathfrak{c}_1^{\gamma} x
  \equiv \mathfrak{c}_1^{\gamma} \mathfrak{c}_2^{\gamma}))$.

\end{itemize}
These can be used to define equality and disequality for lines and
circles:
\begin{itemize}

\item $L = M \,\,\mapsto\,\, \on(\mathfrak{c}_1^L,M) \land
  \on(\mathfrak{c}_2^L,M)$.

\item $L \neq M \,\,\mapsto\,\, \ex x (\on(x,L) \land \lnot \on(x,M))$.

\item $\gamma = \delta \,\,\mapsto\,\, \mathfrak{c}_1^\gamma = 
  \mathfrak{c}_1^\delta \land \mathfrak{c}_1^{\gamma}
  \mathfrak{c}_2^{\gamma} \equiv \mathfrak{c}_1^{\delta}
  \mathfrak{c}_2^{\delta}$. 

\item $\gamma \neq \delta \,\,\mapsto\,\, \ex x (\on(x,\gamma) \land 
    \lnot \on(x,\delta))$.

\end{itemize}
More precisely, the translation involves expanding the
$\overline{\pi}$ images of the literals on the right-hand side, and
bringing the existential quantifiers to the front.

We have not yet indicated the $\overline{\pi}$-images for literals
involving the $\Intersect$ predicate.  The positive literals are
straightforwardly expressed in terms of literals that have already been
translated:
\begin{itemize}

\item $\Intersect(L,M) \,\,\mapsto\,\, L \neq M  \land \ex x (\on(x,L)
  \land \on(x,M))$.

\item $\Intersect(L,\gamma) \,\,\mapsto\,\, \ex {x,y} (x \neq y \land 
  \on(x,L) \land \on(x,\gamma) \land \on(y,L) \land \on(y,\gamma))$.

\item $\Intersect(\gamma,\delta) \,\,\mapsto\,\, \gamma \neq \delta
  \land \ex {x,y} (x \neq y \land \on(x,\gamma) \land \on(x,\delta)
  \land \on(y,\gamma) \land \on(y,\delta))$.

\end{itemize}
The negative literals, which assert nonintersection, require something
more roundabout.  For instance, we express the fact that $\alpha$ and
$\beta$ do not intersect by saying that the line segment from the
center of $\alpha$ to the center of $\beta$ encounters a point on
$\alpha$ strictly before a point on $\beta$:
\[
\neg\Intersect(\alpha,\beta) \,\,\mapsto\,\, \exists p,a,b \left[
  \begin{array}{l}
    \mathfrak{c}_1^\alpha\mathfrak{c}_2^\alpha \equiv \mathfrak{c}_1^\alpha a \,\wedge\,
    \mathfrak{c}_1^\beta\mathfrak{c}_2^\beta \equiv \mathfrak{c}_1^\beta b \,\wedge\,
    a\ne b \,\wedge\, \\
    B(\mathfrak{c}_1^\alpha ap) \,\wedge\, B(\mathfrak{c}_1^\beta bp) \,\wedge\, B(apb)
  \end{array}\right]
\]
Appropriate positive primitive $\overline{\pi}$-images for the
literals $\neg\Intersect(L,\alpha)$ and $\neg\Intersect(L,M)$ can be
found using $\overline{\pi}$-images from above, as well as the
translation for $\angle xyz = \rightangle$ which is given below. For
instance, to say that $\neg\Intersect(L,\alpha)$, we assert the
existence of points $a,b,c$, where $a$ is on $\alpha$, $b\ne c$ are on
$L$, $a$ is strictly between $\mathfrak{c}_1^\alpha$ and $b$, and
$\angle abc=\rightangle$.  Similarly, $\neg\Intersect(L,M)$ can be
expressed by asserting the existence of $a,b,c,d$, where $a\ne b$ are
on $L$, $c\ne d$ are on $M$, and the angles $\angle abc$ and $\angle
bcd$ are right angles.

The last type of literal to treat is that of metric assertions about
segments and angles.  Those for segments are more straightforward. Any
term of the segment sort will be of the form $\overline{p_1
  q_1}+\cdots+\overline{p_k q_k}$ (we can ignore occurrences of $0$;
the translation below also makes sense for ``empty sums''). Two such
sums are equal if the segments can be laid side by side along a line
so that the starting and ending points are the same. So, under our
translation,
\[ \overline{p_1 q_1}+\cdots+\overline{p_k q_k} =
	\overline{u_1 v_1}+\cdots+\overline{u_m v_m} \]
maps to
\[ \exists a_0\dots a_k,b_0\dots b_m \left[\begin{array}{l}
    B(a_0 a_1 a_2),B(a_1 a_2 a_3),\dots,B(a_{k-2} a_{k-1} a_k), \\
    B(b_0 b_1 b_2),B(b_1 b_2 b_3),\dots,B(b_{k-2} b_{k-1} b_k), \\
    (p_1 q_1\equiv a_0 a_1),(p_2 q_2 \equiv a_1 a_2),\dots, (p_k q_k \equiv a_{k-1}a_k), \\
    (u_1 v_1\equiv b_0 b_1),(u_2 v_2 \equiv b_1 b_2),\dots,(u_m v_m \equiv b_{m-1} b_m), \\
    a_0=b_0, a_k=b_m
	\end{array}\right]
\]
The translations of the other segment literals are obtained from this
one with minor changes to the last part.  Namely, the corresponding
translations are obtained by making the following indicated changes to the
last line above:
\begin{eqnarray*}
  \sum_i \overline{p_i q_i} \ne \sum_j \overline{u_j v_j} & \mapsto & a_0=b_0, a_k\ne b_m \\
  \sum_i \overline{p_i q_i} < \sum_j \overline{u_j v_j} & \mapsto & a_0=b_0, a_k\ne b_m, B(b_0,a_k,b_m) \\
  \sum_i \overline{p_i q_i} \not < \sum_j \overline{u_j v_j} & \mapsto & a_0=b_0, B(a_0 b_m a_k)
\end{eqnarray*}

For the angle literals, a little care is needed. First, note that we
can define equality and inequalities of angles as follows:
\begin{itemize}
\item $\angle xyz = \angle x'y'z' \,\,\mapsto\,\,$
  {\fontsize{8}{8}\selectfont
    \[ \exists u,v,u',v'(\underbrace{B(xuy) \wedge B(yvz) \wedge B(x'u'y') \wedge
    B(y'v'z') \wedge (uy \equiv u'y') \wedge (yv \equiv y'v')}_{=: \xi(x,y,z,x',y',z',u,v,u',v')}
    \wedge (uv\equiv u'v')).
    \]
  }

\item $\neg(\angle xyz = \angle x'y'z') \,\,\mapsto\,\,$
  {\fontsize{8}{8}\selectfont
    \[ \exists u,v,u',v' (\xi(x,y,z,x',y',z',u,v,u',v') \wedge
    (uv\not\equiv u'v')).
    \]
  }

\item $\angle xyz < \angle x'y'z' \,\,\mapsto\,\,$
  {\fontsize{8}{8}\selectfont
    \[ \exists u,v,u',v',a' (\xi(x,y,z,x',y',z',u,v,u',v') \wedge
    a'\ne v' \wedge B(u'a'v') \wedge
    (uv \equiv u'a'))
    \]
  }

\item $\neg (\angle xyz < \angle x'y'z') \,\,\mapsto\,\,$
  {\fontsize{8}{8}\selectfont
    \[ \exists u,v,u',v',a (\xi(x,y,z,x',y',z',u,v,u',v') \wedge
    B(uav) \wedge (ua \equiv u'v'))
    \]
  }
\end{itemize}
We can also say that an angle is a right angle:
\begin{itemize}
\item $\angle xyz = \rightangle \,\,\mapsto\,\,$
  {\fontsize{8}{8}\selectfont
    \[ \exists p,u,v,u',v' (x\ne y \wedge y\ne z \wedge p\ne y \wedge
    B(pyz) \wedge \xi(x,y,z,x,y,p,u,v,u',v') \wedge
    (uv \equiv u'v'))
    \]
  }
\end{itemize}
  At issue is how to compare sums of angles. Suppose we have two
  sums $\sum s_i$, $\sum t_i$ of angle terms.  In analogy to the
  segment case, we would like to take the various angles in two given
  sums, reconstruct them by ``stacking them up'' via a series of
  points around respective fixed vertices, and then compare the sums
  by measuring the resulting angles formed by the initial and final
  points.  The reason this can fail is that such a measure does not
  compare the sums themselves, but rather whether
\begin{multline*}
\min\left\{\sum s_i (\mbox{mod }2\pi), 2\pi-(\sum s_i (\mbox{mod
  }2\pi))\right\}
        = \\
\min\left\{\sum t_i (\mbox{mod }2\pi), 2\pi-(\sum t_i (\mbox{mod
  }2\pi))\right\},
\end{multline*}
so that unequal sums might be identified with one another.
(See Figure \ref{indiscernible:figure} for instance.)
\begin{figure}
\begin{center}
\psfrag{a}{$a_1$}\psfrag{b}{$a_2$}\psfrag{c}{$a_3$}\psfrag{d}{$a_4$}
\psfrag{e}{$s_1$}\psfrag{f}{$s_2$}\psfrag{g}{$s_3$}\psfrag{h}{$b$}
\psfrag{i}{$c_1$}\psfrag{j}{$c_2$}\psfrag{k}{$c_3$}\psfrag{l}{$c_4$}
\psfrag{m}{$c_5$}
\psfrag{n}{$t_1$}\psfrag{o}{$t_2$}\psfrag{p}{$t_3$}\psfrag{q}{$t_4$}
\psfrag{r}{$d$}
\includegraphics[height=2.5cm]{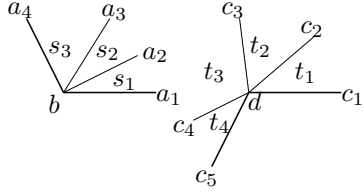}
\caption{Here $\angle a_1ba_4 = \angle c_1dc_5$, but $\sum s_i=2\pi/3$ while
        $\sum t_i=4\pi/3$.}
\label{indiscernible:figure}
\end{center}
\end{figure}

To remedy this, we do not stack the original angles. Instead, if
comparing a $k$-fold sum and an $m$-fold sum, we let $n = \max(k,m)$
and compare $n$-fold bisections of the summand angles.  The
point is that the resulting angles are guaranteed to be no greater
than the greatest of the original angles:
\[ \frac{1}{2^n}\sum_{i=1}^n \angle x_i y_i z_i \leq
	\frac{1}{2^n}(n \max_i\{\angle x_i y_i z_i\}) \leq
	\max\{\angle x_i y_i z_i\}.
\]
Thus our choice of taking $\max(k,m)$-fold bisections means that our
modified stacks all fit within one of the original angles from one of
the sums, and $\na{E}$'s setup guarantees that the term denotes an
angle less than or equal to $\pi$.  Thus we \emph{can} make the kind of
straightforward comparison of these shrunken stacks that we would
like.

Given that longwinded explanation, we will not spell out the
translation of the angle literals in detail, and will only briefly
indicate how one of them proceeds; the others result from minor
modifications of it, as with other groups of literals above.  First we
want an auxiliary $\na{T}$-formula which says ``$\angle p'q'r' =
(1/2^n)\angle pqr$,'' i.e. that the former is an $n$-fold bisection of
the latter. The following works:
\[ \exists a,b,a',b',u_1,\dots,u_n
	\left[\begin{array}{l}
		B(qap),B(qbr),B(q'a'p'),B(q'b'r'), \\
		B(a u_1 u_2),B(u_1 u_2 u_3),\dots,B(u_{n-2} u_{n-1} u_n),B(u_{n-1} u_n b), \\
		(\angle a'q'b' = \angle u_1qu_2),(\angle a'q'b'=\angle u_2qu_3),\dots,(\angle a'q'b'=\angle u_nqb), \\
		\angle aqu_1 = \angle u_1qb
	\end{array}\right]
\]
The translation of the literal
\[ \sum_{i=1}^k \angle x_i y_i z_i = \sum_{j=1}^m p_i q_i r_i \] would
then use the preceding formula, along with the formula $\xi$ from the
translations of the diagrammatic angle literals above, in order to
construct a positive primitive formula asserting the existence of two
stackings of $\max(k,m)$-fold bisections of the original angles which,
when compared in a similar fashion as the segment metric assertions
were, are seen to be equal.  The details are tedious to spell out, but
straightforward.

We now extend $\overline{\pi}$ to a translation $\pi :
\mathcal{L}(\Eu) \rightarrow \mathcal{L}(\Tarski)$ that maps every
sequent $\Gamma \fCenter \ex{\vec{x}.}\Delta$ of $\na{E}$ to a regular
  sequent of $\na{T}$. Suppose $\Gamma \fCenter \ex{\vec{x}.} \Delta$
  is of the form
\[ 
A_1,\dots,A_k \fCenter \ex{\vec{x}.} B_1,\dots,B_m, 
\]
where we have
\[ 
\overline{\pi}(A_i)=\exists\vec{z}_i\left(\bigwedge_{q=1}^{n_i}
  M_{i,q}\right), \,\,\,\,\,\,\,\,
\overline{\pi}(B_j)=\exists\vec{y}_j\left(\bigwedge_{r=1}^{p_j}
  N_{j,r}\right).
\]
Let $\Delta'$ consist of the assumption $\mathfrak{c}_1^L \neq
\mathfrak{c}_2^L$ for each line variable $L$ among $\vec x$, and the
assumption $\mathfrak{c}_1^\gamma \neq \mathfrak{c}_2^\gamma$ for each
circle variable $\gamma$ among $\vec x$. Let $\Gamma'$ consist of the
corresponding assumptions for the remaining line and circle variables
in the sequent. We define the image of this sequent, under $\pi$, to
be the regular sequent
\[ 
\Gamma', M_{1,1}, \ldots, M_{1,q_1}, \ldots M_{k,1}, \ldots, M_{k,q_k}
\fCenter \ex {\vec{x},\vec{y}_1,\dots,\vec{y}_m, \vec{\mathfrak c}}
\bigwedge \Delta' \land \bigwedge_{i=1}^m\left(\bigwedge_{r=1}^{p_i}
  N_{i,r}\right).\footnote{So, with our abuse of notation mentioned
  above, we could render this simply as
        \[ \Gamma',\overline{\pi}(A_1),\dots,\overline{\pi}(A_k) \fCenter
                \exists \vec{x},\vec{y}_1,\dots,\vec{y}_m \bigwedge
                \Delta' \wedge \bigwedge_{i=1}^m
                \overline{\pi}(B_i).
        \]}
\]
The following lemma captures all that we need to know about $\pi$.
\begin{lemma}
  \label{semantic_lemma}
  $\Gamma\fCenter \ex{\vec{x}.}\Delta$ is valid for ruler-and-compass
  constructions if and only if $\pi(\Gamma\fCenter
  \ex{\vec{x}.}\Delta)$ is.
\end{lemma}

Once we have crafted $\pi$ appropriately, the lemma is quite
straightforward to prove, given a precise articulation of the
cartesian interpretation of $\mathcal{L}(\na{E})$ and
$\mathcal{L}(\na{T})$ in the plane built on any Euclidean field.
Given the definition of $\pi$ in terms of $\overline{\pi}$, it
suffices to prove the result for sequents consisting of a single
literal; you can check that, for instance, the Technical Propositions
in Section~\ref{technical:section} prove the $\Rightarrow
\sameside(p,q,L)$ case (given the soundness of $\na{E}$).  Further
details can be found in \cite{dean:08}.

\subsection{Interpreting $\na{T}$ in $\na{E}$}
\label{last:section}

By Lemma~\ref{semantic_lemma}, we know that if a sequent $\Gamma
\fCenter \ex {\vec{x}.} \Delta$ in the language of $\na{E}$ is valid
for ruler-and-compass constructions, then so is $\pi(\Gamma \fCenter
\ex {\vec{x}.} \Delta)$. By Lemma~\ref{tarski:variant:complete}, this
implies that $\pi(\Gamma \fCenter \ex {\vec{x}.}  \Delta)$ has a
cut-free proof in $\na{T}$. All that remains is to define a mapping
$\rho$ from regular sequents in the language of $\na{T}$ to sequents
in the language of $\na{E}$, and show the following:
\begin{itemize}
\item If there is a cut-free proof of $\pi(\Gamma \fCenter \ex
  {\vec{x}.} \Delta)$ in $\na{T}$, then there is a proof of
  $\rho(\pi(\Gamma \fCenter \ex {\vec{x}.} \Delta))$ in $\na{E}$.
\item If there is a proof of $\rho(\pi(\Gamma \fCenter \ex {\vec{x}.}
  \Delta))$ in $\na{E}$, there is a proof of $\Gamma \fCenter \ex
  {\vec{x}.} \Delta$ in $\na{E}$.
\end{itemize}

Once again, we first define a translation $\overline{\rho}$ for individual
atomic formulas, and then extend the map to sequents. (And we will make
the same abuse of notation
below regarding $\rho$ as was noted for $\pi$.)
The atomic formulas are mapped as follows:
{\fontsize{10}{10}\selectfont
\begin{center}
\begin{tabular}{lcl}
  $B(pqr)$ & $\mapsto$ & $(\exists L,a,b).[a\ne b,a\ne p,a\ne q,a\ne r,b\ne p,b\ne q,b\ne r,$ \\
  & & $\OOn(a,L),\OOn(b,L),\OOn(p,L),\OOn(q,L),\OOn(r,L),\Between(a,q,b),$ \\
  & & $\neg\Between(a,q,p),\neg\Between(p,a,q),\neg\Between(q,b,r),$ \\
  & & $\neg\Between(r,q,b)]$ \\
  $\overline{B}(pqr)$ & $\mapsto$ & $\neg\Between(p,q,r),p\ne q,q\ne r$ \\
  $p=q$ & $\mapsto$ & $p=q$ \\
  $p\ne q$ & $\mapsto$ & $\neg(p=q)$ \\
  $xy\equiv vu$ & $\mapsto$ & $\overline{xy}=\overline{vu}$ \\
  $xy\not\equiv vu$ & $\mapsto$ & $\overline{xy}\ne\overline{vu}$
\end{tabular}

\end{center}
}
Why the first two are appropriate should be clear upon reflection (remembering that $\Between(p,q,r)$ is meant to be strict,
while $B(pqr)$ is not), and the others are obvious.

We now extend the map to sequents 
\[
P_1(\vec x), \ldots, P_n(\vec x) \fCenter 
\ex {\vec{y}} \left(\bigwedge_{j=1}^l Q_j(\vec{x},\vec{y})\right).
\]
Assuming each $P_i(\vec x)$ is mapped to $\ex {\vec z_i.} M_i(\vec x,
\vec z_i)$ by $\overline{\rho}$, where each $M_i$ is a set of literals,
and assuming each
$Q_j(\vec x, \vec y)$ is mapped to $\ex {\vec w_i.} N_j(\vec x, \vec
y, \vec z_j)$, the sequent above is mapped to the sequent
\[
M_1(\vec x, \vec z_1), \ldots, M_k(\vec x, \vec z_k) \fCenter
\ex {\vec y, \vec w_1, \ldots, \vec w_l.} N_1(\vec x, \vec y, \vec
z_1), \ldots, N_l(\vec x, \vec y, \vec z_l)
\]
of $\na{E}$.\footnote{Again, with abuse of notation this is just
        \[ \overline{\rho}(P_1),\dots,\overline{\rho}(P_n) \fCenter
                \exists \vec{y},\vec{w}_1,\dots,\vec{w}_l.
                \overline{\rho}(Q_1),\dots,\overline{\rho}(Q_l).
        \]}

We now proceed to establish the two properties indicated
above. The next lemma establishes the first.

\begin{lemma}
\label{rho:translation:lemma}
If there is a cut-free proof of the regular sequent
\[
P_1(\vec x), \ldots, P_n(\vec x) \fCenter 
\ex{\vec{y}} \left(\bigwedge_j Q_j(\vec{x},\vec{y})\right)
\]
in $\na{T}$, then there is a cut-free proof of its $\rho$ translation,
\[
M_1(\vec x, \vec z_1), \ldots, M_k(\vec x, \vec z_k) \fCenter
\ex {\vec y, \vec w_1, \ldots, \vec w_l.} N_1(\vec x, \vec y, \vec
z_1), \ldots, N_l(\vec x, \vec y, \vec z_l),
\]
in $\na{E}$.
\end{lemma}

\begin{proof}
        We proceed by induction on the proof in $\na{T}$.  We need to
        show that every inference of $\na{T}$ is mirrored by a proof
        in $\na{E}$.  The logical axioms and the logical rules which can
        appear in a cut-free proof of a regular sequent are already
        incorporated into the machinery of $\na{E}$:
        \begin{itemize}
                \item (Left/right conjunction rules).  We note that we do
                not have the symbol $\wedge$ in the language of $\na{E}$;
                instances of it get unpacked via the translation $\rho$.
                The left rules becomes vacuous, and the right
                rule is easily checked to be a derived rule of $\na{E}$
                (as an instance of theorem application).

                \item (Right exists rule). Similarly, uses of this
                  rule disappear in the translation.

                \item (Left falsum rules).  The effects of
                these rules are subsumed under $\na{E}$'s notion of direct
                consequence.

                \item (Negativity axioms).  Similarly straightforward.
        \end{itemize}

        We are left with the remaining GRS's from
        Section~\ref{tarski:section}.  With one exception, these are of the
        form
        \begin{center}
        \begin{prooftree}
                \AXN{A_1,\dots,A_n,\Pi \fCenter \Theta}
                \UIN{B_1,\dots,B_m,\Pi \fCenter \Theta}
        \end{prooftree}
        \end{center}
        which is to say, they correspond to the Tarskian axioms which are
        regular.  In these cases, it suffices by the induction hypothesis
        to show that $\na{E}$ proves
        \[ \overline{\rho}(B_1),\dots,\overline{\rho}(B_m) \fCenter
                \exists \vec{x}.
                \overline{\rho}(A_1),\dots,\overline{\rho}(A_n).
        \]
        Note that we are using the abuse of notation described in the
        last section. Checking the details of this for the various
        regular GRS's is pretty painless.  For instance:
        \begin{itemize}
                \item $(\Esymm,\Etrans,\Eid)$.  Given the trivial nature of
                $\overline{\rho}$ for $\equiv$ statements, it is easy to see
                that these cases are handled by our metric rules.

                \item $(\LD)$.  Let $a$ be a point.  Construct a point $b\ne a$.
                Construct line $L$ through $a,b$.  Construct a point $c$ that
                is not on $L$.  Each of $\Between(a,b,c)$ or $\Between(b,a,c)$
                or $\Between(a,c,b)$ leads to $\OOn(c,L)$, hence a
                contradiction.  Thus in $\na{E}$ we can conclude
                $\neg\Between$ for each.  One can check the definitions of
                $\LD$ and $\overline{\rho}$ to see that we have done what is
                needed.

              \item $(\Seg)$.  The Technical Propositions in
                Section~\ref{technical:section} provide the needed
                $\na{E}$ constructions here.

                \item We omit the remaining cases, some of which are slightly
                more involved, but none of which are interesting or enlightening.
        \end{itemize}
        All that remains is the sole GRS which is not regular, the upper
        two-dimensional axiom.  The situation is not really all that
        different from the regular cases; what we have to show, given the
        inductive hypothesis, is only slightly different.  

        The following suffices.  Suppose we have $a\ne b$, and
        $\overline{x_i a}= \overline{x_i b}$ for $i=1,2,3$.  We need
        $\na{E}$ to prove that two instances of
        $\neg\Between(x_i,x_j,x_k)$ hold.  We reason by cases;
        \emph{{\`a} la} Euclid we present only the case in which all
        the $x_i$ are distinct, as the other cases are only easier.

        For each $i$, construct circle $\gamma_i$ with center $x_i$,
        passing through $b$.  Construct line $L$ through $a,b$.  By
        Proposition I.12 (formalized in $\na{E}$ above), construct
        line $M$ perpendicular to $L$.  It is then a direct
        consequence that each $x_i$ is on $M$.

        Once again, we reason by cases, considering each parity for each
        possible $\Between(x_i,x_j,x_k)$; there are eight cases
        (omitting symmetry in the $\Between$ arguments).  In the four
        for which two positive $\Between$ relations were to hold,
        $\na{E}$ derives a contradiction.  In the other four cases, we
        have two negative instances, which is what we needed.
\end{proof}

Given the previous lemma, we are almost home.  We have shown that if
$\Gamma \fCenter \ex {\vec{x}.} \Delta$ is a valid sequent of
$\na{E}$, then there is a cut-free proof of $\pi(\Gamma \fCenter \ex
{\vec{x}.} \Delta)$ in $\na{T}$, and hence a proof of $\rho(\pi(\Gamma
\fCenter \ex {\vec{x}.} \Delta))$ in $\na{E}$. The trouble, of course,
is that $\rho(\pi(\Gamma \fCenter \ex{\vec{x}.} \Delta))$ is not quite
the same thing as $\Gamma \fCenter \ex{\vec{x}.}\Delta$. For one
thing, the lines and circles in the original sequent have been
replaced by pairs of points representing them; and the translated
sequent will typically feature extra points and hypotheses in both the
antecedent and consequent.  The next two lemmas demonstrate that, from
the $\Eu$ proof of the translated proposition, we can in fact recover
a proof of the original proposition, $\Gamma \fCenter
\ex{\vec{x}.}\Delta$.
\begin{lemma}
\label{aux:lemma}
        Let $M(\vec{x})$ be any literal of $\na{E}$.  Suppose that
        \[ \overline{\pi}(M) = \exists \vec{z} \bigwedge_{j=1}^m
                Q_j(\vec{x},\vec{z}), \]
        and further that
        \[ \overline{\rho}(Q_j) = \exists \vec{y}_j.
                A_{j,1},\dots,A_{j,n_j}. \]
Let $\hat \Theta$ consist of the assumptions 
\[
\{ \mathfrak{c}_1^L \neq \mathfrak{c}_2^L, \on(\mathfrak{c}_1^L,L),
\on(\mathfrak{c}_2^L,L) \}
\] for each line variable $L$ in $M$, 
and the assumptions 
\[
\{ \mycenter(\mathfrak{c}_1^\gamma,\gamma), \on(\mathfrak{c}_2^\gamma,
\gamma) \}
\] for each circle variable $\gamma$ in $M$.
        Then $\na{E}$ proves both
        \begin{itemize}
                \item[$(1)$] $\hat \Theta, M \fCenter \exists \vec{z},\vec{y}_1,\dots,\vec{y}_m.
                                A_{1,1},\dots,A_{1,n_1},\dots,A_{m,1},\dots,A_{m,n_m}$.

                \item[$(2)$] $\hat \Theta, A_{1,1},\dots,A_{1,n_1},\dots,A_{m,1},\dots,A_{m,n_m}
                                \fCenter \ex {\vec x.} \hat M$,
        \end{itemize}
        where $\vec x$ are the line and circle variables in
        $M$. Moreover, $\na{E}$ proves all sequents of the form
\[
\mathfrak{c}_1^L \neq \mathfrak{c}_2^L \fCenter \ex {L.} 
  \on(\mathfrak{c}_1^L,L), \on(\mathfrak{c}_2^L,L), 
\]
  and
\[
\mathfrak{c}_1^\gamma \neq \mathfrak{c}_2^\gamma \fCenter \ex {\gamma.
} \mycenter(\mathfrak{c}_1^\gamma,\gamma), \on(\mathfrak{c}_2^\gamma,
\gamma).
\]
\end{lemma}
Before getting to the proof, we note that clause (1) of the lemma just
says that $\na{E}$ proves $\hat \Theta, M \fCenter
\overline{\rho}(\overline{\pi}(M))$ for any literal.  Moreover, with
our abuse of notation we can render the second part more perspicuously
as asserting that $\na{E}$ proves $\hat \Theta,
\overline{\rho}(\overline{\pi}(M)) \fCenter M$.
\begin{proof}
        The last two claims in the lemma are immediate, using the
        construction rules of $\na{E}$. For the first two claims, 
        in order to avoid needless tedium, we indicate details for only a
        few cases (and also indicate how trivial some of the cases are).
        \begin{itemize}
                \item $(\Between(p,q,r))$.  We need to show that
                $\Between(p,q,r)$ is inter-derivable with
                \begin{multline*}
                \exists L. \OOn(p,L),\OOn(q,L),\OOn(r,L),
                        \neg\Between(p,r,q),\neg\Between(q,p,r), \\
                        p\ne q,
                        q\ne r,p\ne r.
                \end{multline*}
                Supposing the latter, we can conclude $\Between(p,q,r)$
                from the sixth betweenness rule.

                For the converse, suppose $\Between(p,q,r)$.
                A couple of applications of our first
                betweenness rule yield $\neg\Between(q,p,r)$,
                $\neg\Between(p,r,q)$ and the distinctness assertions.
                Construct line $L$ through $p,q$; $r$ is on $L$ as well,
                by the sixth and second betweenness rules.

              \item ($\OnCirc(p,\gamma)$ or $\neg\OnCirc(p,\gamma)$).
                This is immediate from the diagram-segment transfer
                axioms.

                \item ($\overline{xy}=\overline{zw}$ or $\overline{xy}\ne
                \overline{zw}$).  Similarly easy.

                \item ($\overline{xy}<\overline{zw}$).  In this case we
                need to show that the literal is inter-derivable with
                \begin{multline*}
                \exists a,L. \OOn(z,L),\OOn(a,L),\OOn(w,L),a\ne w,
                        z\ne w,\\ \neg\Between(a,z,w),\neg\Between(z,w,a),
                        \overline{xy}=\overline{za}.
                \end{multline*}
                Suppose the latter.  In case $z\ne a$, it follows that
                $\Between(z,a,w)$ (betweenness rule 6).  Then
                $\overline{za}+\overline{aw}=\overline{zw}$ (diagram-segment
                rule 1).  As $a\ne w$, $\overline{aw}>0$ (first metric
                inference).  By our linear arithmetic, then,
                $\overline{zw}>\overline{xy}$ as desired.  In the case
                $z=a$, we have $\overline{xy}=\overline{za}=0$ and
                $\overline{zw}=\overline{aw}$.  As $a\ne w$,
                $\overline{aw}>0$, so again we have $\overline{zw}>
                \overline{xy}$.

                Conversely, suppose $\overline{xy}<\overline{zw}$.  So
                $\overline{zw}>0$, hence $z\ne w$.  Construct line $L$
                through $z$ and $w$.  In case $x=y$, then $z$ itself will
                be our $a$.  In case $x\ne y$, apply Proposition I.2 to get
                a $b$ such that $\overline{xy}=\overline{zb}$.  Draw circle
                $\beta$ through $b$ centered at $z$.  As $z$ is inside
                $\beta$ and on $L$, we know that $\beta$ and line $L$
                intersect.  Since $\overline{zb}=\overline{xy}<\overline{zw}$,
                we know that $w$ lies outside $\beta$.  Thus we may take
                the intersection point $a$ of $\beta$ and $L$ such that
                $\Between(z,a,w)$ (by the fourth intersection construction
                rule).  This is the $a$ we need.

                \item ($\overline{xy}\not <\overline{zw}$).  Similar to the
                previous.
        \end{itemize}
\end{proof}

\begin{lemma}
  \label{main_lemma_2}
  If $\rho(\pi(\Gamma \fCenter \ex {\vec x.} \Delta))$ is provable in
  $\na{E}$, then so is $\Gamma \fCenter \ex {\vec x.}  \Delta$.
\end{lemma}

\begin{proof}
  Let $\hat \Gamma$ and $\hat \Delta$ be the sets of formulas
  described at the beginning of Section~\ref{translating:section}.
  Using our abuses of notation, our supposition is that $\na{E}$
  proves
  \[ \overline{\rho}(\overline{\pi}(\Gamma)) \fCenter \ex
  {\vec{z}.} \overline{\rho}(\overline{\pi}(\Delta)).
  \]
  Repeated application of clause (1) of Lemma~\ref{aux:lemma} shows
  that $\na{E}$ proves
  \[ \Gamma, \hat \Gamma \fCenter \ex {\vec u.}
  \overline{\rho}(\overline{\pi}(\Gamma)), \] where $\vec u$ are the
  new variables picked up in the translation. Using theorem
  application, $\na{E}$ proves
\[
  \Gamma, \hat \Gamma \fCenter \ex
  {\vec{z}, \vec{u}.} \overline{\rho}(\overline{\pi}(\Delta)).
\]
  The last part of Lemma~\ref{aux:lemma} shows that $\na{E}$ proves
\[
\Gamma, \hat \Gamma \fCenter \ex {\vec{z}, \vec{u}, \vec v.} \hat
\Delta, \overline{\rho}(\overline{\pi}(\Delta)),
\]
where $\vec v$ are the line and circle variables among $\vec x$, lost
in the translation back and forth, and now restored.  Clause (2) of
Lemma~\ref{aux:lemma} then shows that $\na{E}$ proves
\[
  \Gamma, \hat \Gamma \fCenter \ex {\vec z, \vec u, \vec v.} \hat \Delta, 
    \Delta.
\]
  Since the all the variables $\vec x$ are among $\vec z, \vec u, \vec
  v$, the sequent
\[
  \Gamma, \hat \Gamma \fCenter \ex {\vec x.} \Delta
\]
is subsumed by the previous one. Since $\na{E}$ can also prove $\Gamma
\fCenter \ex{\vec {\mathfrak{c}}}. \hat \Gamma$ for the new point
variables that occur in $\hat \Gamma$, it can prove
\[
  \Gamma \fCenter \ex {\vec x.} \Delta,
\]
  as required.
\end{proof}

Putting everything together, we have the proof of the completeness
theorem.

\begin{proof}[Proof of Theorem~\ref{completeness:theorem}]
  Suppose that $\Gamma\fCenter \ex{\vec{x}.} \Delta$ is valid for
  ruler-and-compass constructions. By Lemma \ref{semantic_lemma},
  $\pi(\Gamma\fCenter \ex{\vec{x}.} \Delta)$ is a valid sequent in the
  language of $\na{T}$. By Lemma~\ref{tarski:variant:complete}, there
  is a cut-free proof of that sequent in $\na{T}$. By
  Lemma~\ref{rho:translation:lemma}, $\rho(\pi(\Gamma\fCenter
  \ex{\vec{x}.} \Delta))$ is provable in $\na{E}$. By
  Lemma~\ref{main_lemma_2}, $\Gamma\fCenter \ex{\vec{x}.} \Delta$ is
  provable in $\na{E}$, as required.
\end{proof}

\section{Implementation}
\label{implementation:section}

In Section~\ref{direct:section}, we argued that the set of one-step
inferences in $\na{E}$ is decidable, as one would expect from any
formal system. But given the fact that we are trying to model the
inferential structure of the \emph{Elements}, there is the implicit
claim that verifying such inferences is within our cognitive
capabilities, at least at the scale of complexity found in the proofs
in the \emph{Elements}. ``Cognitively feasible'' does not always line
up with ``computationally feasible,'' and it is often quite
challenging to get computers to emulate common visual tasks. But, of
course, our case would be strengthened if we could show that our
inferences are computationally feasible as well.

In fact, our analysis should make it possible to design a
computational proof checker based on $\na{E}$ that takes, as input,
proofs that look like the ones in the \emph{Elements}, and verifies
their correctness against the rules of the system. In this section, we
describe some preliminary studies that suggest that general purpose
tools in automated reasoning are sufficient for the task.\footnote{As
  part of his MS thesis work at Carnegie Mellon, Benjamin Northrop has
  written code in Java that carries out diagrammatic inferences using
  an eager saturation method: whenever a new object is added to the
  diagram, the system closes the diagram under rules and derives
  \emph{all} the atomic and negation atomic consequences. The system
  works on small examples, but gets bogged down with diagrams of
  moderate complexity. But this does not rule out the fact that more
  sophisticated representations of the diagrammatic data might render
  such an approach viable. See the discussion later in this section.}

In Section~\ref{direct:section}, we noted that any fact obtained by a
direct diagram inference is contained in the set of first-order
consequences of the set of our universal axioms and the set of
literals constituting the diagram. Furthermore, there are no function
symbols in the language. These types of problems are fairly easy for
off-the-shelf theorem provers for first-order logic. We entered our
betweenness, same-side, and Pasch axioms in the standard TPTP format
(``Thousands of Problems for Theorem Provers,''), described a simple
diagram with five lines and six points, and checked a number of
consequences with the systems E \cite{schultz:02} (no relation to our
``E'') and Spass \cite{weidenbach:et:al:07}. The consequences were
verified instantaneously.

There is also a class of systems called ``satisfiability modulo
theories'' solvers, or SMT solvers for short, which combine decision
procedures for provability of universal sentences modulo the
combination of disjoint theories whose universal fragments are
decidable \cite{manna:zerba:03}. Such systems typically include very
fast decision procedures for linear arithmetic (that is, the linear
theory of the reals). This is particularly helpful to us, since our
metric inferences are of this sort.  Unfortunately, SMT solvers do not
provide complete decision procedures for the set of consequences of
arbitrary universal axioms, which is what is needed to verify our
diagrammatic and transfer inferences. Nonetheless, some solvers, like
Z3 \cite{demoura:bjorner:08} and CVC3 \cite{barrett:tinelli:07}
provide heuristic instantiation of quantifiers. The advantage to using
such systems is that they can handle not just the diagrammatic
inferences, but the metric and transfer inferences as well. We entered
all our axioms in the standard SMT format, and tested it with the two
systems just mentioned. The results were promising; most inferences
were instantaneous, and only a few required more than a few seconds.
The diagram, axioms, and test queries can be found online, at Avigad's
home page.

The fact that SMT solvers can handle arbitrary quantifier-free logic,
and the fact that one can incrementally add and retract statements
from the database of asserted facts, suggests that SMT solvers can
provide a complete back end to a proof checker for $\na{E}$. The proof
checker then need only parse an input proof, assert the relevant facts
to the SMT solver, and check the claimed consequences. More
specifically, when the user asserts a theorem, the proof checker
should declare the new objects (points, lines, and circles) to the SMT
solver, assert the assumptions to the SMT solver, and store the
conclusion. When the user enters a construction rule, the proof
checker should check that the prerequisites are consequences of the
facts already asserted to the SMT solver, create the new objects, and
assert their properties. Applying a previously proved theorem is
handled in a similar way. When a user enters ``hence $A$,'' the proof
checker should check that $A$ is a consequence of the facts already
asserted to the SMT solver, and, if so, assert it explicitly to the
SMT database, to facilitate subsequent inferences. To handle
suppositional reasoning (that is, proof by contradiction, or a branch
of a case split), the proof checker should ``push'' the state of the
SMT database and temporarily assert the local hypothesis, and then,
once the desired conclusion is verified, ``pop'' the state and assert
the resulting conditional. Finally, when the user enters ``Q.E.D.'' or
``Q.E.F.'', the proof checker need only check that the negation of the
theorem's conclusion is inconsistent with the facts that have been
asserted to the SMT solver.

Finally, we note that there has been recent work unifying resolution
and SMT frameworks, for example, with the Spass+T system
\cite{prevosto:waldmann:06}. Such a system should be well-suited to
verifying the inferences of $\na{E}$.

Our explorations are only preliminary, and more experimentation is
needed to support the claim that ordinary Euclidean inferences can be
checked efficiently. Moreover, performance can be sensitive to the
choice of language and the formulation of the axioms. For example, we
were surprised to find that performance was reduced when we replaced
our strict ``between'' predicate with a nonstrict one (presumably
because many additional facts, like $\mybetween(a,a,b)$, were
generated). Thus the data which we report is only suggestive. 

We emphasize that the point of these explorations is to show that it
should be possible to verify, automatically, proof texts which closely
approximate the proofs in the \emph{Elements}. From the point of fully
automated geometric reasoning, our methods are fairly simplistic.
There are currently at least four approaches to proving geometric
theorems automatically. The first is to translate the theorem to the
language of real closed fields and use decision procedures, based on
cylindrical algebraic decomposition \cite{collins:75}, for the latter;
but, in practice, this is too slow even for very simple geometric
theorems. A second method, known as Wu's method \cite{wu:94},
similarly translates geometric statements into algebraic problems and
uses computational algebraic techniques.  The method is stunningly
successful at verifying many difficult geometric theorems, but it
cannot handle the order relation between magnitudes, or the
``between'' predicate for points on a line; and so it is inadequate
for much of the \emph{Elements}. It is also limited to statements that
can be translated to universal formulas in the language of fields. A
third method, known as the area method \cite{chou:gao:zhang:94}, has
similar features. Finally, there are so-called ``synthetic methods,''
which use heuristic proof search from geometric axioms. Our methods
fall under this heading, but are not very advanced. One would expect
to do better with intelligent heuristics and more efficient
representations of diagrammatic information, along the lines described
by Chou, Gao, and Zhang \cite{chou:gao:zhang:94}. (See also
\cite{chou:gao:01} for an overview of the various methods.)

In other words, our work does not constitute a great advance in
automated geometric theorem proving, even for the kinds of theorems
one find in the \emph{Elements}. Our methods show how to verify the
smaller, diagrammatic inferences in Euclid's proofs, given the
higher-level structure, and, most importantly, the requisite
construction. It is an entirely different question as to how a system
might be able to \emph{find} such a construction automatically. We
have not addressed this question at all.

We do hope, however, that our analysis of the way that Euclidean
reasoning combines metric and diagrammatic components can provide some
useful insights towards modeling proof search in structured domains.
Rather than model geometry as a first-order axiomatic system, we have
taken advantage of specific features of the domain that reduce the
search space dramatically. Particularly notable is the way that we
understand Euclidean proofs as building up contexts of data (in our
case, ``diagrammatic information'' and ``metric information'') that
can be handled in domain-specific ways. In other words, adding objects
``to the diagram'' and inferring metric consequences means adding
information to a database of local knowledge that will be accessed and
used in very particular ways. We expect that such approaches will be
fruitful in modeling other types of mathematical reasoning as well.

\section{Conclusions}
\label{conclusions:section}

We conclude by summarizing what we take our analysis of Euclidean
proof to have accomplished, discussing questions and other work
related to our project, and indicating some of the questions
and broader issues that our work does not purport to address.

\subsection{Summary of results}

We claim to have a clean analysis of the argumentative structure of
the proofs in Books I to IV of the \emph{Elements}. We tried to make
this claim more precise in Section~\ref{characterizing:section} by
discussing the features of the \emph{Elements} that we have tried to
model. We have also gone out of our way, in
Section~\ref{comparison:section}, to indicate ways in which proofs in
our formal system differ from Euclid's. 

It is important to keep in mind that modeling the ``argumentative
structure'' of the \emph{Elements} is not just a matter of modeling
the Euclidean entailment relation in semantic or deductive terms, or
giving an account of geometric validity. Rather, our goal has been to
understand which \emph{individual inferences} are licensed by
Euclidean practice, so that a line-by-line comparison renders our
formal proofs close to Euclid's. To the extent in which we have
succeeded, this provides a sense in which the proofs in the
\emph{Elements} are more rigorous than is usually claimed. In
particular, we have identified precise rules that govern diagrammatic
inferences, which are sound relative to modern semantics; and we have
shown that, for the most part, Euclid's proofs obey these rules. As a
result, the proofs in the \emph{Elements} now seem to us to be much
\emph{closer} to formal proof texts than almost any other instance of
informal mathematics.

In Section~\ref{completeness:section}, we showed that our formal
system is sound and complete for an appropriate semantics of ruler and
compass constructions. Insofar as our formal system captures Euclidean
practice, this shows that the modern semantics provides an accurate
characterization of the provable Euclidean theorems.

In Section~\ref{implementation:section}, we described some initial but
promising attempts to verify the inferences of $\na{E}$ using current
automated reasoning technology. Our findings suggest that it should
not be difficult to develop a formal proof checker for $\na{E}$. This
provides further support to our claim that proofs in the
\emph{Elements} are much closer to formal proofs than is usually
acknowledged. The way proofs in $\na{E}$ organize data into metric and
diagrammatic components, each of which is individually more manageable
than their union, hints at a strategy that should have broader
application to formal verification.

Finally, we emphasize that we have provided a \emph{logical} analysis,
which screens off cognitive, historical, and broader philosophical
questions related to diagram use. This is not to deny the importance
of such questions. On the contrary, we feel that by fixing ideas and
clarifying basic notions, the logical analysis can support the study
of diagram use and Euclidean practice. Thus we take our analysis to show
how the norms of a mathematical practice can be analyzed on their own
terms, in a way that can support broader inquiry. We hope that we have
also demonstrated that such analysis can be rewarding, providing us
with a better understanding of the mathematics itself.

\subsection{Questions and related work}

Our work is situated in a long tradition of axiomatic studies of
geometry, from Hilbert to Tarski and through to the present day. Our
emphasis is novel, in that we have tried to characterize a particular
geometric practice and style of argumentation. In contrast, modern
axiomatic studies aim to provide a deeper understanding of geometry in
modern terms, focusing, for example, on the dependence and
independence of axioms and theorems, the results of dropping or
modifying various axioms, and the relationships to other axiomatic
systems. We cannot provide an adequate survey of these topics here,
but recommend textbooks by Coxeter~\cite{coxeter:69} and
Hartshorne~\cite{hartshorne:05}. (See also the article by Tarski and
Givant~\cite{tarski-givant}, which surveys the history of geometric
studies by Tarski and his students.)

Our project does raise some traditional logical questions, however.
For example, our diagrammatic axioms are all universal axioms, and
describe a subset of the universal consequences of Tarski's axioms for
Euclid's geometry. It would be nice to have a natural semantic
characterization of this set of universal sentences. We know that it
is a strict subset of the set of universal consequences of affine
geometry: Hilbert \cite[Chapter V]{hilbert:99} showed that Desargues'
theorem, which is a consequence of affine geometry, cannot be proved
in the plane without the axioms of congruence. Also, given that our
construction rules are not independent, it would be nice to
have a more principled way of generating the list, beyond simply
running through the \emph{Elements} and making a list of the ones that
Euclid seems to use. Finally, as we have mentioned, the question as to
the decidability of the $\forall \exists$ consequences of Tarski's
axioms, and hence the decidability of $\na{E}$, remain open.

Read as first-order axioms, all the basic rules of $\na{E}$ are given
by universal formulas, except for the construction rules, which have
$\forall \exists$ form. If we introduce Skolem functions for these
axioms, Herbrand's theorem implies that any theorem of $\na{E}$ can be
witnessed by an explicit construction involving these functions,
together with ``if \ldots then \ldots else'' statements involving
atomic conditions. This provides one sense in which Euclidean geometry
is ``constructive.'' However, conditional expressions are undesirable;
from a constructive perspective, for example, it may be impossible to
determine whether a point is actually on a line or only very close to
it. Jan von Plato \cite{VPla95} provides a strictly constructive
version of affine geometry (see also \cite{VPla98}).
Michael Beeson \cite{beeson:unp} characterizes the problem nicely
by observing that Euclid's constructions are not \emph{continuous} in
the input data, and offers a constructive version of Euclidean
geometry.

Our project also gives rise to computational questions. On the
theoretical side, there is, of course, the problem of
providing sharp upper and lower bounds on the complexity of
recognizing the various types of inference that, according to
$\na{E}$, Euclid sanctions as immediate. The challenge of 
obtaining \emph{practical} implementations should give rise to
interesting problems and solutions as well.

The implementation of a proof-checker for $\na{E}$ could be used to
help teach Euclidean geometry, and Euclidean methods of proof. There
are a number of graphical software packages in existence that support
geometric exploration and reasoning, of which the best known are
perhaps the \emph{Geometer's Sketchpad} \cite{geometers:sketchpad},
\emph{Cabri} \cite{cabri:geometry}, and \emph{Cinderella}
\cite{gebert:kortenkamp:99}. These systems do not, however, focus on
teaching geometric \emph{proof}. Others have explored the use of
graphical front ends to conventional proof assistants, supported by
specialized decision procedures for geometry. As we were completing
a draft of this paper, we came across Narboux~\cite{narboux:07}, which
not only provides a thorough survey of such work, but also describes
an impressive effort, \emph{Geoproof}, along these lines. Even though
\emph{Geoproof} is not based on an explicit analysis of Euclidean
proof, it is interesting to note that its primitives and construction
rules bear a striking similarity to ours.

\subsection{Broader issues}

In the end, what is perhaps least satisfying about our analysis is
that we do not go beyond the logical and computational issues: we
provide a detailed description of the norms governing Euclidean proof
without saying anything at all about how those norms arose, or why
they should be followed. We will therefore close with just a few words
about some of the cognitive, historical, and more broadly
philosophical issues that surround our work.

On the surface, it might seem that there is a straightforward
cognitive explanation as to why some of Euclid's diagrammatic
inferences are basic to geometric practice, namely, that these
inferences rely on spatial properties that are ``hardwired'' into our
basic perceptual faculties. In other words, thanks to evolution, we
have very good faculties for picking out edges and surfaces in our
environment and inferring spatial relationships; and these are the
kinds of abilities that are needed to support diagrammatic inference.
But one should be wary of overly simplistic explanations of this sort;
see the discussion in \cite{avigad:giaquinto:review}. In particular,
one should keep in mind that mature mathematical behavior is only
loosely related to more basic perceptual tasks. For instance, the
example discussed in Section~\ref{generality:section} shows that
Euclidean geometric reasoning requires keeping in mind that only some
features present in a diagram are essential to the mathematical
context it is supposed to illustrate. Informal experimentation on some
of our nonmathematical friends and family members shows that the
expected response to this exercise is by no means intuitively clear;
in other words, there seems to be a learned mathematical component to
the normative behavior. At the same time, we do not doubt that a
better understanding of our cognitive abilities can help explain why
certain geometric inferences are easier than others. It would
therefore be nice to have a better understanding of the cognitive
mechanisms that are involved in such reasoning.

We hope that our analysis can support a refined historical
understanding as well. Historians will cringe at our naive claim to
have analyzed ``the text of the \emph{Elements}''; there is a long and
complicated history behind the \emph{Elements}, and we have focused
our attention on only one translation (Heath's) of one version of the
text (Heiberg's). We do expect that, for the most part, our findings
are robust across the various editions. In fact, some features of the
historical record nicely support our claims. Saito \cite{saito:06} has
compared the diagrams in a number of editions of the \emph{Elements},
and has noted that earlier versions exhibit some striking differences
from the modern ones. For example, earlier diagrams are often
``overspecified'': a parallelogram mentioned in the statement of a
theorem may be depicted by a rectangle, or even a square. This sits
well with our claim that angle and metric information is never
inferred from the diagram; the fact that the metric information in the
diagrams is so blatantly misleading can be viewed as a subtle reminder
to the reader that it should not be relied upon.\footnote{We are
  grateful to Anthony Jones and Karine Chemla for this observation.}
On the other hand, if it turns out that there are ways in which our
analysis does not hold up well across historical developments, we
expect that our work can help clarify the nature of the historical
changes.

Moreover, we hope our analysis can help support a better historical
understanding of the evolution of geometric reasoning, and the
relationship between different geometric practices. There have been
rich historical analyses of the problems and methods found in the
ancient geometric tradition \cite{knorr:85,Netz99}, as well as, say,
the transition to the analytic tradition of Descartes \cite{bos:01}.
Ken Manders has remarked to us that diagrams are used in fundamentally
different ways in nineteenth century projective geometry texts; as the
diagrams get more complicated, more of the burden of keeping track of
the information they represent is shifted to the text. We expect that
the type of analysis we carry out here can complement the historical
study, and sharpen our understanding of the mathematical developments.

Finally, there is hope that the rules of Euclidean proof can be
``explained'' or ``justified'' not by cognitive or historical data,
but, rather, by broader epistemological considerations. For example,
Marco Panza \cite{panza:unp} takes Euclidean practice to inform a
metaphysical account of the nature of geometric objects; Marcus
Giaquinto \cite{giaquinto:07} takes cognitive data to support
epistemological conclusions regarding the role of visualization in
mathematics (but see the critique in \cite{avigad:giaquinto:review});
and Jamie Tappenden \cite{tappenden:05} explores ways of treating
visualization as an ``objective'' feature of mathematics, rather than
merely a cognitive device. It is possible that a suitably abstract
characterization of our cognitive abilities or the spatial situations
the practice tries to model can provide an informative sense in which
our fundamental inferences are the ``right'' ones for the task.

Kant famously took the fundamental principles of geometry to provide
synthetic knowledge, grounded by our \emph{a priori} intuition of
space:
\begin{quote}
  Take the proposition that with two straight lines no space at all
  can be enclosed, thus no figure is possible, and try to derive it
  from the concept of straight lines and the number two; or take the
  proposition that a figure is possible with three straight lines, and
  in the same way try to derive it from these concepts. All of your
  effort is in vain, and you see yourself forced to take refuge in
  intuition, as indeed geometry always does. You thus give yourself an
  object in intuition; but what kind is this, is it a pure \emph{a
    priori} intuition or an empirical one? If it were the latter, then
  no universally valid, let alone apodictic proposition could ever
  come from it: for experience can never provide anything of this
  sort. You must therefore give your object \emph{a priori} in
  intuition, and ground your synthetic proposition on this.
  \cite[A47--A48/B64--B65]{Kant}.
\end{quote}
Indeed, his discussion of Euclid's Proposition I.32 in the
Transcendental Doctrine of Method \cite[A712--A725/B740--753]{Kant}
provides an illuminating account of how he takes such synthetic
reasoning to work. Kant's views on geometry have been studied in
depth; see, for example,
\cite{friedman:85,shabel:03a,shabel:03,shabel:04}. Lisa Shabel writes:
\begin{quote} [The] Kantian account of informal but contentful axioms
  of Euclidean geometry stemming directly from an \emph{a priori}
  representation of space is itself consistent with Euclidean
  practice: neither Euclid's elements nor its eighteenth-century
  analogs offer formal axioms but rather definitions and postulates
  which, if taken seriously, provide a \emph{mereotopological}
  description of the relations among the parts of the euclidean plane.
  The content of these relations is \ldots\ precisely what Kant alleges
  is accessible to us in pure intuition, prior to geometric
  demonstration. \cite[p.~213]{shabel:04}
\end{quote}
This provides us with a convenient way of framing our project: we have
provided a logical description of the mereotopological relations that
are implicit in Euclid's definitions and postulates, without feigning
hypotheses as to their origins. As Shabel's remarks suggest (see also
\cite[footnote 4]{shabel:04} and \cite{shabel:03a}), it would be
interesting if one could describe a more fundamental account of
spatial intuition that can serve to justify or explain the rules of
our system. Stewart Shapiro has suggested to us that it would also be
interesting to explain what distinguishes Euclid's axioms and
postulates from everything he does not say, that is, the assumptions
and rules of inference that we take to be implicit in the
\emph{Elements}.

In Section~\ref{introduction:section}, we noted that philosophers have
historically been concerned with the problem of how the particular
diagrams in the \emph{Elements} can warrant general conclusions. In
particular, a central goal of Kant's account
\cite[A712--A725/B740--753]{Kant} is to explain how singular objects
given in intuition can provide general knowledge. Jeremy Heis has
pointed out to us that a curious feature of our account of Euclidean
geometry is that the role of the singular --- that is, the particular
diagram --- drops out of the story entirely; we focus only on the
diagrammatic features that are generally valid in a given context, and
say nothing about a particular instantiation.

There is a fairly mundane, if partial, explanation of the role that
concrete diagrams play in geometric practice. Although not every
feature found in a particular diagram will be generally valid, the
converse is more or less true: any generally valid consequence of the
diagrammatic hypotheses will be present in a sufficiently well-drawn
diagram. A particular diagram can therefore serve as a heuristic
guide, suggesting candidates for diagrammatic consequences that are,
perhaps, confirmed by other forms of reasoning. Mumma's original
system, $\na{Eu}$, is more faithful to this understanding of diagram
use; for example, the prover can label a point of intersection in a
particular diagram associated with a proof, independent of the
mechanisms that are invoked to justify the fact that the intersection
is present in general. Some systems of automated reasoning rely on
crude procedures to search for possible proof candidates, and then
employ other methods to check and fill in the details (see, for
example, \cite{meng:paulson:08,veroff:01}). It therefore seems to us
worth noting that diagram use in mathematics raises two separate
issues: first, how (or whether) alternative, nonpropositional
representations of mathematical data can be used to facilitate or
justify inferences; and, second, how overspecific or imperfect
representations can be used to support the reasoning process.  Leitgeb
\cite{leitgeb:unp} begins to address the first issue.

As the vast literature on the \emph{Elements} indicates, Euclidean
geometry has been a lively source of questions for scholars of all
persuasions for more than two millennia. We only hope that the
understanding of Euclidean proof we present here will prove useful in
furthering such inquiry.


\end{document}